\documentclass[a4paper,11pt]{article}

\usepackage{ucs}
\usepackage[utf8]{inputenc}
\usepackage{bbm}

\usepackage{graphicx}
\usepackage{amsfonts}
\usepackage{dsfont}
\usepackage{amssymb}
\usepackage{amsmath}
\usepackage{amsthm}
\usepackage{enumerate}
\usepackage{stmaryrd}
\usepackage{fullpage}
\usepackage{ifthen}
\usepackage{subfigure}
\usepackage{epic}
\usepackage{authblk}
\usepackage{textcomp}
\usepackage{mathrsfs}
\usepackage{bm}
\usepackage[small]{caption}
\usepackage{tikz,tkz-tab}
\usetikzlibrary{matrix}
\usepackage{pgfplots}
\pgfplotsset{compat=newest} 
\pgfplotsset{plot coordinates/math parser=false}

\usepackage[hypertexnames=false,colorlinks=true,linkcolor=blue,citecolor=blue]{hyperref}
\usepackage[numbers,comma,square,sort&compress]{natbib}
\usepackage[letterpaper,text={6.5in,9in},centering]{geometry}

\usepackage{color}
\usepackage{titlesec}

\graphicspath{{eps/}{pdf/}{png/}}

\newcommand{\bqq}{\begin{equation}}
\newcommand{\eqq}{\end{equation}}
\newcommand{\bqs}{\begin{equation*}}
\newcommand{\eqs}{\end{equation*}}

\newcommand{\C}{\mathbb{C}}

\newcommand{\bbE}{\mathbb{E}}

\newcommand{\R}{\mathbb{R}} 
\newcommand{\N}{\mathbb{N}}

\newcommand{\Z}{\mathbb{Z}}

\renewcommand{\H}{\mathscr{H}}

\newcommand{\br}{\mathbf{r}}

\newcommand{\bG}{\mathbf{G}}

\newcommand{\bW}{\mathbf{W}}

\newcommand{\G}{\mathscr{G}}

\newcommand{\md}{\mathrm{d}}

\newcommand{\mbi}{\mathbf{i}}

\renewcommand{\Re}{\mathrm{Re}}
\renewcommand{\Im}{\mathrm{Im}}

\newtheorem{lem}{Lemma}[section]
\newtheorem{thm}{Theorem}
\newtheorem{prop}[lem]{Proposition}

 {\begin{trivlist}\item[]\textbf{Proof#1 }}%
 {\hspace*{\fill}$\rule{0.3\baselineskip}{0.35\baselineskip}$\end{trivlist}}

  {\begin{trivlist}\item[]{\bf Hypothesis #1 }\em}{\end{trivlist}}

\numberwithin{equation}{section}

\title{The logarithmic Bramson correction for Fisher-KPP equations on the lattice $\Z$}

\author[1]{Christophe Besse, Gr\'egory Faye\footnote{Corresponding author: \texttt{gregory.faye@math.univ-toulouse.fr}}, Jean-Michel Roquejoffre \& Mingmin Zhang}
\affil[1]{\small CNRS, UMR 5219, Institut de Math\'ematiques de Toulouse, 31062 Toulouse Cedex, France}

\begin{document}
\maketitle

\begin{abstract}
We establish in this paper the logarithmic Bramson correction for Fisher-KPP equations on the lattice $\Z$. The level sets of solutions with step-like initial conditions are located at position $c_*t-\frac{3}{2\lambda_*}\ln t+O(1)$ as $t\rightarrow+\infty$  for some explicit positive constants $c_*$ and $\lambda_*$. This extends a well-known result of Bramson in the continuous setting to the discrete case using only PDE arguments. A by-product of our analysis also gives that the solutions approach the family of logarithmically shifted traveling front solutions with minimal wave speed $c_*$ uniformly on the positive integers, and that  the solutions converge along their level sets to the minimal traveling front for large times.
\end{abstract}

\noindent {\small {\bf Keywords:} discrete Fisher-KPP equations; Bramson logarithmic corrections; Green's function.}
\bigskip

\section{Introduction}

We consider the following Cauchy problem
\bqq
\left\{
\begin{split}
\frac{\md}{\md t} u_j(t) &= u_{j-1}(t)-2u_j(t)+u_{j+1}(t)+f(u_j(t)), \quad t>0, \quad j\in\Z,\\
u_j(0) & =u_j^0, \quad j\in\Z,
\end{split}
\right.
\label{KPP}
\eqq
for some nontrivial bounded initial sequence $(u_j^0)_{j\in\Z}\in\ell^\infty(\Z)$. Throughout, we let $\ell^\infty(\Z)$ denote the Banach space of bounded valued sequences indexed by $\Z$ and equipped with the norm:
\bqs
\|u\|_{\ell^\infty(\Z)}:=\underset{j\in\Z}{\sup}|u_j|\,.
\eqs
Fur future reference, we also let $\ell^q(\Z)$ with $1\leq q<+\infty$ denote the Banach space of sequences indexed by $\Z$ such that the $\ell^q(\Z)$ norm, defined for $u:\Z\to\C$ by
\bqs
\| u\|_{\ell^q(\Z)}:=\left(\sum_{j\in\Z}|u_j|^q\right)^{1/q}\,,
\eqs
is finite.  Here the reaction term $f\in\mathscr{C}^2([0,1])$ is assumed to be of Fisher-KPP type, that is
\bqs
f(0)=f(1)=0, \quad f'(0)>0, \quad f'(1)<0, \text{ and } 0<f(u)\leq f'(0)u \text{ for all } u\in(0,1).
\eqs
Without loss of generality, we extend $f$ linearly on $(-\infty,0)\cup(1,+\infty)$. Throughout, we will further assume that the nontrivial initial sequence $(u_j^0)_{j\in\Z}\in\ell^\infty(\Z)$ satisfies
\bqs
0 \leq u_j^0 \leq 1, \quad \forall j\in\Z, \quad \text{ and } \quad 
u_j^0=0, \quad j\geq J,
\eqs
for some integer $J\in\Z$. The solution $u_j(t)$ to \eqref{KPP}  is classical for $t>0$ in the sense that $(u_j)_{j\in\Z}\in \mathscr{C}^1((0,\infty),\ell^\infty(\Z)) \cap \mathscr{C}^0([0,\infty),\ell^\infty(\Z))$ and verifies $0<u_j(t)<1$ for each $t>0$ and $j\in\Z$ thanks to the strong maximum principle which applies in the discrete setting. 

The Cauchy problem \eqref{KPP} can be interpreted as the discrete version of the standard spatially-extended Fisher-KPP equation
\bqq
\left\{
\begin{split}
&\partial_t u = \partial_x^2 u  +f(u),  \quad t>0, \quad x\in\R, \\
&u(0,x) = u^0(x), ~~~\quad x\in\R,
\end{split}
\right.
\label{KPPcont}
\eqq
for some step-like initial datum $0\leq u^0\leq 1$ with $u^0\not\equiv 0$ and $u^0(x)=0$ for $x\geq A$ with $A\in\R$. Such discrete and continuous equations arise in many mathematical models 
in biology, ecology, epidemiology or genetics, see e.g. \cite{fisher,KPP37,W82,BF21}, and $u$ typically stands for the density of a population. Much is known regarding the long time asymptotics of the solutions to \eqref{KPPcont} and a famous result due to Aronson-Weinberger \cite{AW78} states that the solutions have asymptotic spreading speed $c_*=2\sqrt{f'(0)}$. That is
\bqs
\underset{t\rightarrow+\infty}{\lim}~\underset{|x| \leq ct}{\min}~ u(t,x) = 1, \quad \text{ for all } c\in(0,c_*),
\eqs
and 
\bqs
\underset{t\rightarrow+\infty}{\lim}~\underset{x \geq ct}{\max}~  u(t,x) = 0, \quad \text{ for all } c>c_*.
\eqs
In the celebrated papers \cite{Bram1,Bram2}, Bramson obtained sharp asymptotics of the location of the level sets of $u(t,x)$ through probabilistic arguments using the relationship between the classical Fisher-KPP equation \eqref{KPPcont} and branching Brownian motion. More precisely, for any $m\in(0,1)$, if $x_m(t):=\sup \left\{ x \in\R ~|~ u(t,x) \geq m\right\}$ denotes the leading edge of the solution $u$ at level $m$, then Bramson proved that
\bqq
x_m(t) = c_*t-\frac{3}{2\lambda_*} \ln t + \sigma_m +o(1), \quad \text{ as } t\rightarrow+\infty,
\label{BramsonShift}
\eqq
for some shift $\sigma_m$ depending on $m$ and the initial datum with $\lambda_*:= c_*/2$. There is a long history of works \cite{KPP37,Lau,Uchiyama} regarding the large time behavior for the solutions of the Cauchy problem \eqref{KPPcont} dating back to the pioneer work of Kolmogorov, Petrovskii and Piskunov which relate spreading properties of the solutions to the convergence towards a selected traveling front solution, often referred to as the critical pulled front. Indeed for \eqref{KPPcont}, it is well-known \cite{AW78,KPP37} that there exist traveling front solutions of the form $\varphi_c(x-ct)$ where $0<\varphi_c<1$ with limits $\varphi_c(-\infty)=1$ and $\varphi_c(+\infty)=0$ if and only if $c\geq c_*=2\sqrt{f'(0)}$. Furthermore the profiles $\varphi_c$ are decreasing, unique up to translations and verify
\bqq
\varphi_c''+c\varphi_c'+f(\varphi_c)=0.
\label{TFkppCont}
\eqq

On the other hand, the logarithmic shift in KPP type equations has been much revisited in the recent years. Most notably, we mention the recent developments \cite{BD15,BBD17,HNRR13,NRR17,NRR19} which recover and extend Bramson's results using PDE arguments, or the extensions of his work using the connection with branching Brownian motion and branching random walks \cite{ABR,Aidekon,roberts}. As postulated in \cite{EvS00}, the logarithmic shift has a universal character and has been retrieved in several other contexts: Fisher-KPP equations with nonlocal diffusion \cite{Graham,Roque22} and nonlocal interactions \cite{BHR20}, in a periodic medium \cite{HNRR16} or in more general monostable reaction-diffusion equations \cite{Giletti}. It is also intimately linked to the selection of pulled fronts, and we refer to the recent work \cite{AS20} and references therein. In this work, we show that the Bramson logarithmic correction also holds in the discrete setting of \eqref{KPP} rigorously justifying some formal asymptotic results from \cite{EvSP}. Apart from its own mathematical interest, our motivation for investigating the Bramson logarithmic shift for \eqref{KPP} also stems from our recent works \cite{BF20,BF21} on SIR epidemic models set on graphs where the discrete Fisher-KPP equation naturally arises for a specific choice of the nonlinearity $f$.  In this modeling context, a precise description of the long time dynamics of the solutions is of importance to sharply describe the spatial spread of an epidemic outbreak.

\paragraph{Main results.} From \cite{W82,HH19}, we know that the solutions $u_j(t)$ to \eqref{KPP} satisfy the following spreading property:
\bqq
\underset{t\rightarrow+\infty}{\lim}~\underset{j \leq ct}{\min}~ u_j(t) = 1, \quad \text{ for all } c\in(0,c_*),
\label{spreading1}
\eqq
and
\bqq
\underset{t\rightarrow+\infty}{\lim}~\underset{j \geq ct}{\max}~  u_j(t) = 0,  \quad \text{ for all } c>c_*.
\label{spreading0}
\eqq
Here, the spreading speed $c_*>0$ is uniquely defined as
\bqq
c_*:= \underset{\lambda>0}{\min}~  \frac{e^{\lambda}-2+e^{-\lambda}+f'(0)}{\lambda}.
\label{wavespeed}
\eqq
Let us note that there is a unique $\lambda_*>0$ where the above minimum is attained so that the couple $(c_*,\lambda_*)$ is solution of the nonlinear problem
\bqq
\left\{
\begin{split}
c_*\lambda_* &= e^{\lambda_*}-2+e^{-\lambda_*}+f'(0),\\
c_* & =e^{\lambda_*}-e^{-\lambda_*}.
\end{split}
\right.
\label{systemclamb}
\eqq
 Indeed, as in the continuous case, the spreading speed $c_*$ is also the threshold to the existence of traveling front solutions to \eqref{KPP}. More precisely, combining the results from \cite{CG02,CG2003,CC04,CFG06,ZHH91} for each $c\geq c_*$, there exists a unique (up to translation) monotone front $U_c \in \mathscr{C}^1(\R)$ solution of 
	\bqq
	\left\{
	\begin{split}
		0 &= c U_c'(x)+U_c(x+1)-2U_c(x)+U_c(x-1)+f(U_c(x)),\quad x\in\R,\\
		&U_c(-\infty)  =1, \quad U_c(+\infty)=0, \quad 0< U_c<1.
	\end{split}
	\right.
	\label{TFkpp}
	\eqq
We normalize the minimal traveling front profile $U_{c_*}$, according to its asymptotic behavior at $+\infty$. More precisely, using the result of \cite{CC04}, we normalize $U_{c_*}$ so as to satisfy asymptotically 
	\bqq
	\frac{U_{c_*}(x)}{x e^{-\lambda_* x}}\underset{x\to+\infty}{\longrightarrow} 1.
	\label{normalization}
	\eqq

As previously explained, our aim is to provide sharp asymptotics of the level sets of $u_j(t)$. We define for any $m\in(0,1)$  and for each $t\in(0,+\infty)$ the quantity 
\bqs
j_m(t):=\sup \left\{ j \in\Z ~|~ u_j(t)\geq m \right\}.
\eqs
Our main result is the following.

\begin{thm}\label{thmlog} Let $c_*>0$ and $\lambda_*>0$ be defined in \eqref{wavespeed} and \eqref{systemclamb}. For each $m\in(0,1)$ there exist $C>1$ and $T>1$, such that
\bqs
j_m(t) \in \Z \cap \left[c_*t-\frac{3}{2\lambda_*}\ln t-C,c_*t-\frac{3}{2\lambda_*}\ln t+C\right], \quad t\geq T.
\eqs
\end{thm}

The above Theorem~\ref{thmlog} is closely related to some results obtained on discrete-time branching random walks \cite{ABR,Aidekon} where the Bramson logarithmic correction is known to greater precision, that is up to $o(1)$ error as in \eqref{BramsonShift}.  As explained in \cite[Appendix A]{Graham}, the connection between continuous in time branching random walks and the Fisher-KPP equation  with nonlocal diffusion can be established when the reaction $f$ takes a special form and the kernel defining the nonlocal diffusion is a Borel probability measure. Let us finally insist on the fact that the Bramson logarithmic shift proved in \cite[Theorem 1.1]{Graham} for nonlocal Fisher-KPP equations includes the  diffusion measure $\delta_1-2\delta_0+\delta_{-1}$ studied here. However, our main Theorem~\ref{thmlog} allows to handle slightly more general initial conditions than the Heaviside step data considered in \cite{Graham}. Most notably, the core of our proof greatly differs from \cite{Graham} where key  estimates for the long time behavior of the linear Dirichlet problem are proved using probabilistic arguments via a Feynman-Kac representation. 

Combining the above Theorem~\ref{thmlog} with the spreading property \eqref{spreading1}, we get that
\bqq
\underset{t\rightarrow+\infty}{\liminf}\left(~\underset{0\leq j \leq c_*t-\frac{3}{2\lambda_*}\ln t-C}{\min}~ u_j(t) \right) \longrightarrow 1 \text{ as } C\rightarrow+\infty,
\label{spreading1refined}
\eqq
since $u_j(t)\rightarrow 1$ as $t\rightarrow+\infty$ locally uniformly in $j$. Furthermore, using that $\underset{j\rightarrow+\infty}{\lim}~ u_j(t)= 0$ for each $t\geq0$, we also deduce that
\bqq
\underset{t\rightarrow+\infty}{\limsup}\left(~\underset{ j \geq c_*t-\frac{3}{2\lambda_*}\ln t+C}{\max}~ u_j(t) \right) \longrightarrow 0 \text{ as } C\rightarrow+\infty.
\label{spreading2refined}
\eqq
The limits \eqref{spreading1refined} and \eqref{spreading2refined} show that the region in the lattice where the solution $u_j(t)$ is bounded away from $0$ and $1$ is located around the position $c_*t-\frac{3}{2\lambda_*}\ln t$ and has a bounded width in the limit $t\rightarrow+\infty$.

Refining the arguments of Theorem~\ref{thmlog}, one can actually prove that the solution $u_j(t)$ approaches the family of shifted traveling fronts $U_{c_*}\left(j-c_*t+\frac{3}{2\lambda_*}\ln t+\zeta\right)$ uniformly in $j\geq0$. Our second main result reads as follows.

\begin{thm}\label{thmconv} Let $c_*>0$ and $\lambda_*>0$ be defined in \eqref{wavespeed} and \eqref{systemclamb}. There exist a constant $C>0$ and a function $\zeta:(0,+\infty)\rightarrow\R$ with $|\zeta(t)|\leq C$ such that
\bqq
\underset{t\rightarrow+\infty}{\lim}~  \underset{j \geq 0}{\sup} \left| u_j(t)- U_{c_*}\left(j-c_*t+\frac{3}{2\lambda_*}\ln t+\zeta(t)\right) \right|  =0.
\label{convfront}
\eqq
Furthermore, for every $m\in(0,1)$ and every sequence $(t_n,j_n)_{n\in\N}$ such that $t_n\to+\infty$ as $n\to+\infty$ and $j_n= j_m(t_n)$ for all $n\in\N$, there holds
	$$u_{j+j_n}(t+t_n)\to U_{c_*}(j-c_*t+U^{-1}_{c_*}(m))~~~\text{locally uniformly in}~(t,j)\in\R\times\Z,$$
	where $U^{-1}_{c_*}$ denotes the inverse of the function $U_{c_*}$.
\end{thm}

\paragraph{Strategy of proof.} The strategy of our proof is inspired by the PDE arguments that were developed recently in the continuous setting \cite{HNRR13}.  The take home message from \cite{HNRR13} is that the long time dynamics of \eqref{KPPcont} can be read out from the solutions of the linearized problem around the unstable state $u=0$ with a Dirichlet boundary condition at $x=c_*t$. Indeed, these solutions can be compared to the critical front $\varphi_{c_*}$, appropriately shifted to the position $c_*t-\frac{3}{2\lambda_*}\ln t$, in the diffusive regime of the linearized equation, that is, for all $x\in[c_*t,c_*t+\vartheta \sqrt{t}]$ and for any $\vartheta>0$. One of the main difficulty in our analysis comes from the discrete nature of our equation such that the above argument has to be largely adapted. Our starting point is also the linearized equation around the unstable state $0$. We first perform,  in the linearized equation of $r_j(t)$, the change of variable $r_j(t)=e^{-\lambda_*(j-c_*t)}w_j(t)$ for some new sequence $(w_j(t))_{j\in\Z}$ which solves 
\bqs
\frac{\md}{\md t} w_j(t) =e^{\lambda_*} \left(w_{j-1}(t)-2w_j(t)+w_{j+1}(t)\right)-c_*(w_{j+1}(t)-w_j(t)), \quad t>0, \quad j\in\Z.
\eqs
Solutions of the above equation starting from some initial sequence $(w_j^0)_{j\in\Z}\in\ell^\infty(\Z)$ are given through the representation formula
\bqs 
w_j(t)= \sum_{\ell \in \Z} \mathscr{G}_{j-\ell}(t) w_\ell^0, \quad t>0, \quad j\in \Z,
\eqs
where $(\mathscr{G}_j(t))_{j\in\Z}$ stands for the temporal Green's function (see \eqref{temporalGreen} below for a precise definition). The key step of the analysis is to obtain sharp pointwise estimates on the temporal Green's function. We establish that $\mathscr{G}_j(t)$ behaves like a Gaussian profile centered at $j=c_*t$ for all $|j-c_*t| \leq \vartheta t^\alpha$ for any $\vartheta>0$, $\alpha\in(0,1)$ and $t>1$. Actually, we prove a sharper result, which is of independent interest, by showing that the temporal Green's function can be decomposed as a universal Gaussian profile plus some reminder term which can be bounded and which also satisfies a generalized Gaussian estimate. We refer to Proposition~\ref{proprefined} for a precise statement. Let us note that similar generalized  Gaussian bounds have been recently derived for discrete convolution powers \cite{CF20,DSC14,Coeuret22} and are reminiscent of so-called local limit theorems in probability theory \cite{Petrov75}.  Owing to this sharp estimate on the temporal Green's function, we manage in a second step to construct appropriate sub and super solutions which allow us to precisely locate any level sets of the solutions to \eqref{KPP}. For this part, we take benefit from the new results obtained by one of the authors~\cite{Roque22} in the continuous setting with nonlocal diffusion. Let us finally emphasize that our analysis of the temporal Green's function does not rely on Fourier analysis but rather on a {\em spatial } point of view through the Laplace inversion formula by defining each $\mathscr{G}_j(t)$ as
\bqs
\forall \, t>0, \quad \forall j\in\Z, \, \quad \mathscr{G}_j(t) = \frac{1}{2\pi \mbi}\int_\Gamma e^{\nu t} \bG_j(\nu)\md \nu,
\eqs
where $\Gamma\subset\C$ is some well-chosen contour in the complex plane and $(\bG_j(\nu))_{j\in\Z}\in\ell^2(\Z)$ is the associated spatial Green's function (see \eqref{spatialGreen} below for a precise definition). 

Compared to Bramson's result in the continuous setting, our main Theorem~\ref{thmlog} only captures the logarithmic correction up to some $O(1)$ terms as $t\rightarrow+\infty$. We expect that our logarithmic expansion could be refined along the lines of \cite{NRR17,NRR19} with convergence to a single traveling front solution. We leave it for a future work.

\paragraph{Outline.} The rest of the paper is mainly dedicated to the proof of Theorem~\ref{thmlog}. For expository reasons,
instead of focusing directly on it, we first revisit in Section~\ref{seccontinuous} the continuous case and provide an alternate proof of the Bramson's logarithmic correction up to some $O(1)$ terms as $t\rightarrow+\infty$. Compared to \cite{HNRR13}, the novelty of this alternate proof, which takes its inspiration from the nonlocal continuous case \cite{Roque22}, is to solely focus on the linearized problem around the unstable state $u=0$ with a Dirichlet boundary condition at $x=c_*t$ without relying on self-similar variables. Indeed, the use of self-similar variables for the discrete Fisher-KPP equation is prohibited. Then, in Section~\ref{seclinear}, we study the linearized problem for the discrete Fisher-KPP equation and, in a first step, we prove a generalized Gaussian bound for the associated temporal Green's function. In a second step, we provide in Section~\ref{secRefined} a sharp asymptotic expansion for the temporal Green's function in the sub-linear regime which will be crucial to the proof of Theorem~\ref{thmlog}. Finally, in Section~\ref{secproofthmlog} we provide the lower and upper bounds in our main Theorem~\ref{thmlog} following the strategy presented in the continuous case. In the last Section~\ref{secconv}, we study the convergence to the logarithmically shifted minimal front and prove Theorem~\ref{thmconv}.

\paragraph{Notations.} Throughout the manuscript, we will use the notation $f\lesssim g$ whenever $f \leq C g$ for some universal constant $C > 0$ independent of $t$ and $x$ or $j$. Furthermore, for two functions $f(t)$ and $g(t)$, we use the notation $f(t)\ll g(t)$ whenever $\frac{f(t)}{g(t)}\rightarrow0$ as $t\rightarrow+\infty$, while we use $f(t) \sim g(t)$ whenever $\frac{f(t)}{g(t)}\rightarrow C$ as $t\rightarrow+\infty$ for some universal constant $C > 0$.

\section{An alternative proof of the logarithmic Bramson correction in the continuous case}\label{seccontinuous}

In this section, we revisit \cite{HNRR13} and propose an alternative proof of the logarithmic Bramson correction in the continuous case. More precisely, the purpose of this section is to prove the following result.

\begin{prop}\label{propCont}
Let $u$ be the solution of the Cauchy problem~\eqref{KPPcont} starting from some step-like initial datum $0\leq u^0\leq 1$ with $u^0\not\equiv 0$ and $u^0(x)=0$ for $x\geq A$ with $A\in\R$. Then, for any $\varepsilon>0$ very small, there exist $T_0>1$ large enough, $-\infty<\hat b<\hat a<+\infty$ and $\eta\in(0,1/2)$ such that
\begin{equation}
\label{1-two TWs}
(1-\varepsilon)\varphi_{c_*}\left(x-c_*t+\frac{3}{2\lambda_*}\ln t+\hat a\right)\le u(t,x)\le (1+\varepsilon) \varphi_{c_*}\left(x-c_*t+\frac{3}{2\lambda_*}\ln t+\hat b\right),
\end{equation}
uniformly in $1\le x-c_*t +\frac{3}{2\lambda_*}\ln t\le t^\eta$ for all $t\geq T_0$, where $c_*=2\sqrt{f'(0)}$, $\lambda_*=c_*/2$ and the minimal traveling front $\varphi_{c_*}$ solution of \eqref{TFkppCont} is normalized such that 
\begin{equation*}
\frac{\varphi_{c_*}(x)}{xe^{-\lambda_*x}}\underset{x\to+\infty}{\longrightarrow} 1\,.
\end{equation*}
\end{prop}

A direct consequence of the above proposition is that for each $m\in(0,1)$, there exists $C> 0$  such that
\bqs
x_m(t)\in \left[c_*t-\frac{3}{2\lambda_*}\ln t-C,c_*t-\frac{3}{2\lambda_*}\ln t+C\right], \text{ as } t\rightarrow+\infty,
\eqs
where $x_m(t)$ denotes the leading edge of the solution $u$ at level $m$. Moreover, it is worth to note that, based on \eqref{1-two TWs},  the argument for the large time convergence of the solutions to the family of shifted traveling fronts as well as of the solutions along their level sets to the profile of the minimal traveling front in Theorem 1.2 of \cite{HNRR13} can be simplified by applying this time directly the Liouville type result Theorem 3.5 of \cite{BH2007} instead of using Lemma 4.1 in \cite{HNRR13}.

\subsection{Preliminaries}\label{1-sec-2}

In a first step, we establish upper and lower barriers for a variant $v$ (see \eqref{1-v} below) of the solution $u$, by using the solution for the linear equation \eqref{1-linear-eqn-w} for $t$ sufficiently large and $x\in\R$ ahead of the position $x-c_*t \approx 0$ albeit with a small shift. These upper and lower bounds will play a crucial role in  showing a refined estimate on the expansion of the level sets of $u$ in the sequel.

Set
\begin{equation}\label{1-v}
	v(t,x)=e^{\lambda_*(x-c_*t)}u(t,x).
\end{equation} $$$$ 
This leads to
\bqq
\label{1-v-equation}
\left\{
\begin{split}
	&v_t -v_{xx}+c_*v_x+R(t,x;v)=0, \quad t>0,~x\in\R,\\
	&v(0,x)  =v^0(x)=e^{\lambda_*x}u^0(x),~~~~ \quad x\in\R,
\end{split}
\right.
\eqq
in which the nonlinear term $R(t,x;s)$ has the following  precise form
\begin{equation}
	\label{1-R term}
	R(t,x;s):=f'(0)s-e^{\lambda_*(x-c_*t)}f\left(e^{-\lambda_*(x-c_*t)}s\right)\ge 0
\end{equation}
 for $s\in\R$ and for $(t,x)\in(0,+\infty)\times\R$, due to the assumption on $f$ that $0<f(s)\le f'(0)s$ for $s\in(0,1)$ and due to the linear extension of $f$ in $(-\infty,0)\cup(1,+\infty)$. One then has that $R(t,x;v)\ge 0$ for $(t,x)\in (0,+\infty)\times\R$. Note that the change of function \eqref{1-v} is motivated by our forthcoming study of the discrete case where it is not possible to write the equation in a moving frame due to the discrete nature of the problem.
 
 Let us now consider the linear equation
\bqq
(\partial_t-\mathcal{L})w:=w_t-w_{xx}+c_*w_x=0, \quad t>0,~x\in\R.
\label{1-linear-eqn-w}
\eqq
It is easy to see that the function $p(t,y)=w(t,y+c_*t)$ satisfies $p_t-p_{yy}=0$ for $t>0$, $y\in\R$. We impose an odd and compactly supported initial datum $p_0\not\equiv 0$ in $\R$ with $p_0\geq0$ on $[0,\infty)$,  then it follows that
\begin{equation}
	\label{1-formula of w}
w(t,y+c_*t)=p(t,y)=\frac{1}{\sqrt{4\pi t}}\int_0^{+\infty}\left(e^{-\frac{(y-z)^2}{4t}}-e^{-\frac{(y+z)^2}{4t}}\right)p_0(z)dz, \quad t>0,~y\in\R.
\end{equation}
In particular, one has $p(t,y)> 0$ for $t>0$ and $y>0$, and the following asymptotic behavior holds true:
\begin{equation*}
p(t,y)\sim ye^{-\frac{y^2}{4t}}t^{-\frac{3}{2}},~~~\text{as}~t\to+\infty,
\end{equation*}
for $-\sqrt{t}\le y\le \sqrt{t}$. This then implies that $w(t,x)$  changes its sign exactly at $x-c_*t=0$, i.e., $w(t,x)>0$ for $t>0$ and $x\in\R$ with $x-c_*t>0$, whereas $w(t,x)<0$ for $t>0$ and $x\in\R$ with $x-c_*t<0$. Moreover, 
\begin{equation}
\label{1-asymp-w}
w(t,x)\sim (x-c_*t)e^{-\frac{(x-c_*t)^2}{4t}}t^{-\frac{3}{2}},~~~\text{as}~t\to+\infty,
\end{equation}
for $-\sqrt{t}\le x-c_*t\le \sqrt{t}$. 

We are now in position to take advantage of $w(t,x)$ in the constructions of the upper and lower barriers for $v$.

\subsection{Upper barrier for $v$}
We start with the construction of a supersolution to \eqref{1-v-equation} for all $t$ large enough and $x\in\R$ ahead of $x-c_*t\approx 0$ by following the strategy recently proposed in \cite{Roque22} for the continuous setting with nonlocal diffusion, which is itself reminiscent of the strategy used in \cite{HNRR13}. Ideally, the solution $v$ of \eqref{1-v-equation} would be controlled from above by the function $w$ solution of the linear equation \eqref{1-linear-eqn-w}  starting from an odd, nontrivial and compactly supported initial condition, since it is an actual supersolution by construction. However,  as the initial condition is chosen to be odd, the function  $w(t,x)$ is negative for $t>0$ and $x-c_*t<0$ which prevents us from readily comparing the two functions. The key observation from \cite{Roque22}, very much in the spirit of Fife \& McLeod \cite{FMcL}, is that we can add a cosine perturbation to $w$ which will eventually enable us to compare $v$ with this new supersolution slightly to the left of $x-c_*t=0$.

Consider now $\delta\in(0,1/3)$, that will be as small as needed. We now look for a barrier of $v(t,x)$ from above  for $t$ large enough and $x\in\R$ ahead of $x-c_*t=-t^\delta$. To do so, we construct a supersolution for \eqref{1-v-equation} in the form
\begin{equation}
	\label{1-overline v}
	\overline v(t,x)=\overline\xi(t) w(t,x)+\frac{1}{(1+t)^{\frac{3}{2}-\beta}}\cos\left(\frac{x-c_*t}{(1+t)^\alpha}\right)\mathbbm{1}_{\left\{x\in\R~|~ -t^\delta\le x-c_*t\le  \frac{3\pi}{2}(1+t)^\alpha\right\}},
\end{equation}
for $t$ large enough and $x\in\R$ with $x-c_*t\ge -t^\delta$, where the unknown $\overline \xi(t)\in \mathscr{C}^1$ is assumed to be positive and bounded in $(0,+\infty)$, and $\overline \xi'(t)\ge 0$  in $(0,+\infty)$,  $\alpha\in(1/3,1/2)$ and $\beta>0$, all of which will be chosen in the course of investigation. Note that $\delta<\alpha$.  Let us now check  that $\overline v(t,x)$ is a supersolution of \eqref{1-v-equation} for $t$ large enough and $x\in\R$ with $x-c_*t\ge -t^\delta$.

First of all, we look at the region $\{x\in\R~|~x-c_*t\ge \frac{3\pi}{2}(1+t)^\alpha\}$ for all  $t$ large enough, in which $\overline v(t,x)=\overline\xi(t)w(t,x)$. It is obvious to see that 
$(\partial_t-\mathcal{L})\overline v(t,x)=\overline \xi'(t)w(t,x)\ge 0$ since $w(t,x)>0$ for all $t\ge 0$ in this area and since we assume that
 $\overline \xi'(t)\ge 0$ for $t> 0$.
It remains to discuss the region $\{x\in\R~|~ -t^\delta\le x-c_* t\le \frac{3\pi}{2}(1+t)^\alpha\}$ for all large times. We divide it into three zones: 
\begin{align*}
R_1&:=\left\{x\in\R~|~ -t^\delta\le x-c_*t\le 1\right\}, \quad R_2: =\left\{x\in\R~|~ 1\le x-c_*t\le \frac{\pi}{4}(1+t)^\alpha\right\},\\
 R_3&:=\left\{x\in\R~|~ \frac{\pi}{4}(1+t)^\alpha\le x-c_*t\le \frac{3\pi}{2}(1+t)^\alpha\right\}.
\end{align*}

\paragraph{In region $R_1$.} There holds
\bqs
\frac{-t^\delta}{(1+t)^\alpha}\le \frac{x-c_*t}{(1+t)^\alpha}\le \frac{1}{(1+t)^\alpha}.
\eqs
Thanks to \eqref{1-asymp-w}, one gets 
\bqs
 -\frac{1}{(1+t)^{\frac{3}{2}-\delta}} \lesssim w(t,x)\lesssim \frac{1}{(1+t)^{\frac{3}{2}}}~~~\text{for all}~t~\text{large enough}.
\eqs
 Notice that $w(t,x)$ can be negative in this area,  therefore  the cosine perturbation needs to play a role here. We first require that $\beta>\delta>0$ so that the cosine term will be the dorminant term, that is,
\begin{equation}\label{1-BC-upper}
	\overline v(t,x)\sim \frac{1}{(1+t)^{\frac{3}{2}-\beta}}>0 ~~~\text{for all}~t~\text{large enough}.
\end{equation}
 Moreover, a straightforward computation gives that, for $t$ large enough,
 \begin{equation*}
 (\partial_t-\mathcal{L})(\overline\xi(t)w(t,x))=\overline\xi'(t)w(t,x)\gtrsim -\frac{\overline\xi'(t)}{(1+t)^{\frac{3}{2}-\delta}}
 \end{equation*}
 and
\begin{align*}
	&(\partial_t-\mathcal{L})\left(\frac{1}{(1+t)^{\frac{3}{2}-\beta}}\cos\left(\frac{x-c_*t}{(1+t)^\alpha}\right)\right)\\
	=&\frac{\beta-\frac{3}{2}}{(1+t)^{\frac{5}{2}-\beta}}\cos\left(\frac{x-c_*t}{(1+t)^\alpha}\right)+\frac{1}{(1+t)^{\frac{3}{2}-\beta}}\Bigg(\frac{c_*}{(1+t)^\alpha}+\frac{\alpha(x-c_*t)}{(1+t)^{\alpha+1}}\Bigg)\sin\left(\frac{x-c_*t}{(1+t)^\alpha}\right)
	\\
	&+\frac{1}{(1+t)^{\frac{3}{2}-\beta+2\alpha}}\cos\left(\frac{x-c_*t}{(1+t)^\alpha}\right)
	-\frac{c_*}{(1+t)^{\frac{3}{2}-\beta+\alpha}}\sin\left(\frac{x-c_*t}{(1+t)^\alpha}\right)\\
	\sim& \frac{1}{(1+t)^{\frac{3}{2}-\beta+2\alpha}}.
\end{align*}
Therefore, in order to ensure that $(\partial_t-\mathcal{L})\overline{v}(t,x)\ge 0$ for $t$ large enough in this region, it suffices to impose the condition
\bqq
\frac{1}{(1+t)^{\frac{3}{2}-\beta+2\alpha}}\gg\frac{\overline \xi'(t)}{(1+t)^{\frac{3}{2}-\delta}}~~~\text{for all}~t~\text{large enough},
\eqq
 namely,
\begin{equation}\label{1-R1-cdn}
	0\le \overline \xi'(t)\ll \frac{1}{(1+t)^{2\alpha+\delta-\beta}} ~~~\text{for all}~t~\text{large enough}.
\end{equation}

\paragraph{In region $R_2$.} There holds
\bqs
\frac{1}{(1+t)^\alpha}
\le \frac{x-c_*t}{(1+t)^\alpha}\le \frac{\pi}{4}.
\eqs
We notice that $w(t,x)$ is positive for all $t>0$ in this region, as is the cosine perturbation. Since $\overline \xi(t)$ is assumed a priori to be a bounded positive function in $(0,+\infty)$, one has that $\overline v(t,x)>0$ for all $t> 0$ in this region and 
\begin{equation*}
(\partial_t-\mathcal{L})(\overline\xi(t)w(t,x))=\overline \xi'(t)w(t,x)\ge 0
\end{equation*}
for all $t> 0$ in this region, thanks to $\overline \xi'(t)\ge 0$ in $(0,+\infty)$.
Moreover, by the same calculation as the previous region, one also obtains in $R_2$ that
\bqs
(\partial_t-\mathcal{L})\left(\frac{1}{(1+t)^{\frac{3}{2}-\beta}}\cos\left(\frac{x-c_*t}{(1+t)^\alpha}\right)\right)\sim \frac{1}{(1+t)^{\frac{3}{2}-\beta+2\alpha}}>0~~~~\text{for}~t~\text{large enough}.
\eqs
Consequently,  there obviously holds 
$(\partial_t-\mathcal{L})\overline{v}(t,x)\ge 0$ for $t$ large enough in region $R_2$.

\paragraph{In region $R_3$.} There holds
\bqs
\frac{\pi}{4}
\le \frac{x-c_*t}{(1+t)^\alpha}\le \frac{3\pi}{2}.
\eqs
It is worth to note that the cosine perturbation may be negative in this area.
Moreover, following an analogous procedure as in preceding cases, it follows that
\begin{equation*}
(\partial_t-\mathcal{L})(\overline\xi(t)w(t,x))=\overline \xi'(t)w(t,x)\sim \frac{\overline\xi'(t)}{(1+t)^{\frac{3}{2}-\alpha}}\ge 0~~~\text{for}~t~\text{large enough},
\end{equation*}
and
\bqs
(\partial_t-\mathcal{L})\left(\frac{1}{(1+t)^{\frac{3}{2}-\beta}}\cos\left(\frac{x-c_*t}{(1+t)^\alpha}\right)\right)\ge  \frac{-1}{(1+t)^{\frac{3}{2}-\beta+2\alpha}} ~~~\text{for}~t~\text{large enough}.
\eqs
 To ensure that $\overline v(t,x)>0$ and $(\partial_t-\mathcal{L})\overline{v}(t,x)\ge 0$ for $t$ large enough in this region, we require this time $\alpha>\beta$ and
\bqq\label{1-R3-cdn}
\frac{\overline \xi'(t)}{(1+t)^{\frac{3}{2}-\alpha}}\gg\frac{1}{(1+t)^{\frac{3}{2}-\beta+2\alpha}}~~~\text{for}~t~\text{large enough}.
\eqq

\paragraph{Conclusion.} Gathering \eqref{1-R1-cdn} and \eqref{1-R3-cdn}, we should impose 
$$\frac{1}{(1+t)^{3\alpha-\beta}}\ll \overline \xi'(t)\ll \frac{1}{(1+t)^{2\alpha+\delta-\beta}} ~~~\text{for all}~t~\text{large enough}.$$
This is possible so long as $\delta<\alpha$, which is exactly what we have assumed. Let us take 
$$\overline \xi'(t)\sim \frac{1}{(1+t)^{3\alpha-2\beta}}~~ \text{and}~~\overline \xi(t)=1-\frac{1}{(1+t)^{3\alpha-2\beta-1}}~~\text{for}~t\ge 0.$$
Due to our assumption that the function $\overline\xi(t)$ is positive and bounded in $(0,+\infty)$, it suffices to require
$3\alpha-2\beta-1>0$. Hence, we can fix $\delta\in(0,1/4)$ very small, then there exist $\alpha\in(1/3,1/2)$ and $\beta>0$ such that $3\alpha-2\beta>2\alpha+\delta-\beta$. To be more precise, the parameters $\delta, \beta, \alpha$ are chosen such that 
\begin{equation}
	\label{1-parameters}
	0<\delta<\beta<\min\left(\alpha-\delta,\frac{3\alpha-1}{2}\right)<\alpha<\frac{1}{2}.
\end{equation}

As a consequence,  there exists  $T_0>0$  large enough such that $c_*T_0-T_0^\delta>A$ (recall that $A\in\R$ is beyond the support of $u^0$) and such that 
$(\partial_t-\mathcal{L})\overline{v}(t,x)\ge 0$ for $t\ge T_0$ and  $x\in\R$ with $x-c_*t\ge-t^\delta$. Therefore, it immediately follows from \eqref{1-R term} that
$(\partial_t-\mathcal{L})\overline v(t,x)+R(t,x;\overline v)\ge 0$ for $t\ge T_0$ and  $x\in\R$ with $x-c_*t\ge-t^\delta$. On the other hand, we have  $\overline v(t,x)>0$ for $t\ge T_0$ and $x\in\R$ with $x\ge c_*t-t^\delta$.  Hence,  there holds $\overline v(T_0,x)>0=e^{\lambda_*x} u^0(x)=v(0,x)$ for all $x-c_*T_0 \ge -T_0^\delta$. For $x-c_*t=-t^\delta$, we observe that $v(t-T_0,x)=e^{\lambda_*(x-c_*(t-T_0))}u(t-T_0,x)\le e^{\lambda_*(c_*T_0- t^\delta)}$ since  $0\le u(t,x)\le  1$ for all $t\ge 0$ and $x\in\R$, while $\overline v(t,x)\sim (1+t)^{-3/2+\beta}$ for $t\ge T_0$ by \eqref{1-BC-upper}, up to increasing $T_0$. Then, there is $T_1>T_0$ sufficiently large such that $e^{\lambda_*(c_*T_0-t^\delta)}<(1+t)^{-3/2+\beta}$ for all $t\ge T_1$, which will yield that $\overline v(t,x)> v(t-T_0,x)$ at $x-c_*t=t^\delta$ for all $t\ge T_1$. On the other hand, one can choose $K>0$ large enough such that  $K(1+t)^{-3/2+\beta}>e^{\lambda_*(c_*T_0-t^\delta)}$, that is, $K\overline v(t,x)>v(t-T_0,x)$ at $x-c_*t=t^\delta$ for $t\in[T_0,T_1)$. Therefore, there holds $K\overline v(t,x)> v(t-T_0,x)$ for all $t\ge T_0$ and $x\in\R$ with $x-c_*t=-t^\delta$. We then conclude that $K\overline v(t,x)$ is a supersolution of \eqref{1-v-equation} for all $t\ge T_0$ and  $x\in\R$ with $x-c_*t\ge -t^\delta$. The strong maximum principle implies that
\begin{equation}
	\label{2.2-conclusion}
	K\overline v(t+T_0,x)>v(t,x)~~~\text{for}~t\ge0,~x-c_*t> -t^\delta.
\end{equation} 

\subsection{Lower barrier for $v$}

Let $\delta$, $\beta$ and $\alpha$ be fixed as in \eqref{1-parameters}. The idea is to estimate $v(t,x)$ ahead of $x-c_*t=t^\delta$. To do so, let us construct a lower barrier as follows:
\begin{equation}
	\label{1-underline v}
	\underline v(t,x)=\max\big(0,\underline\xi(t) w(t,x)  \big)
\end{equation} 
for $t>0$ and  $x\in\R$, where we assume that the unknown $\underline \xi(t)\in \mathscr{C}^1$  is positive and bounded away from 0 in $(0,+\infty)$ and satisfies $\underline \xi'(t)\le 0$ in $(0,+\infty)$, which will be made clear in the sequel. Remember that $w(t,x)>0$ for $t>0$ and $x\in\R$ with $x-c_*t>0$, whereas $w(t,x)<0$ for $t>0$ and $x\in\R$ with $x-c_*t<0$. We then derive that $\underline v\equiv0$ for $t>0$ and $x-c_*t\le 0$, and $\underline v=\underline\xi(t) w(t,x)$ for $t>0$ and $x-c_*t>0$. Let now identify that $\underline v(t,x)$ is a generalized subsolution of \eqref{1-v-equation} for $t>0$ and $x\in\R$. 

It is sufficient to look at the region $\{(t,x)\in(0,+\infty)\times\R : x>c_*t\}$.  We first claim that, there is $C>0$ sufficiently large such that
\begin{equation}
	\label{1-w-upper control}
	e^{-\lambda_*(x-c_*t)} w(t,x)\le C(1+t)^{-3/2},
\end{equation}
for $t>0$ and $x-c_*t>0$. Indeed, it follows from \eqref{1-asymp-w} that there exist $t_0$ large enough and $C>0$ large enough such that \eqref{1-w-upper control} is true for $t\ge t_0$ and $0<x-c_*t\le \sqrt{t}$. Since $w(t,x)$ is a bounded function for $t>0$ and $x\in\R$, we derive that, up to increasing $C$, \eqref{1-w-upper control} holds true for $t\in[0,t_0]$ and $0<x-c_*t\le \sqrt{t}$. In a similar way, we see that, by increasing $C$ if needed, there holds $e^{-\lambda_*(x-c_*t)} w(t,x)<e^{-\lambda_*\sqrt{t}} w(t,x)\le C(1+t)^{-3/2}$ for $t\ge 0$ and $x-c_*t>\sqrt{t}$.  Therefore, our claim \eqref{1-w-upper control} is proved for $t>0$ and $x-c_*t>0$.
On the other hand, since $f\in\mathscr{C}^2([0,1])$, there exist $M>0$ and $s_0\in(0,1)$ such that $f(s)-f'(0)s\ge -Ms^2$ for $s\in [0,s_0)$.  Eventually, let us require $\underline\xi(t)$ to solve 
\begin{equation}
	\label{1-underline xi}
	\underline\xi'(t)=-CM\underline\xi^2(t)(1+t)^{-3/2},~~~~~\text{for}~t>0,
\end{equation}
starting from $\underline\xi(0)=\underline\xi_0>0$ which is set very small such that  $\underline \xi_0 w(t,x)<s_0$ for $t>0$ and $x-c_*t>0$. We then note that $\underline\xi(t)$ is positive and uniformly bounded from above and below such that
$$0<\frac{\underline\xi_0}{1+2\underline\xi_0CM}\le \underline\xi(t)\le \underline\xi_0<+\infty~~~~~\text{for}~t\ge 0.$$
Moreover, for $t>0$ and $x-c_*t>0$, the function $\underline v(t,x)=\underline\xi(t)w(t,x)$ satisfies
\begin{align*}
	(\partial_t-\mathcal{L})\underline v(t,x)+R(t,x;\underline v)&= \underline\xi'(t) w(t,x)+f'(0)\underline v(t,x)-e^{\lambda_*(x-c_*t)}f(e^{-\lambda_*(x-c_*t)}\underline v(t,x))\\
	&\le  \underline\xi'(t) w(t,x)+M e^{-\lambda_*(x-c_*t)}\underline\xi^2(t)(w(t,x))^2\\
	&=  \left( -CM\underline\xi^2(t)(1+t)^{-3/2}+M\underline\xi^2(t) e^{-\lambda_*(x-c_*t)} w(t,x)\right) w(t,x)\le 0,
\end{align*}
 thanks to \eqref{1-w-upper control} and \eqref{1-underline xi}. 

Since $v(1,x)=e^{\lambda_*(x-c_*)}u(1,x)>0$ in $\R$ and since $\underline\xi_0w(0,x)=\underline\xi_0 p_0$ is bounded and  compactly supported in $(0,+\infty)$, there exists  $\kappa>0$ very small such that $\kappa\underline v(0,x)< v(1,x)$ in $\R$. Therefore,   $\kappa\overline v(t,x)$ is a subsolution of \eqref{1-v-equation} for all $t\ge 0$ and  $x\in\R$. By the strong maximum principle, we then conclude that 
\begin{equation}
	\label{2.3-conclusion}
	\kappa\underline v(t-1,x)<v(t,x)~~~~~\text{for}~t\ge 1,~x\in\R.
\end{equation}

\subsection{Proof of Proposition~\ref{propCont}}
\label{1-sec-proof of thm 1}

We have shown in the previous sections that the function $v$, solution of \eqref{1-v-equation}, has an upper barrier  and a lower barrier  given respectively by \eqref{1-overline v} and \eqref{1-underline v}. We are now in position to prove Proposition~\ref{propCont}. The proof is based on the comparison between $t^{\frac{3}{2}}v$ and a variant of the shifted critical KPP traveling front with logarithmic correction  in a well-chosen moving zone $|x-c_*t|\le t^\eta$ with certain small $\eta$ for all large times. Set 
\begin{equation*}
	V(t,x)=t^{\frac{3}{2}}v(t,x)~~~~\text{for}~t\ge 1,~x\in\R,
\end{equation*}
then we observe that $V$ sloves
\begin{equation}	
	\label{V-eqn}
	\begin{aligned}
		\begin{cases}
			V_t-V_{xx}+c_*V_x-\frac{3}{2t}V+\widehat R(t,x;V)=0,~~&t>1,x\in\R,\cr
			V(1,x)=v(1,x),&x\in\R.
		\end{cases}
	\end{aligned}
\end{equation}
with nonnegative term $\widehat R$ given explicitly by
\begin{equation}
	\label{hat R term}
	\widehat R(t,x;s):=f'(0)s-e^{\lambda_*(x-c_*t+\frac{3}{2\lambda_*}\ln t)}f\left(e^{-\lambda_*(x-c_*t+\frac{3}{2\lambda_*}\ln t)}s\right),~~~t\ge 1,~x\in\R, ~s\in\R.
\end{equation}
 
Let $\delta,\beta,\alpha$ be chosen as in \eqref{1-parameters}. We notice that $\beta<1/4$. Fix now
 \begin{equation}
 	\label{1-varep}
 	\eta=\beta+\varepsilon<\alpha
 \end{equation} 
 for  $\varepsilon>0$ small enough.

\paragraph{Step 1: Upper bound.}   From \eqref{2.2-conclusion}, we deduce that $Kt^{\frac{3}{2}}\overline v(t+T_0,x)>V(t,x)$ for $t$ large enough and $x-c_*t\ge -t^\delta$, with $\overline v$ given in \eqref{1-overline v}. 
Define a function $\psi(t,x)$  by
\begin{equation*}
\psi(t,x)=e^{\lambda_*(x-c_*t+\frac{3}{2\lambda_*}\ln t)}\varphi_{c_*}\left(x-c_*t+\frac{3}{2\lambda_*}\ln t+b\right)~~~\text{for}~t~\text{large enough},~|x-c_*t +\frac{3}{2\lambda_*}\ln t|\le t^\eta,
\end{equation*}
where $b\in\R$ is fixed such that  $\psi(t,x)\geq Kt^{\frac{3}{2}} \overline v(t+T_0,x)$ for $t$ large enough and $x\in\R$ with $x-c_*t +\frac{3}{2\lambda_*}\ln t=t^\eta$.    Substituting $\psi$ into the equation of $V$ yields
\begin{equation*}
\left|\psi_t-\psi_{xx}+c_*\psi_x-\frac{3}{2t}\psi+\widehat R(t,x;\psi)\right|=\frac{3}{2\lambda_* t}e^{\lambda_*(x-c_*t+\frac{3}{2\lambda_*}\ln t)} \left|\varphi'_{c_*}\left(x-c_*t+\frac{3}{2\lambda_*}\ln t+b\right)\right|\lesssim t^{-(1-\eta)}
\end{equation*}
for $t$ large enough and $|x-c_*t  +\frac{3}{2\lambda_*}\ln t|\le t^\eta$. Now, set $s:=\left(V-\psi\right)^+$, where we use the convention that $a^+=\max(0,a)$. We are then led to the following problem
\begin{equation}\label{1-s+}
	\begin{aligned}
		\begin{cases}
			s_t -s_{xx}+c_*s_x  -\frac{3}{2t}s+Q(t,x;s)\lesssim t^{-(1-\eta)},~~~&|x-c_*t +\frac{3}{2\lambda_*}\ln t|\le t^\eta,\\
				s(t,x)  =O\big(t^{\frac{3}{2}}e^{-\lambda_*t^\eta}\big), ~~~~~&x-c_*t +\frac{3}{2\lambda_*}\ln t=-t^\eta,\\
				s(t,x)  =0, ~~~~~~~~~~& x-c_*t +\frac{3}{2\lambda_*}\ln t=t^\eta,
		\end{cases}
	\end{aligned}
\end{equation}
for $t$ large enough.
Here, $Q(t,x;s)=0$ if $s=0$; otherwise,
\begin{align*}
Q(t,x;s)&=\widehat R(t,x;V)-\widehat R(t,x;\psi)\\
&=f'(0)s-e^{\lambda_*(x-c_*t+\frac{3}{2\lambda_*}\ln t)}\left(f(e^{-\lambda_*(x-c_*t+\frac{3}{2\lambda_*}\ln t)}V)-f(e^{-\lambda_*(x-c_*t+\frac{3}{2\lambda_*}\ln t)}\psi)\right)\\
&=f'(0)s-d(t,x)s\ge0, 
\end{align*} 
in which $d(t,x)$ is continuous and  bounded in $L^\infty$ norm by $f'(0)$
since $0<f(s)\le f'(0)s$ for $s\in(0,1)$. We claim that, there holds
\begin{equation}
\label{1-claim1}
\lim\limits_{t\to+\infty}\sup_{x\in\R,~|x-c_*t +\frac{3}{2\lambda_*}\ln t|\le t^\eta} s(t,x)=0.
\end{equation}
We use a comparison argument to verify this. Define 
\begin{equation*}
\overline s(t,x)=\frac{1}{t^\lambda}\cos\left(\frac{x-c_*t}{t^\gamma}\right)~~~~\text{for}~t~\text{large enough and}~|x-c_*t +\frac{3}{2\lambda_*}\ln t|\le t^\eta,
\end{equation*}
 with $0<\eta<1/4<\gamma<1/3$ such that $2\gamma+\eta<1$ and with $0<\lambda<1-2\gamma-\eta$.  One observes that $\overline s(t,x)\sim t^{-\lambda}\gg t^{\frac{3}{2}}e^{-\lambda_*t^\eta}$ for $t$ large enough and $x\in\R$ with $|x-c_*t +\frac{3}{2\lambda_*}\ln t|\le t^\eta$. Moreover, it follows from a direct computation that
\begin{equation*}
\overline s_t -\overline s_{xx}+c_* \overline s_x  -\frac{3}{2t}\overline s \sim  \frac{1}{t^{2\gamma+\lambda}}\gg\frac{1}{t^{1-\eta}} ~~~~\text{for}~t~\text{large enough and}~|x-c_*t +\frac{3}{2\lambda_*}\ln t|\le t^\eta,
\end{equation*} 
 thanks to the choice of the parameters $\eta, \gamma$ and $\lambda$. Since $Q(t,x;\overline s)\ge 0$, one concludes that $\overline s(t,x)$ is a supersolution of \eqref{1-s+} for $t$ large enough and $x\in\R$ with $|x-c_*t +\frac{3}{2\lambda_*}\ln t|\le t^\eta$.
Our claim \eqref{1-claim1} is then reached by noticing that
$$\lim\limits_{t\to+\infty}\sup_{x\in\R,~|x-c_*t +\frac{3}{2\lambda_*}\ln t|\le t^\eta} \overline s(t,x)=0.$$
 Then, it follows that $V(t,x)\le \psi(t,x)+o(1)$ uniformly in $x\in\R$ with $|x-c_*t +\frac{3}{2\lambda_*}\ln t|\le t^\eta$ as $t\to+\infty$. Hence,
\begin{equation}
\label{1-u-upper bound}
u(t,x)\le \varphi_{c_*}\left(x-c_*t+\frac{3}{2\lambda_*}\ln t+b\right)+o(1) e^{-\lambda_*(x-c_*t+\frac{3}{2\lambda_*}\ln t)}
\end{equation}
uniformly in $1\le x-c_*t  +\frac{3}{2\lambda_*}\ln t\le t^\eta$ as $t\to+\infty$.

\paragraph{Step 2: Lower bound.} The proof of this part is similar to Step 1. We sketch it for the sake of completeness. Thanks to \eqref{2.3-conclusion}, we have that  $\kappa t^{\frac{3}{2}}\underline v(t-1,x)< V(t,x)$ for $t\ge 1$  and $x\in\R$, where $\underline v(t,x)$ is given in \eqref{1-underline v}.
Define a function $\phi(t,x)$  by
\begin{equation*}
\phi(t,x)=e^{\lambda_*(x-c_*t+\frac{3}{2\lambda_*}\ln t)}\varphi_{c_*}\left(x-c_*t+\frac{3}{2\lambda_*}\ln t+a\right)~~~\text{for}~t~\text{large enough},~|x-c_*t+\frac{3}{2\lambda_*}\ln t|\le t^\eta.
\end{equation*}
Here, we fix $a\in\R$  such that  $\phi(t,x)\le \kappa t^{\frac{3}{2}}\underline v(t-1,x)$ for $t$ large enough and $x-c_*t +\frac{3}{2\lambda_*}\ln t=t^\eta$. It is also noticed that $a>b$. 

 Analogously to the previous step,
substituting $\phi$ into the equation of $V$ yields
\begin{equation*}
\left|\phi_t-\phi_{xx}+c_*\phi_x -\frac{3}{2t}\phi+\widehat R(t,x;\phi)\right|\lesssim t^{-(1-\eta)}
\end{equation*}
for $t$ large enough and $|x-c_*t +\frac{3}{2\lambda_*}\ln t|\le t^\eta$. Set $z:=\left(V-\phi\right)^-$, where we follow the convention that $a^-=\max(0,-a)$. Then, the function $z$ satisfies
\begin{equation}
	\begin{aligned}
		\begin{cases}
			z_t -z_{xx}+c_*z_x -\frac{3}{2t}z+H(t,x;z)\lesssim t^{-(1-\eta)},~~~~&|x-c_*t +\frac{3}{2\lambda_*}\ln t|\le t^\eta,\\
				z(t,x)  =O\big( t^{\frac{3}{2}}e^{-\lambda_*t^\eta}  \big), ~~~~~~  &x-c_*t +\frac{3}{2\lambda_*}\ln t=-t^\eta,\\
				z(t,x) =0, ~ & x-c_*t +\frac{3}{2\lambda_*}\ln t=t^\eta,
		\end{cases}
	\end{aligned}
\end{equation}
for $t$ large enough.
Here, $H(t,x;z)=0$ when $z=0$; otherwise,
\begin{align*}
H(t,x;z)&= \widehat R(t,x;V)-\widehat R(t,x;\varphi)\\
&=f'(0)z-e^{\lambda_*(x-c_*t+\frac{3}{2\lambda_*}\ln t)}\left(f(e^{-\lambda_*(x-c_*t+\frac{3}{2\lambda_*}\ln t)}V)-f(e^{-\lambda_*(x-c_*t+\frac{3}{2\lambda_*}\ln t)}\varphi)\right)\\
&=f'(0)z-h(t,x)z\ge0, 
\end{align*} 
in which $h(t,x)$ is  continuous and  bounded in $L^\infty$ norm by $f'(0)$
since $0<f(s)\le f'(0)s$ for $s\in(0,1)$. Following the proof of \eqref{1-claim1} in Step 1, one can show that
\begin{equation}
\label{1-claim2}
\lim\limits_{t\to+\infty}\sup_{x\in\R,~|x-c_*t +\frac{3}{2\lambda_*}\ln t|\le t^\eta} z(t,x)=0.
\end{equation}
This implies that $V(t,x)\ge \phi(t,x)+o(1)$  uniformly in   $|x-c_*t +\frac{3}{2\lambda_*}\ln t|\le t^\eta$ as $t\to+\infty$, whence
\begin{equation}
\label{1-u-lower bound}
u(t,x)\ge \varphi_{c_*}\left(x-c_*t+\frac{3}{2\lambda_*}\ln t+a\right)+o(1)e^{-\lambda_*(x-c_*t+\frac{3}{2\lambda_*}\ln t+a)},
\end{equation}
uniformly in  $1\le x-c_*t +\frac{3}{2\lambda_*}\ln t\le t^\eta$ as $t\to+\infty$.

\paragraph{Step 3: Conclusion.} Gathering \eqref{1-u-upper bound} and \eqref{1-u-lower bound}, along with the asymptotics of $\varphi_{c_*}$, it follows that for any small $\varepsilon>0$, there exists $T>0$ sufficiently large and $\hat a, \hat b\in\R$ satisfying $-\infty<\hat b<b<a<\hat a<+\infty$ such that
\begin{equation*}
	(1-\varepsilon) \varphi_{c_*}\left(x-c_*t+\frac{3}{2\lambda_*}\ln t+\hat a\right)\le u(t,x)\le (1+\varepsilon)\varphi_{c_*}\left(x-c_*t+\frac{3}{2\lambda_*}\ln t+\hat b\right)
\end{equation*}
uniformly in  $1\le x-c_*t  +\frac{3}{2\lambda_*}\ln t\le t^\eta$ for all $t\ge T$. We have therefore achieved the conclusion of Proposition~\ref{propCont}.

\section{The linearized problem on $\Z$}\label{seclinear}

One of the key message from Section~\ref{seccontinuous} in the continuous case is that solutions $w(t,x)$ of the linear equation~\eqref{1-linear-eqn-w} starting from odd and compactly supported initial conditions play a crucial role in designing accurate upper and lower barriers for the full nonlinear problem. In fact, the asymptotic expansion~\eqref{1-asymp-w} is the corner stone of the proof. We will dedicate all our efforts to proving an equivalent expansion in the discrete case.  More precisely, we consider the linearized problem
\bqq
\left\{
\begin{split}
\frac{\md}{\md t} r_j(t) &= r_{j-1}(t)-2r_j(t)+r_{j+1}(t)+f'(0)r_j(t), \quad t>0, \quad j\in\Z,\\
r_j(0) & =r_j^0, \quad j\in\Z,
\end{split}
\right.
\label{KPPlin}
\eqq
for some nontrivial bounded initial sequence $(r_j^0)_{j\in\Z}\in\ell^\infty(\Z)$. First, we perform the change of variable $r_j(t)=e^{-\lambda_*(j-c_*t)}w_j(t)$ for some new sequence $(w_j(t))_{j\in\Z}$, which  now satisfies
\bqq
\frac{\md}{\md t} w_j(t) =e^{\lambda_*} \left(w_{j-1}(t)-2w_j(t)+w_{j+1}(t)\right)-c_*(w_{j+1}(t)-w_j(t)), \quad t>0, \quad j\in\Z.
\label{KPPlin}
\eqq
And we recall that $(c_*,\lambda_*)$ are defined through \eqref{systemclamb}.

Our aim in this section is to study the temporal Green's function $(\mathscr{G}_j(t))_{j\in\Z}$ which is the solution of \eqref{KPPlin} starting from the initial sequence $\boldsymbol{\delta}$ which is defined as $\boldsymbol{\delta}_j=1$ if $j=0$ and $\boldsymbol{\delta}_j=0$ otherwise, that is
\bqq
\left\{
\begin{split}
\frac{\md}{\md t}\mathscr{G}_j(t) &=e^{\lambda_*} \left(\mathscr{G}_{j-1}(t)-2\mathscr{G}_j(t)+\mathscr{G}_{j+1}(t)\right)-c_*(\mathscr{G}_{j+1}(t)-\mathscr{G}_j(t)), \quad t>0, \quad j\in\Z.\\
\mathscr{G}_j(0) & =\boldsymbol{\delta}_j, \quad j\in\Z.
\end{split}
\right.
\label{temporalGreen}
\eqq
The motivation for studying the temporal Green's function $(\mathscr{G}_j(t))_{j\in\Z}$ stems from the fact that solutions to \eqref{KPPlin},  starting from some initial sequence $(w_j^0)_{j\in\Z}\in\ell^\infty(\Z)$, can be represented as
\bqs 
w_j(t)= \sum_{\ell \in \Z} \mathscr{G}_{j-\ell}(t) w_\ell^0,  \quad \forall t>0, \quad \forall j\in \Z.
\eqs
Unlike the continuous case, there does not exist an explicit representation formula for the temporal Green's function $(\mathscr{G}_j(t))_{j\in\Z}$. Nevertheless, we will obtain pointwise estimates for each $t>1$ and $j\in\Z$, see Propositions~\ref{propptwestim} and~\ref{proprefined}. Roughly speaking, in a first step, we will prove in Proposition~\ref{propptwestim} that for $|j-c_*t|>\theta_* t$ the temporal  Green's function is both exponentially localized in space and time, whereas for $|j-c_*t|\leq \theta_* t$ it behaves as a Gaussian profile centered at $j=c_*t$, for $t$ large enough and some universal constant $\theta_*>0$. However, such a generalized Gaussian estimate will not be enough to obtain an equivalent asymptotic expansion of \eqref{1-asymp-w} in our discrete setting. We thus need to refine our asymptotics and this is the key result of Proposition~\ref{proprefined} where we provide a full asymptotic expansion of the temporal Green's function up to Gaussian corrections of order $t^{-3/2}$. We really enforce that this result is of independent interest and should be compared to existing results in the fully discrete case for discrete convolution powers \cite{CF20,DSC14,Coeuret22} and to local limit theorems in probability theory \cite{Petrov75}. Eventually, we will show in Section \ref{sec2.4} how sharp asymptotics of $w_j(t)$ emanating from odd and finitely supported initial data can be derived from this improved result.

\subsection{The spatial Green's function}

The starting point of our approach is the study of the so-called spatial Green's function which we now introduce. Let $\mathscr{L}$ denote the operator acting on a sequence $\br=(r_j)_{j\in\Z}$ as 
\bqq
\forall j\in\Z, \quad (\mathscr{L}\br)_j:=e^{\lambda_*} \left(r_{j-1}-2r_j+r_{j+1}\right)-c_*(r_{j+1}-r_j),
\label{deflinopL}
\eqq
 then its spectrum on $\ell^2(\Z)$ is given by the parametrized closed curve 
\bqs
\sigma(\mathscr{L})=\left\{ e^{\lambda_*} \left(e^{-\mbi\xi}-2+e^{\mbi\xi}\right)-c_*(e^{-\mbi\xi}-1) ~|~ \xi \in [-\pi, \pi] \right\}.
\eqs
This is a direct a consequence of the Wiener-Levy theorem, see~\cite{Newman}, which characterizes invertible elements of $\ell^1(\Z)$ for the convolution. Indeed, $\mathscr{L}$ can naturally be written as a convolution product
\bqs
(\mathscr{L}\br)_j=\sum_{\ell \in\Z} a_\ell r_{j-\ell},
\eqs
with the sequence $(a_\ell)_{\ell \in\Z}\in\ell^1(\Z)$ defined as
\bqs
a_{-1}= e^{\lambda_*}-c_*, \quad a_0=-2e^{\lambda_*}+c_*, \quad a_1=e^{\lambda_*}, \quad \text{ with } a_j=0 \text{ for } |j|\geq 2.
\eqs
For each $\nu\in\sigma(\mathscr{L})$, we have $\Re(\nu)=\left(e^{\lambda_*}+e^{-\lambda_*}\right)\left(\cos(\xi)-1\right)\leq 0$ and $\Im(\nu)=\left(e^{\lambda_*}-e^{-\lambda_*}\right)\sin(\xi)$ such that $\sigma(\mathscr{L})$ is an ellipse that touches the imaginary axis only at the origin and is located in the left half plane otherwise. Since $\sigma(\mathscr{L})$ is a closed curve in the complex plane, the resolvent set of $\mathscr{L}$, defined as $\C\backslash\sigma(\mathscr{L})$, is given by the union of two open sets which we refer to as the interior and the exterior sets. The interior set is the region enclosed by $\sigma(\mathscr{L})$, while the exterior, denoted by $\mathscr{U}$, is the complementary region which at least contains the set $\left\{ \nu \in\C ~|~ \Re(\nu)\geq 0\right\} \backslash\{0\}$.  We refer to Figure~\ref{fig:spectrum} for an illustration of the geometry of $\sigma(\mathscr{L})$ in the complex plane.

\begin{figure}[t!]
  \centering
  \includegraphics[width=.35\textwidth]{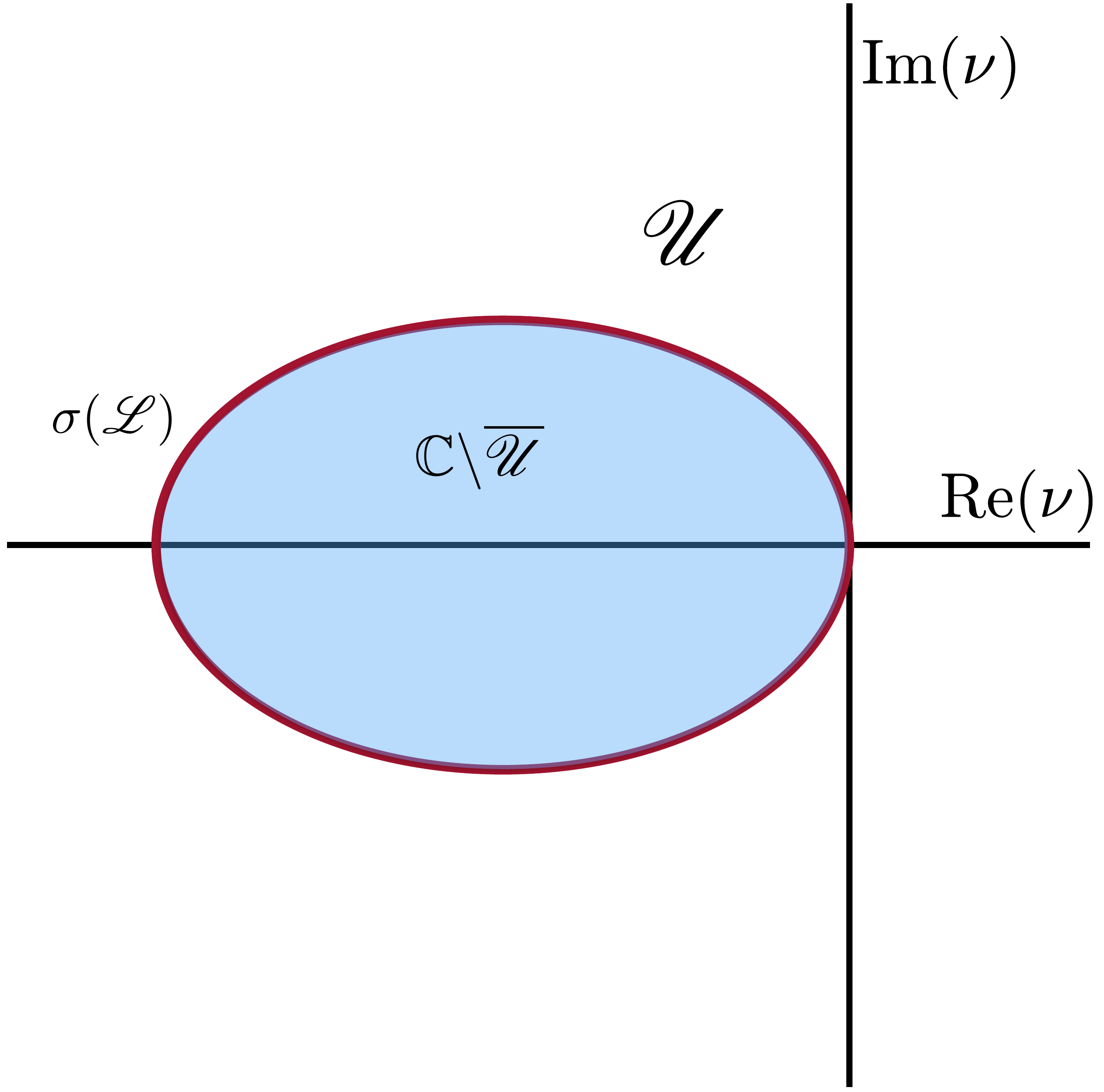}
  \caption{Illustration of the geometry of the spectrum $\sigma(\mathscr{L})$ (dark red curve) in the complex plane $\C$. The curve $\sigma(\mathscr{L})$ is an ellipse that touches the imaginary axis only at the origin and is located in the left half plane otherwise. The resolvent set of $\mathscr{L}$ is composed of two open regions:  the so-called  exterior $\mathscr{U}$ and the interior $\C\backslash\overline{\mathscr{U}}$ (blue shaded region) which is the region enclosed by $\sigma(\mathscr{L})$.}
  \label{fig:spectrum}
\end{figure}

For each $\nu\in\mathscr{U}$, we can define the spatial Green's function as the sequence $\bG(\nu)\in \ell^2(\Z)$ solution of
\bqq
\left(\nu \mathrm{Id}-\mathscr{L}\right)\bG(\nu)=\boldsymbol{\delta},
\label{spatialGreen}
\eqq
where we recall that the sequence $\boldsymbol{\delta}$ is defined as $\boldsymbol{\delta}_j=1$ if $j=0$ and $\boldsymbol{\delta}_j=0$ otherwise, and that we have denoted $\mathrm{Id}$ the identity operator acting on $\ell^2(\Z)$.  Anticipating with the forthcoming section, we already remark that $\mathscr{G}_j(t)$ can be represented through the inverse Laplace formula
\bqs
\mathscr{G}_j(t) = \frac{1}{2\pi \mbi}\int_\Gamma e^{\nu t} \bG_j(\nu)\md \nu, \quad \forall \, t>0, \quad \forall j\in\Z,
\eqs
where $\Gamma\subset\mathscr{U}$ is some well-chosen contour in the complex plane. The key point of the analysis of the remaining of this section will be to obtain pointwise estimates on $\bG_j(\nu)$ which will eventually lead to pointwise estimates for $\mathscr{G}_j(t)$. To do so, we introduce the vector $\bW_j(\nu)=(\bG_{j-1}(\nu),\bG_j(\nu))^\mathbf{T}$ for $j\in\Z$ such that the above relation \eqref{spatialGreen} rewrites
\bqq
\bW_{j+1}(\nu)=\mathbb{A}(\nu)\bW_j(\nu)-e^{\lambda_*}\boldsymbol{\delta}_j \mathbf{e}, \quad j\in\Z,
\label{recurrenceW}
\eqq
where $\mathbf{e}:=(0,1)^\mathbf{T}$ and the matrix $\mathbb{A}(\nu)$ is given by
\bqs
\mathbb{A}(\nu):=\left(
\begin{matrix}
0 & 1 \\
-e^{2\lambda_*} & \frac{\nu+e^{\lambda_*}+e^{-\lambda_*}}{e^{-\lambda_*}}
\end{matrix}
\right).
\eqs
The eigenvalues of the above matrix are given by
\bqq
\rho_\pm(\nu)= \frac{\nu+e^{\lambda_*}+e^{-\lambda_*}}{2e^{-\lambda_*}}\pm \frac{\sqrt{\left(\nu+e^{\lambda_*}+e^{-\lambda_*}\right)^2-4}}{2e^{-\lambda_*}}.
\label{eigenvaluesAnu}
\eqq
For each $\nu\in\mathscr{U}$, we readily remark that we have the spectral splitting $|\rho_+(\nu)|>1$ and $|\rho_-(\nu)|<1$. With these notations in hands, we start to estimate the spatial Green's function away from the tangency point at the origin.

\begin{lem}[Local bounds]\label{lemloc}
For each $\underline{\nu}\in\mathscr{U}$, there exist $C,\delta,\beta>0$ depending on $\underline{\nu}$ such that 
\bqs
\left| \bG_j(\nu)\right| \leq Ce^{-\beta |j|}, \quad \forall \, \nu \in B_\delta(\underline{\nu}), \quad  \forall \, j\in\Z.
\eqs
\end{lem}

\begin{proof} We let $\underline{\nu}\in\mathscr{U}$ be fixed. We know that the matrix $\mathbb{A}(\nu)$ above is well defined and holomorphic in a sufficiently small neighborhood of $\underline{\nu}$. Since we have an explicit expression of the two eigenvalues of the matrix $\mathbb{A}(\nu)$, we trivially have a consistant splitting for all $\nu$ close to $\underline{\nu}$. We define $\bbE^{s,u}(\nu)$ the stable/unstable subspaces of $\mathbb{A}(\nu)$ for all $\nu$ close to $\underline{\nu}$. As a consequence, for $\delta>0$ small enough, we have the decomposition
\bqs
\C^2=\bbE^{s}(\nu)\oplus\bbE^{u}(\nu), \quad \forall \, \nu \in B_\delta(\underline{\nu}).
\eqs
And we denote $\pi^{s,u}(\nu)$ the associated projectors, which are given by contour integrals. For instance, we have
\bqs
\pi^s(\nu)=\frac{1}{2\pi \mbi}\int_\gamma (z I_2-\mathbb{A}(\nu))^{-1}\md z,
\eqs
where $\gamma$ is a contour that encloses the stable eigenvalue $\rho_-(\nu)$ and $I_2\in\mathscr{M}_2(\R)$ is the $2\times2$ identity matrix. A similar formula holds for $\pi^u(\nu)$, with this time a contour that encloses the unstable eigenvalue $\rho_+(\nu)$. This shows that $\pi^{s,u}(\nu)$ depends holomorphically on $\nu$ in $B_\delta(\underline{\nu})$ and, consequently, the stable and unstable subspaces $\bbE^{s,u}(\nu)$ also depend holomorphically on $\nu$ in $B_\delta(\underline{\nu})$, see \cite{Kato}. Up to taking $\delta$ even smaller, we can ensure that any $\nu\in B_\delta(\underline{\nu})$ lies in the exterior $\mathscr{U}$ of the resolvent set of $\mathscr{L}$, hence there exists a unique sequence $(\bW_j(\nu))_{j\in\Z}\in\ell^2(\Z)$ solution to \eqref{recurrenceW}.  Since the dynamics of the iteration \eqref{recurrenceW} for such $\nu$ has a hyperbolic dichotomy \cite{CF20}, the solution to \eqref{recurrenceW} is given by integrating either from $j$ to $+\infty$, or from $-\infty$ to $j-1$, depending on whether we compute the unstable or stable components of the vector $\bW_j(\nu)$. As a consequence, we can easily obtain the stable and unstable components of each $\bW_j(\nu)$ for $j\in\Z$ as
\bqs
\pi^{s}(\nu)\bW_j(\nu) = -e^{\lambda_*} \sum_{\ell=-\infty}^{j-1} \boldsymbol{\delta}_{\ell} \rho_-(\nu)^{j-1-\ell}\pi^{s}(\nu)\mathbf{e}, \quad j\in\Z,
\eqs
for the stable component, and
\bqs
\pi^{u}(\nu)\bW_j(\nu) = e^{\lambda_*} \sum_{\ell=0}^\infty \boldsymbol{\delta}_{j+\ell} \rho_+(\nu)^{-1-\ell}\pi^{u}(\nu)\mathbf{e}, \quad j\in\Z,
\eqs
for the unstable one. 
It is important to note that the sequence $\boldsymbol{\delta}$ has only one nonzero coefficient at $j=0$. Hence, we deduce that
\bqs
\pi^{s}(\nu)\bW_j(\nu)=\left\{
\begin{array}{lcl}
0 & \text{for} & j\leq 0,\\
-e^{\lambda_*}\rho_-(\nu)^{j-1}\pi^{s}(\nu)\mathbf{e}& \text{for} & j\geq 1,
\end{array}
\right.
\eqs
while
\bqs
\pi^{u}(\nu)\bW_j(\nu)=\left\{
\begin{array}{lcl}
0 & \text{for} & j\geq 1,\\
e^{\lambda_*} \rho_+(\nu)^{-1+j}\pi^{u}(\nu)\mathbf{e}& \text{for} & j\leq 0.
\end{array}
\right.
\eqs
Next, we remark that there exists $\beta>0$ such that
\bqs
\forall \, \nu \in B_\delta(\underline{\nu}), \quad |\rho_-(\nu)| \leq e^{-\beta}, \quad \text{ and } \quad |\rho_+(\nu)| \geq e^{\beta}.
\eqs
As a consequence, we have
\bqq
\forall \, \nu \in B_\delta(\underline{\nu}) \quad \left| \pi^{s}(\nu)\bW_j(\nu) \right| \leq \left\{
\begin{array}{lcl}
0 & \text{for} & j\leq 0,\\
C e^{-\beta(j-1)}& \text{for} & j\geq 1,
\end{array}
\right.
\label{estimateWs}
\eqq
while
\bqq
\forall \, \nu \in B_\delta(\underline{\nu}) \quad  \left|\pi^{u}(\nu)\bW_j(\nu)\right|\leq \left\{
\begin{array}{lcl}
0 & \text{for} & j\geq 1,\\
Ce^{\beta(j+1)}& \text{for} & j\leq 0,
\end{array}
\right.
\label{estimateWu}
\eqq
where $C>0$ is a uniform bound on $B_\delta(\underline{\nu})$ of $| e^{\lambda_*}\pi^{s,u}(\nu)\mathbf{e} |$. Finally, adding \eqref{estimateWs} and \eqref{estimateWu} concludes the proof of the lemma, since the spatial Green's function $\bG_j(\nu)$ is the second coordinate of the vector $\bW_j(\nu)\in\C^2$.
\end{proof}

\begin{lem}[Bounds at infinity]\label{leminf} There exist a radius $R>1$ and two constants $C>0$ and $\kappa>0$ such that $\left\{ \nu \in \C~|~ |\nu|\geq R\right\}\subset \mathscr{U}$ and
\bqs
\forall \, |\nu|\geq R, \quad  \forall \, j\in\Z, \quad \left| \bG_j(\nu)\right| \leq Ce^{-\kappa |j|}.
\eqs
\end{lem}

\begin{proof}
The proof is similar to Lemma~\ref{lemloc} and simply relies on the fact that from their definition \eqref{eigenvaluesAnu} we have
\bqs
\underset{|\nu|\rightarrow+\infty}{\lim}|\nu| \left| \rho_-(\nu)\right| = e^{\lambda_* }\quad \text{ and } \quad \underset{|\nu|\rightarrow+\infty}{\lim} \frac{\left| \rho_+(\nu)\right|}{|\nu|} = e^{\lambda_*},
\eqs
together with our expression for $\pi^{s,u}(\nu)\bW_j(\nu)$ when $\nu\in\mathscr{U}$.
\end{proof}

Next, we recall that $\sigma(\mathscr{L})$ is tangent to the imaginary axis at the origin. This translates to $\rho_-(0)=1$ with $\rho_+(0)=e^{2\lambda_*}>1$. For $|\nu|\rightarrow 0$, we also note that $\rho_-(\nu)$ has the following expansion
\bqq
\rho_-(\nu)= 1- \frac{\nu}{c_*}+\frac{\nu^2}{c_*^3e^{-\lambda_*}}+O\left(|\nu|^3\right).
\label{expansionrho}
\eqq
It is thus important to examine the behavior of the spatial Green's function close to the origin. Let us remark that the spatial Green's function is well-defined in $B_\delta(0)\cap \mathscr{U}$ for any radius $\delta>0$. Our goal is to extend holomorphically $\bG_j(\nu)$ to a whole neighborhood of $\nu=0$ which amounts to passing through the essential spectrum of $\mathscr{L}$, see \cite{CF20,ZH98} for a similar argument in different contexts.

\begin{lem}[Bounds close to the origin]\label{lemorigin}
There exist some $\epsilon>0$ and constants $C>0$ and $\omega>0$ such that for any $j\in\Z$, the component $\bG_j(\nu)$ defined on $B_\epsilon(0)\cap \mathscr{U}$ extends holomorphically to the whole ball $B_\epsilon(0)$ with respect to $\nu$, and the holomorphic extension (still denoted by $\bG_j(\nu)$) satisfies the bounds
\bqs
\forall \, \nu \in B_\epsilon(0), \quad  \forall \, j\leq0, \quad \left| \bG_j(\nu)\right| \leq C e^{-\omega|j|} ,
\eqs
and
\bqs
\forall \, \nu \in B_\epsilon(0), \quad  \forall \, j\geq1, \quad \left|\bG_j(\nu)-\frac{1}{c_*} e^{\varpi(\nu)j} \right| \leq C e^{\Re(\varpi(\nu))j}|\nu|,
\eqs
where
\bqq
\varpi(\nu) := \ln( \rho_-(\nu))=- \frac{\nu}{c_*}+\frac{\cosh(\lambda_*)}{c_*^3}\nu^2+O\left(|\nu|^3\right), \quad \forall \, \nu \in B_\epsilon(0),
\label{expansionvarpinu}
\eqq
is holomorphic in the ball $B_\epsilon(0)$.
\end{lem}

\begin{proof}
Most ingredients are similar to the ones we already used in the proof of Lemma~\ref{lemloc}. We just need to adapt our notation, since the hyperbolic dichotomy of $\mathbb{A}(\nu)$ does not hold any longer in a whole neighborhood of the origin $\nu=0$. We readily see that $\rho_-(\nu)$ extends holomorphically to a neighborhood of $\nu=0$, with corresponding eigenvector
\bqs
\bbE^0(\nu):=\left(
\begin{matrix}
1\\
\rho_-(\nu)
\end{matrix}
\right)\in\C^2,
\eqs
which also depends holomorphically on $\nu$ in a neighborhood of $\nu=0$. Note that $\bbE^0(\nu)$ contributes to the stable subspace of $\mathbb{A}(\nu)$ for those $\nu\in\mathscr{U}$ close to $\nu=0$ but the situation is less clear as $\nu$ passes the essential spectrum. Nevertheless, we still have the decomposition
\bqs
\C^2=\bbE^0(\nu)\oplus \bbE^{u}(\nu), \quad \forall \, \nu \in B_\epsilon(0),
\eqs 
for a sufficiently small $\epsilon>0$. We denote by $\pi^0(\nu)$ and $\pi^u(\nu)$ the holomoprhic projectors associated with the above decomposition. We readily see that for each $\nu\in B_\epsilon(0) \cap \mathscr{U}$, we can reuse the expressions derived in Lemma~\ref{lemloc} to write
\bqs
\pi^0(\nu)\bW_j(\nu) = -e^{\lambda_*} \sum_{\ell=-\infty}^{j-1} \boldsymbol{\delta}_{\ell} \rho_-(\nu)^{j-1-\ell}\pi^0(\nu)\mathbf{e}, \quad j\in\Z,
\eqs
 and
\bqs
\pi^{u}(\nu)\bW_j (\nu)= e^{\lambda_*} \sum_{\ell=0}^\infty \boldsymbol{\delta}_{j+\ell} \rho_+(\nu)^{-1-\ell}\pi^{u}(\nu)\mathbf{e}, \quad j\in\Z.
\eqs
The unstable component $\pi^{u}(\nu)\bW_j(\nu)$ obviously extends holomorphically to the whole ball $B_\epsilon(0)$, as a consequence, we get
\bqs
\left| \pi^{u}(\nu)\bW_j (\nu)\right| \leq C e^{-\omega|j|}, \quad \forall \, \nu \in B_\epsilon(0), \quad  \forall \, j\leq0,
\eqs
for appropriate constants $C>0$ and $\omega>0$. We now focus on the vector $\pi^0(\nu)\bW_j(\nu) $ and readily remark that it also extends holomorphically to the whole ball $B_\epsilon(0)$, and we deduce that
\bqq
\forall \, \nu \in B_\epsilon(0), \quad \bG_j(\nu)=\langle \pi^0(\nu)\bW_j (\nu),\mathbf{e}\rangle_{\C^2} =  \left\{
\begin{array}{lcl}
0 & \text{for} & j\leq 0,\\
\Psi(\nu)e^{\varpi(\nu)j}& \text{for} & j\geq 1,
\end{array}
\right.
\label{defGjnu0}
\eqq
where we have set
\bqq
\Psi(\nu):= - e^{\lambda_*}\frac{\langle \pi^0(\nu)\mathbf{e},\mathbf{e}\rangle_{\C^2}}{e^{\varpi(\nu)}}, \quad \forall \, \nu \in B_\epsilon(0).
\label{defPsinu}
\eqq
Note that eventually upon reducing the size of $\epsilon$, we can always ensure that $\varpi(\nu)=\ln (\rho_-(\nu))$ is holomorphic on $B_\epsilon(0)$, and thus 
$\Psi$ is also holomorphic on $B_\epsilon(0)$  with $\Psi(0)=\frac{1}{c_*}>0$.  Indeed, the projection $\pi^0(\nu)$ can be expressed as
\bqs
\pi^0(\nu)= \frac{1}{\rho_+(\nu)-\rho_-(\nu)} \left\langle \, \cdot \,,  \left( \begin{matrix} \rho_+(\nu) \\ -1 \end{matrix}\right) \right\rangle_{\C^2} \left( \begin{matrix} 1 \\ \rho_-(\nu) \end{matrix}\right),
\eqs
such that
\bqs
\Psi(0)=-e^{\lambda_*} \langle \pi^0(0)\mathbf{e}, \mathbf{e} \rangle_{\C^2}= \frac{e^{\lambda_*}}{e^{2\lambda_*}-1}=\frac{1}{c_*}>0.
\eqs
Here, we denoted by $\langle \cdot, \cdot \rangle_{\C^2}$ the usual scalar product on $\C^2$.
\end{proof}

Recalling that $\sigma(\mathscr{L})$ is an ellipse tangent to the imaginary axis at the origin, with $\epsilon>0$ fixed in Lemma~\ref{lemorigin}, there exist $\gamma_\epsilon>0$ and $\omega_\epsilon \in(0,\pi/2)$ such that the sector $\left\{ \gamma_\epsilon+\nu \in \C ~|~ |\arg(\nu)| \leq \frac{\pi}{2}+\omega_\epsilon\right\}$ is contained in $\mathscr{U}$ and intersects $B_\epsilon(0)$ twice in the left-half plane. We then denote
\bqs
\mathscr{S}_{\epsilon,R}:=\left\{ \gamma_\epsilon+\nu \in \C ~|~ |\arg(\nu)| \leq \frac{\pi}{2}+\omega_\epsilon\right\} \cap \overline{\left( B_R(0) \backslash B_\epsilon(0)\right)},
\eqs
with $R>1$ defined Lemma~\ref{leminf}, and readily remark that $\mathscr{S}_{\epsilon,R}\subset \mathscr{U}$.

\begin{lem}[Intermediate bounds]\label{lemint}
For $R>1$ and $\epsilon>0$ fixed in Lemma~\ref{leminf} and Lemma~\ref{lemorigin} respectively, there exist $C>0$ and $\beta_0>0$ such that 
\bqs
\forall \, \nu \in  \mathscr{S}_{\epsilon,R},  \quad \left| \bG_j(\nu)\right| \leq C \left\{
\begin{array}{lcl}
 e^{-\beta_0|j|}, & \text{for} & j\leq 0,\\
1, & \text{for} & j\geq 1.
\end{array}
\right.
\eqs
\end{lem}
\begin{proof}
By construction, the set $\mathscr{S}_{\epsilon,R}$ is a compact domain of $\C$. As a consequence, we can  obtain uniform constants from Lemma~\ref{lemloc} by using a simple compactness argument.
\end{proof}

\subsection{The temporal Green's function}

In this section, we study the temporal Green's function $(\mathscr{G}_j(t))_{j\in\Z}$ solution of \eqref{temporalGreen} starting from the initial sequence $\boldsymbol{\delta}$.  We use the inverse Laplace transform to express $(\mathscr{G}_j(t))_{j\in\Z}$ as the following contour integral
\bqs
\forall \, t>0, \quad \forall j\in\Z, \, \quad \mathscr{G}_j(t) = \frac{1}{2\pi \mbi}\int_\Gamma e^{\nu t} \bG_j(\nu)\md \nu,
\eqs
where $\Gamma\subset\C$ is some well-chosen contour in the complex plane which does not intersect the spectrum of $\mathscr{L}$. For example, one can take the sectorial contour $\Gamma=\left\{ \gamma_0-\gamma_1|\xi|+\mbi \xi ~|~ \xi\in\R\right\}$ for some well chosen $\gamma_{0,1}>0$. Our aim is to obtain pointwise bounds on $\mathscr{G}_j(t)$ for $t\geq 1$ and our main technique is to deform the contour $\Gamma$ in the complex plane in order to obtain sharp asymptotics. Note that we allow the deformed contour to depend eventually on $j$ and $t$. We will divide the analysis into several cases. We will first show that the pointwise temporal Green's function $\mathscr{G}_j(t)$ decays exponentially in time and space whenever $j-c_*t$ is sufficiently large and $t\geq1$ whereas for $j\approx c_*t$ we will prove that the pointwise temporal Green's function $\mathscr{G}_j(t)$ has a generalized Gaussian estimate. Let us note that the choice of the contours are rather standard, and we refer to \cite{CF20,ZH98} for similar computations in different settings.


\subsubsection{Exponential pointwise estimates away from $j\approx c_*t$}

In the next lemmas, we show that the pointwise temporal Green's function $\mathscr{G}_j(t)$ decays exponentially in time and space whenever $j-c_*t$ is sufficiently large and $t\geq1$.

\begin{lem}\label{lemjlarge}
There exists a constant $L>4R/\kappa$ with $R>1$ chosen in Lemma~\ref{leminf}, such that for $|j|>Lt$ and $t\geq 1$, one has the estimate
\bqs
\left| \mathscr{G}_j(t) \right| \leq C e^{-\kappa_1 t-\kappa_2|j|},
\eqs
for some $C>0$ and $\kappa_{1,2}>0$.
\end{lem}

\begin{proof}
We let $|j|>Lt$ and $t\geq 1$.  We recall the estimate of Lemma~\ref{leminf} 
\bqs
\forall |\nu| \geq R, \quad  \forall \, j\in\Z, \quad \left| \bG_j(\nu)\right| \leq Ce^{-\kappa  |j|}.
\eqs
Let $L>4R/\kappa$ and we take the contour which is defined as
\bqs
\Gamma=\left\{ 2R-\delta |\xi|+\mbi \xi ~|~ \xi\in\R\right\},
\eqs
with $0<\delta<\sqrt{3}$ small enough such that $|\nu|>R$ for all $\nu\in\Gamma$. As a consequence, we obtain that
\bqs
\left| \frac{1}{2\pi \mbi}\int_{\Gamma} e^{\nu t} \bG_j(\nu)\md \nu \right| \leq C e^{2Rt -\kappa|j|} \int_0^\infty e^{-\delta \xi t}\md \xi \leq C_0 \frac{e^{2Rt -\kappa|j|}}{t}.
\eqs
Since $|j|>Lt$, we have that
\bqs
2Rt -\kappa|j| < \underbrace{\left(2R-\frac{\kappa L}{2}\right)}_{<0}t -\frac{\kappa}{2}|j|,
\eqs
and the result follows. 
\end{proof}

From now on, we will always assume that $|j| \leq Lt$ and $t\geq1$. We first deal with the case where $j\leq0$.

\begin{lem}
There are some $C>0$ and $\beta_{1,2}>0$ such that for any $-Lt \leq j \leq 0$ with $t\geq1$, one has the bound
\bqs
\left| \mathscr{G}_j(t) \right| \leq C e^{-\beta_1 t-\beta_2|j|}.
\eqs
\end{lem}

\begin{proof}
We consider $\Gamma$ as the union of $\Gamma_1$ and $\Gamma_2$ where $\Gamma_1$ is a vertical line within the ball $B_\epsilon(0)$ and $\Gamma_2$ is given as $\Gamma_2=\left\{-\delta |\xi|+\mbi \xi~|~ \xi\in\R,~ |\xi| \geq \xi_*\right\}$. More precisely, there exists $\eta \in(0,\epsilon)$ small enough such that $-\eta\pm \mbi \sqrt{\epsilon^2-\eta^2} \in \mathscr{U}$. We define $\Gamma_1=\left\{ - \eta + \mbi \xi ~|~ \xi\in\R,~ |\xi| \leq \sqrt{\epsilon^2-\eta^2} \right\} \subset B_\epsilon(0)$ and $\delta=\frac{\eta}{\sqrt{\epsilon^2-\eta^2}}$ with $\xi_*=\sqrt{\epsilon^2-\eta^2}$ in the definition of $\Gamma_2$. We have
\bqs
\left| \frac{1}{2\pi \mbi}\int_{\Gamma_1} e^{\nu t} \bG_j(\nu)\md \nu \right| \leq C e^{-\eta t - \omega |j|},
\eqs
using the estimate on $\bG_j(\nu)$ from Lemma~\ref{lemorigin}. We denote $\Gamma_2^{in}=\Gamma_2 \cap B_R(0)$ and $\Gamma_2^{out}=\Gamma_2\backslash \Gamma_2^{in}$. On $\Gamma_2^{in}$, we use Lemma~\ref{lemint}
\bqs
\left| \frac{1}{2\pi \mbi}\int_{\Gamma_2^{in}} e^{\nu t} \bG_j(\nu)\md \nu \right| \leq \frac{C}{t} e^{-\eta t-\beta_0|j|},
\eqs
and on $\Gamma_2^{out}$ we use the bound of Lemma~\ref{leminf} to get
\bqs
\left| \frac{1}{2\pi \mbi}\int_{\Gamma_2^{out}} e^{\nu t} \bG_j(\nu)\md \nu \right| \leq \frac{C}{t} e^{-\eta t-\beta|j|}.
\eqs
\end{proof}

Actually, the above estimate can be easily extended for $j\in\Z$ with $1\leq j\leq (c_*-\theta)t$ for any $\theta\in(0,c_*)$.

\begin{lem}
For all $\theta\in(0,c_*)$, there exist $C>0$ and $\beta_{1,2}>0$ such that for all $1\leq j\leq (c_*-\theta)t$ with $t\geq1$  such that $(c_*-\theta)t\geq1$, one has the bound
\bqs
\left| \mathscr{G}_j(t) \right| \leq C e^{-\beta_1 t-\beta_2|j|}.
\eqs
\end{lem}

\begin{proof}
We consider the same contours as in the previous Lemma. On $\Gamma_2^{out}$, we have the exact same estimates since the spatial Green's function enjoys the same pointwise bound. On $\Gamma_2^{in}$, we note that we have
\bqs
\left| \frac{1}{2\pi \mbi}\int_{\Gamma_2^{in}} e^{\nu t} \bG_j(\nu)\md \nu \right| \leq \frac{C}{t} e^{-\eta t},
\eqs
and as $j \leq (c_*-\theta)t$ we have $-\eta t\leq -\frac{\eta}{2} t -\frac{\eta}{2(c_*-\theta)} |j|$.  The main difference, is that now, on $\Gamma_1$, we have
\bqs
\forall \nu \in \Gamma_1,\quad | \bG_j(\nu)| \leq Ce^{\Re(\varpi(\nu))j},
\eqs
from Lemma~\ref{lemorigin}. We denote by $\Gamma_1^\pm$ the portion of $\Gamma_1$ where $\Re(\varpi(\nu))$ is respectively positive and negative. On $\Gamma_1^-$, we readily obtain
\bqs
\left| \frac{1}{2\pi \mbi}\int_{\Gamma_1^{-}} e^{\nu t} \bG_j(\nu)\md \nu \right| \leq C e^{-\eta t}.
\eqs
Next, from the expansion \eqref{expansionvarpinu}, we get that
\bqs
\forall \nu \in B_\epsilon(0), \quad \Re(\varpi(\nu))=-\frac{\Re(\nu)}{c_*}+\frac{\cosh(\lambda_*)}{c_*^3}\left( \Re(\nu)^2-\Im(\nu)^2\right)+O(|\nu|^3),
\eqs
such that we can solve $ \Re(\varpi(\nu))=0$ for $\Re(\nu)$ as a function of $\Im(\nu)$ and we get
\bqs
\Re(\varpi(\nu))=0 \Leftrightarrow \Re(\nu) = -\frac{\cosh(\lambda_*)}{c_*^2} \Im(\nu)^2 + O(|\Im(\nu)|^3) \text{ for } \nu \in B_\epsilon(0).
\eqs
This implies that there exists $\beta_0>0$ such that
\bqs
|\Im(\nu)| \leq \beta_0 \sqrt{\eta}, \quad \nu \in  \Gamma_1^+.
\eqs
We deduce that
\bqs
\Re(\varpi(\nu)) \leq \frac{\eta}{c_*} +\beta_1 \eta ^{3/2},
\eqs
for some $\beta_1>0$ and any $\nu \in  \Gamma_1^+$. As a consequence, we can derive the following bound
\bqs
-\eta t+j \Re(\varpi(\nu)) \leq  \eta t \left( -1 + \frac{j}{c_* t} + \beta_1 \frac{\sqrt{\eta}j}{ t}\right)  \leq \eta t \left( -\frac{\theta}{c_*} + \beta_1(c_*-\theta)\sqrt{\eta}\right).
\eqs
We can always reduce $\eta$ such that $\beta_1(c_*-\theta)\sqrt{\eta}< \frac{\theta}{2c_*}$ and we get
\bqs
\left| \frac{1}{2\pi \mbi}\int_{\Gamma_1^{+}} e^{\nu t} \bG_j(\nu)\md \nu \right| \leq C e^{-\frac{\eta\theta}{2c_*} t}.
\eqs
\end{proof}

\subsubsection{Generalized Gaussian estimate}

We are thus now led to study the case where $  (c_*-\theta)t \leq j \leq Lt$ with $t\geq1$ and the main result of this section is a generalized Gaussian estimate for the Green's function which reads as follows. Let us already note that the large constant $L>0$ from Lemma~\ref{lemjlarge} can be fixed large enough such that it further verifies $L>c_*$, which we assume throughout the sequel.

\begin{lem}[Generalized Gaussian estimate]\label{lemgengauss}
For all $\theta\in(0,c_*)$,  there exist two constants $C>0$ and $\beta>0$ such that for any $(c_*-\theta)t \leq j \leq Lt$ with $j\in\Z$ and $t\geq1$, the temporal Green's function satisfies the estimate
\bqs
\left| \mathscr{G}_j(t) \right| \leq   \frac{C}{\sqrt{t}} \exp\left(-\beta \frac{(j-c_*t)^2}{t}\right).
\eqs
\end{lem}

Before proceeding with the proof, we need to introduce some notations. First, from Lemma~\ref{lemorigin}, one can find $\epsilon_*>0$ and two constants $0<\beta_*<\cosh(\lambda_*)<\beta^*$ such that for each $\epsilon \in (0,\epsilon_*)$ there exist constants $C>0$ and $\omega>0$ such that
\bqs
\forall \nu\in B_\epsilon(0),\quad |\bG_j(\nu)| \leq \left\{
\begin{split}
C e^{-\omega|j|}, &\quad j\leq0,\\
C e^{\Re(\varpi(\nu))j},& \quad j\geq 1,
\end{split}
\right.
\eqs
together with the bound
\bqq
\Re(\varpi(\nu))\leq -\frac{1}{c_*}\Re(\nu)+\frac{\beta^*}{c_*^3}\Re(\nu)^2-\frac{\beta_*}{c_*^3}\Im(\nu)^2, \quad \forall \nu \in B_\epsilon(0),
\label{boundRevarpi}
\eqq
where $\varpi$ is holomorphic in $B_\epsilon(0)$ and given in Lemma~\ref{lemorigin}. Note that the bound \eqref{boundRevarpi} and the existence of the constants $0<\beta_*<\cosh(\lambda_*)<\beta^*$ directly come from the expansion \eqref{expansionvarpinu} of $\varpi$. Note also that $\beta_*,\beta^*$ only depend on $\epsilon_*>0$ and are uniform in $\epsilon\in (0,\epsilon_*)$. Next, following the strategy developed in \cite{ZH98,CF20}, we introduce a family of  parametrized curves of the form
\bqq
\Gamma_p:=\left\{ \Re(\nu)-\frac{\beta^*}{c_*^2} \Re(\nu)^2+ \frac{\beta_*}{c_*^2}\Im(\nu)^2= \psi(p) ~|~ -\eta  \leq \Re(\nu) \leq p \right\},
\label{contourGp}
\eqq
with $\psi(p):=p-\frac{\beta^*}{c_*^2}p^2$. We readily note that the curve $\Gamma_p$ intersects the real axis at $\nu=p$. We now explain how we select $\eta>0$ and $p>-\eta$ in the above definition of $\Gamma_p$.

\begin{figure}[t!]
  \centering
  \includegraphics[width=.45\textwidth]{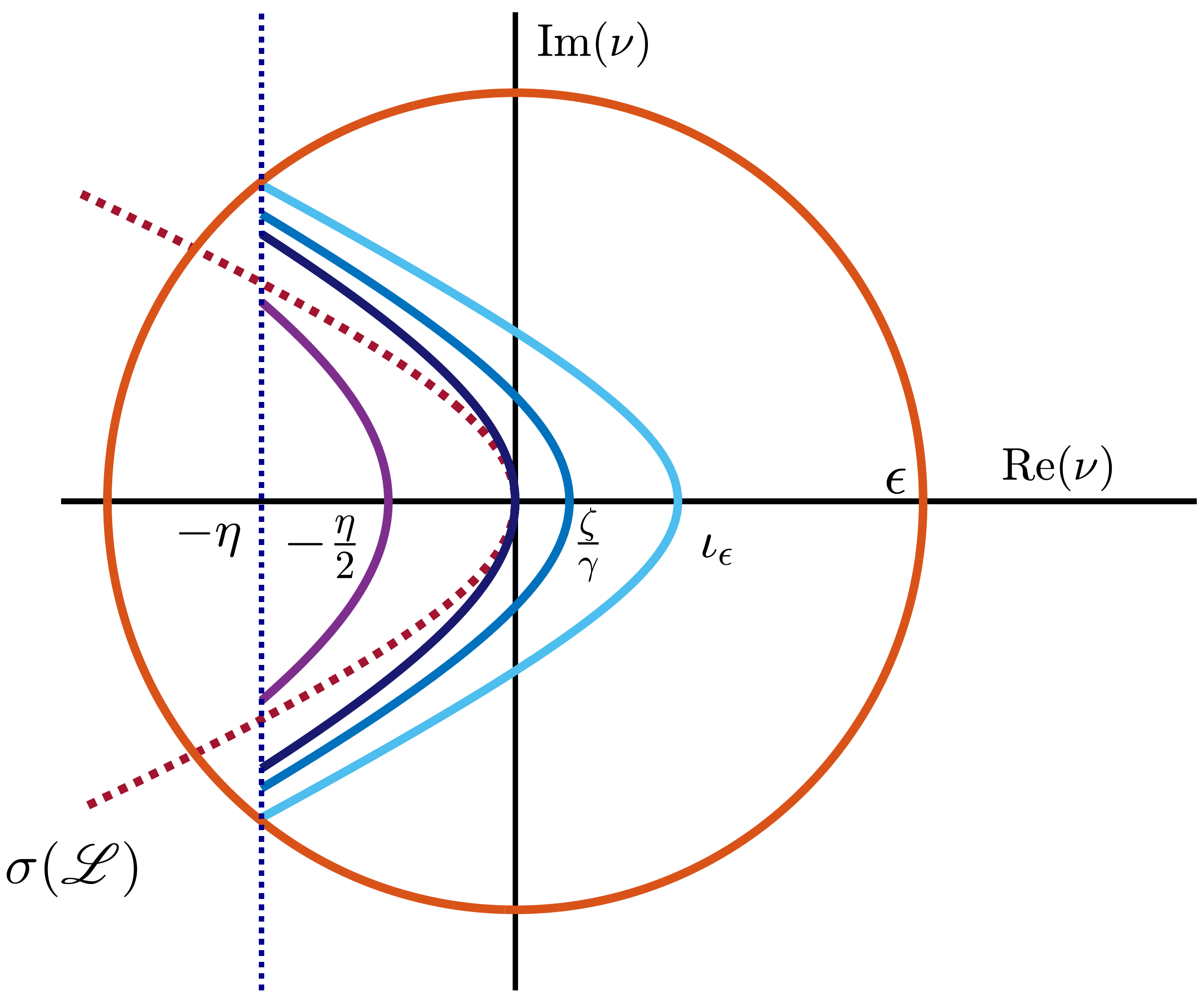}
  \caption{Illustration of the geometry of the family of parametrized curved $\Gamma_p$ within the ball $B_\epsilon(0)$ for different values of $p \in 
  [-\eta/2,\iota_\epsilon]$, adapted from \cite{CF20}. The extremal curves are given when $p=-\eta/2$ to the left (magenta curve) and when $p=\iota_\epsilon$ to the right (light 
  blue curve) where $0<\iota_\epsilon<\epsilon$ is precisely defined such that $\Gamma_p$ with $p=\iota_\epsilon$ intersects the line $-\eta 
  +\mathbf{i} \R$ on the boundary of $B_\epsilon(0)$. The dashed dark red curve represents the spectrum $\sigma(\mathscr{L})$.}
  \label{fig:CurveGammap}
\end{figure}

First,  for each $\epsilon \in (0,\epsilon_*)$, we denote by $\nu_\epsilon, \overline{\nu_\epsilon} \in \partial B_\epsilon(0) \cap \sigma(\mathscr{L})$ and let $\Re(\nu_\epsilon)=-\eta_\epsilon<0$. Then, we fix $\eta>0$ such that the curve $\Gamma_p$ with $p=0$ intersects the line $-\eta+\mathbf{i}\R$ inside the open ball $B_\epsilon(0)$ and readily note that $\eta\in(0,\eta_\epsilon)$. And next, we denote $\iota_\epsilon\in(0,\epsilon)$ the unique real number such that $\Gamma_{p}$ with $p=\iota_\epsilon$ intersects the line $-\eta+\mathbf{i}\R$ precisely on the boundary of $B_\epsilon(0)$ with $\eta$ fixed previously. Finally, we also introduce 
\bqs
\zeta := \frac{j-c_*t}{2t} \quad \text{ and } \quad \gamma :=\frac{j}{t}\frac{\beta^*}{c_*^2}>0,
\eqs
and the specific value of $p$ in the definition of $\Gamma_p$ is fixed depending on the ratio $\zeta/\gamma$ as follows 
\bqs
p:=\left\{
\begin{array}{lc}
\frac{\zeta}{\gamma}, & \text{ if } -\frac{\eta}{2}\leq \frac{\zeta}{\gamma} \leq \iota_\epsilon,\\
\iota_\epsilon , & \text{ if }  \frac{\zeta}{\gamma} > \iota_\epsilon,\\
-\frac{\eta}{2} , & \text{ if } \frac{\zeta}{\gamma} < -\frac{\eta}{2}.
\end{array}
\right.
\eqs
The geometry of the family of parametrized curves $\Gamma_p$ is illustrated in Figure~\ref{fig:CurveGammap} for different values of $p$. Note that with our careful choice of parametrization, we have that $\Gamma_p$ with $p=0$ (dark blue curve in Figure~\ref{fig:CurveGammap}) lies to the right of the spectral curve with tangency at the origin. 

Let us remark that the motivation for introducing such quantities comes from our above estimate on the Green's function. Indeed, for all $\nu\in\Gamma_p\subset B_\epsilon(0)$ and any $j\geq1$, we have
\begin{eqnarray}
j \Re(\varpi(\nu))&\leq j\left(-\frac{1}{c_*}\Re(\nu)+\frac{\beta^*}{c_*^3}\Re(\nu)^2-\frac{\beta_*}{c_*^3}\Im(\nu)^2\right)=j\left(-\frac{\psi(p)}{c_*}\right) \nonumber \\
&=\frac{j}{c_*}\left(-p+\frac{\beta^*}{c_*^2}p^2 \right)=-t p+\frac{t}{c_*}\left(-2\zeta p+\gamma p^2 \right) \nonumber \\
&=-t p-\frac{t\zeta^2}{c_*\gamma}+\frac{t}{c_*}\gamma\left(p-\frac{\zeta}{\gamma} \right)^2.
\label{EstimatejRenu}
\end{eqnarray}
Our precise choice of $p$, which depends on the ratio $\zeta/\gamma$, will always allow us to control the above terms. Furthermore, since we consider here the range $(c_*-\theta)t \leq j \leq Lt$, we readily have that $-\frac{\theta}{2}\leq\zeta\leq \frac{L-c_*}{2}$, and our generalized Gaussian estimate will come from those $\zeta\approx 0$.

Before proceeding with the analysis, we make two final claims that will be useful in the course of the proof of Lemma~\ref{lemgengauss}. 
\begin{enumerate}
\item We note that for each $\nu\in\Gamma_p$, with $\Gamma_p$ defined in \eqref{contourGp}, there exists some constant $c_\star>0$ such that
\bqq
\Re(\nu)\leq p-c_\star \Im(\nu)^2.
\label{IneqRenu}
\eqq
Indeed, since the function $\psi$, defined as $\psi(p)=p-\frac{\beta^*}{c_*^2}p^2$, satisfies $\psi'(0)=1$, there exists $c_0>0$ such that one has $\psi'(p)\leq c_0$ for each $p\in[-\eta,\epsilon]$. This implies that for each $\nu\in\Gamma_p$ we have 
\bqs
-\frac{\beta_*}{c_*^2}\mathrm{Im}(\nu)^2=\psi(\mathrm{Re}(\nu))-\psi(p)=-\int_{\mathrm{Re}(\nu)}^p\psi'(t)\mathrm{d} t \geq c_0(\mathrm{Re}(\nu)-p),
\eqs
which gives \eqref{IneqRenu} with  $c_\star=\frac{\beta_*}{c_*^2c_0}$.
\item Up to further reducing the size of $\epsilon_*>0$, a direct application of the implicit function theorem demonstrates that there exists a smooth function $\Phi:(-\epsilon_*,\epsilon_*)\times(-\epsilon_*,\epsilon_*)\rightarrow\R$ and some $C>0$ such that for any $\epsilon\in(0,\epsilon_*)$ and $p\in(-\epsilon,\epsilon)$, the curve $\Gamma_p$ can be parametrized as
\bqs
\Gamma_p=\left\{\nu\in B_\epsilon(0)~|~ \Re(\nu)=\Phi(\Im(\nu),p)\right\},
\eqs 
with 
\bqq
\Re(\nu)=p-\frac{\beta^*}{c_*^2}\Im(\nu)^2+O\left( |\Im(\nu)|^3+|p|^3\right),
\label{TaylorRenu}
\eqq
together with 
\bqs
\left|\frac{\partial \Phi(\Im(\nu),p)}{\partial \Im(\nu)}\right|\leq C, \quad \text{ for each } \quad |\Im(\nu)|\leq \epsilon \text{ and } |p|\leq \epsilon.
\eqs
\end{enumerate}

\begin{figure}[t!]
  \centering
  \includegraphics[width=.45\textwidth]{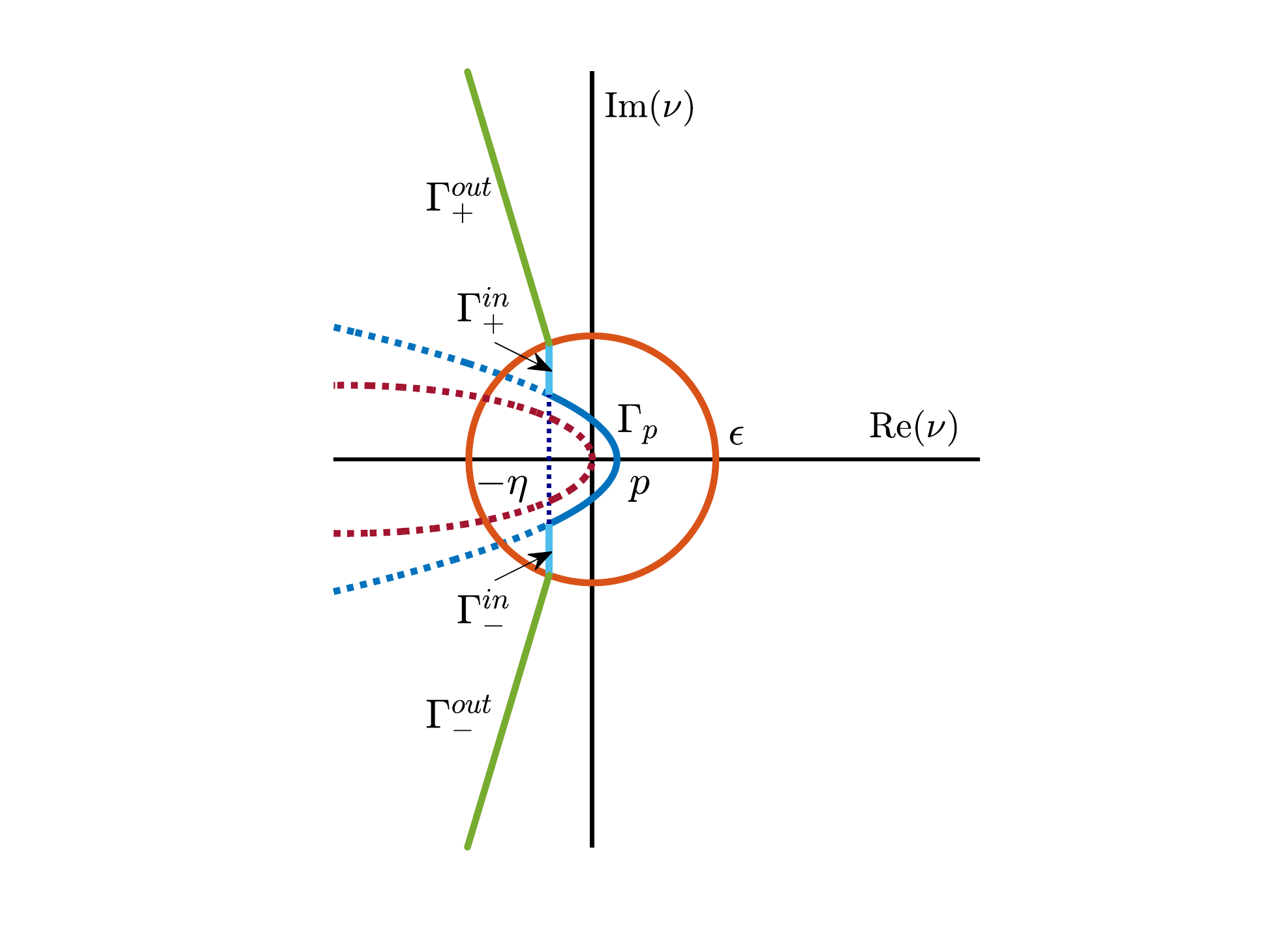}
  \caption{Illustration of the contour used in the case $ -\frac{\eta}{2}\leq \frac{\zeta}{\gamma} \leq \iota_\epsilon$ when $(c_*-\theta)\leq j \leq Lt$ with $t\geq1$. The contour is composed of $\Gamma_-^{out}\cup\Gamma_-^{in}\cup\Gamma_p\cup\Gamma_+^{in}\cup\Gamma_+^{out}$. The contours 
  $\Gamma_\pm^{in}$ are the portions of the line $-\eta+\mbi \R$ which lie inside $B_\epsilon(0)$ (light blue curves) while the contours $\Gamma_\pm^{out}$ are half lines lying in the resolvent set $\mathscr{U}$ with $|\Im(\nu)|\geq \sqrt{\epsilon^2-\eta^2}$ (green curves).  Finally, $\Gamma_p$ is defined in \eqref{contourGp} and intersects the real axis at $p$ (blue curve).}
  \label{fig:contour}
\end{figure}

We are now ready to prove Lemma~\ref{lemgengauss}.

\begin{proof}[Proof of Lemma~\ref{lemgengauss}.]
Throughout, we have $(c_*-\theta)\leq j \leq Lt$ with $t\geq 1$ and for each $\epsilon\in(0,\epsilon_*)$ the constant $\eta\in(0,\eta_\epsilon)$ has been fixed as explained above. We divide the analysis into three cases depending on the ratio $\zeta/\gamma$.
\paragraph{Case: $-\frac{\eta}{2}\leq \frac{\zeta}{\gamma} \leq \iota_\epsilon$.} We will consider a contour $\Gamma$ which is the union of three contours given by $\Gamma_p$ with $p=\zeta/\gamma$, together with $\Gamma^{in}$ and $\Gamma^{out}$ which are defined as follows. We refer to Figure~\ref{fig:contour} for a geometrical illustration. The contour $\Gamma^{in}$ is composed of the two portions of the segment $\Re(\nu)=-\eta$ which lie inside the ball $B_\epsilon(0)$, that is
\bqs
\Gamma^{in}=\Gamma^{in}_-\cup\Gamma^{in}_+=\left\{ -\eta + \mbi \xi ~|~  -\sqrt{\epsilon^2-\eta^2}\leq \xi \leq - \xi_* \right\}\cup\left\{ -\eta + \mbi \xi ~|~ \xi_*\leq \xi \leq \sqrt{\epsilon^2-\eta^2}\right\},
\eqs
with 
\bqs
\xi_*:=\sqrt{\frac{c_*^2}{\beta_*}\left(\psi(p)-\psi(-\eta)\right)}>0.
\eqs
On the other hand, the contour $\Gamma^{out}$ is taken of the form
\bqs
\Gamma^{out}=\Gamma^{out}_-\cup\Gamma^{out}_+=\left\{-\delta_0+\delta_1\xi+\mbi \xi ~|~ \xi \leq -\sqrt{\epsilon^2-\eta^2}\right\}\cup\left\{-\delta_0-\delta_1\xi+\mbi \xi ~|~ \xi \geq \sqrt{\epsilon^2-\eta^2}\right\},
\eqs
for two constants $\delta_0>0$ and $\delta_1>0$ chosen such that $\Gamma^{out}\subset \mathscr{U}$ and $\Gamma^{out}_\pm$ intersect $\Gamma^{in}_\pm$ exactly on $\partial B_\epsilon(0)$. The later condition implies that $\delta_0=\eta-\delta_1\sqrt{\epsilon^2-\eta^2}$, and then one can take $0<\delta_1<\frac{\eta}{\sqrt{\epsilon^2-\eta^2}}$ as small as required such that both $\delta_0>0$ and $\Gamma^{out}\subset \mathscr{U}$ are verified. As a consequence, we have that
\bqs
\mathscr{G}_j(t)=\frac{1}{2\pi \mbi}\int_{\Gamma_{\zeta/\gamma}} e^{\nu t} \bG_j(\nu)\md \nu +\frac{1}{2\pi \mbi}\int_{\Gamma^{in}} e^{\nu t} \bG_j(\nu)\md \nu+\frac{1}{2\pi \mbi}\int_{\Gamma^{out}} e^{\nu t} \bG_j(\nu)\md \nu,
\eqs
and we have to estimate the above three integrals. We start with the first one, which is the one that will produce the desired Gaussian estimate. For each $\nu \in \Gamma_{p}\subset B_\epsilon(0)$, we know from \eqref{EstimatejRenu}-\eqref{IneqRenu} that
\begin{align*}
t\Re(\nu)+j \Re(\varpi(\nu)) &\leq t\left(\Re(\nu) -p\right)-\frac{t\zeta^2}{c_*\gamma}+\frac{t}{c_*}\gamma\left(p-\frac{\zeta}{\gamma} \right)^2\\
& \leq -c_\star \Im(\nu)^2 t -\frac{t\zeta^2}{c_*\gamma}+\frac{t}{c_*}\gamma\left(p-\frac{\zeta}{\gamma} \right)^2
\end{align*}
such that, with $p=\zeta/\gamma$, we obtain that
\bqs
t\Re(\nu)+j \Re(\varpi(\nu))\leq -c_\star \Im(\nu)^2t -\frac{\zeta^2}{\gamma}\frac{t}{c_*}.
\eqs
As a consequence, we have
\begin{align*}
\left|\frac{1}{2\pi \mbi}\int_{\Gamma_{\zeta/\gamma}} e^{\nu t} \bG_j(\nu)\md \nu\right| & \lesssim \int_{\Gamma_{\zeta/\gamma}} e^{t\Re(\nu)+j \Re(\varpi(\nu)) } |\md \nu|\\
&\lesssim e^{-\frac{\zeta^2}{\gamma}\frac{t}{c_*}} \int_{\Gamma_{\zeta/\gamma}} e^{-c_\star \Im(\nu)^2t} |\md \nu|\\
&\lesssim e^{-\frac{\zeta^2}{\gamma}\frac{t}{c_*}} \int_{-\xi_*}^{\xi_*} e^{-c_\star \xi^2 t}\md \xi\\
&\lesssim \frac{e^{-\frac{\zeta^2}{\gamma}\frac{t}{c_*}}}{\sqrt{t}}.
\end{align*}

Finally, we remark that
\bqs
-\frac{\zeta^2}{\gamma}\frac{t}{c_*}=- \frac{(j-c_*t)^2}{4t}\frac{1}{\gamma c_*} \leq - \frac{(j-c_*t)^2}{4t}\frac{c_*}{L\beta^*},
\eqs
since $(c_*-\theta)\frac{\beta^*}{c_*^2}\leq \gamma \leq \frac{L\beta^*}{c_*^2}$ as $(c_*-\theta)t\leq j \leq Lt$, and we have obtained the desired Gaussian estimate. 

Next, for all $\nu=-\eta+\mbi \xi \in\Gamma^{in}\subset B_\epsilon(0)$ and $j\geq1$ we have that
\begin{align*}
j \Re(\varpi(\nu)) & \leq j\left(-\frac{1}{c_*}\Re(\nu)+\frac{\beta^*}{c_*^3}\Re(\nu)^2-\frac{\beta_*}{c_*^3}\Im(\nu)^2\right)=j\left(-\frac{\psi(-\eta)}{c_*}-\frac{\beta_*}{c_*^3}\xi^2\right)\\
&=j\left( -\frac{\psi(p)}{c_*}-\frac{\beta_*}{c_*^3}\underbrace{\left(\xi^2-\xi_*^2\right)}_{\geq 0}\right) \leq j\left( -\frac{p}{c_*}+\frac{\beta^*}{c_*^3}p^2\right),
\end{align*}
since $\psi(-\eta)+\frac{\beta_*}{c_*^2}\xi_*^2=\psi(p)$ by definition of $\xi_*>0$. Thus, we have
\begin{align*}
t \Re(\nu)+j \Re(\varpi(\nu))&\leq - \eta t +j\left(-\frac{p}{c_*}+\frac{\beta^*}{c_*^3}p^2\right)\\
&=\frac{t}{c_*}\left[ -(\eta+p) c_* -2\zeta p+ \gamma p^2 \right]\\
&\leq -\frac{\eta}{2}t- \frac{t}{c_*}\frac{\zeta^2}{\gamma},
\end{align*}
since $-\eta/2\leq \zeta/\gamma=p$ in this case. As a consequence, we obtain an estimate of the form
\bqs
\left| \frac{1}{2\pi \mbi}\int_{\Gamma^{in}} e^{\nu t} \bG_j(\nu)\md \nu \right| \lesssim e^{-\frac{\eta}{2}t- \frac{t}{c_*}\frac{\zeta^2}{\gamma}}.
\eqs
This term can be subsumed into the previous Gaussian estimate. 

For the remaining contribution on $\Gamma^{out}$, we further split it into two parts $\Gamma^{out}=\Gamma^{out}_1\cup\Gamma^{out}_2$ where $\Gamma^{out}_1:=\Gamma^{out}\cap B_R(0)$ and $\Gamma_2^{out}$ is the complementary part. It follows from Lemma~\ref{leminf} and Lemma~\ref{lemint} that
\begin{align*}
\left| \frac{1}{2\pi \mbi}\int_{\Gamma^{out}_1} e^{\nu t} \bG_j(\nu)\md \nu \right| \lesssim e^{-\delta_0 t},\\
\left| \frac{1}{2\pi \mbi}\int_{\Gamma^{out}_2} e^{\nu t} \bG_j(\nu)\md \nu \right| \lesssim e^{-\delta_0 t-\kappa j} \underbrace{\int_R^{+\infty}e^{-\delta_1 \xi t}\md \xi}_{\leq \frac{C}{t}}.
\end{align*}
It only remains to check that a temporal exponential can be subsumed into a Gaussian estimate. Indeed, using the definition of $\zeta$ and $\gamma$ together with the fact we supposed that $-\frac{\eta}{2}\leq \frac{\zeta}{\gamma} \leq \iota_\epsilon$, we can always find a constant $\kappa_0>0$ such that
\bqs
\kappa_0 \frac{(j-c_*t)^2}{t}\leq t.
\eqs

\paragraph{Case: $\frac{\zeta}{\gamma} > \iota_\epsilon$.} The contour $\Gamma$ is decomposed into $\Gamma_p\cup\Gamma^{out}$ with
\bqs
\Gamma^{out}=\Gamma^{out}_-\cup\Gamma^{out}_+=\left\{-\delta_0+\delta_1\xi+\mbi \xi ~|~ \xi \leq -\sqrt{\epsilon^2-\eta^2}\right\}\cup\left\{-\delta_0-\delta_1\xi+\mbi \xi ~|~ \xi \geq \sqrt{\epsilon^2-\eta^2}\right\},
\eqs
for two constants $\delta_0>0$ and $\delta_1>0$ chosen as in the previous case. This time for $\Gamma_p$ we select $p=\iota_\epsilon$. Note that for each $\nu\in\Gamma_{\iota_\epsilon}\subset B_\epsilon(0)$, we have from \eqref{EstimatejRenu}-\eqref{IneqRenu}
\begin{align*}
t\Re(\nu)+j \Re(\varpi(\nu)) &\leq -c_\star \Im(\nu)^2 t +\frac{t}{c_*}\left(-2\zeta p+\gamma p^2 \right) \\
& = -c_\star \Im(\nu)^2 t +\frac{t}{c_*}\left(-2\zeta \iota_\epsilon +\gamma \iota_\epsilon^2 \right).
\end{align*}
But as $\frac{\zeta}{\gamma} > \iota_\epsilon$, we get in particular that $\zeta>\gamma \iota_\epsilon>0$ such that
\bqs
-2\zeta \iota_\epsilon +\gamma \iota_\epsilon^2 < - \gamma \iota_\epsilon^2<0,
\eqs
and thus
\bqs
\left|\frac{1}{2\pi \mbi}\int_{\Gamma_{\iota_\epsilon}} e^{\nu t} \bG_j(\nu)\md \nu\right| \lesssim \frac{e^{- \frac{t}{c_*} \gamma \iota_\epsilon^2}}{\sqrt{t}}.
\eqs
Now, since $\frac{\zeta}{\gamma} > \iota_\epsilon$, we get that $2\gamma\iota_\epsilon t<2\zeta t = j-c_*t \leq (L-c_*)t$, and thus there exists some constant $\kappa_1>0$ such that
\bqs
 \frac{(j-c_*t)^2}{t^2}\leq \kappa_1,
\eqs
which shows that the above exponential decaying in time bound can be subsumed into a Gaussian estimate. The estimate on $\Gamma^{out}$ is similar as in the previous case and we get
\bqs
\left| \frac{1}{2\pi \mbi}\int_{\Gamma^{out}} e^{\nu t} \bG_j(\nu)\md \nu \right| \leq \left| \frac{1}{2\pi \mbi}\int_{\Gamma^{out}_1} e^{\nu t} \bG_j(\nu)\md \nu \right| + \left| \frac{1}{2\pi \mbi}\int_{\Gamma^{out}_2} e^{\nu t} \bG_j(\nu)\md \nu \right|  \lesssim e^{-\delta_0 t},
\eqs
which can once again be subsumed into a Gaussian estimate.

\paragraph{Case: $\frac{\zeta}{\gamma} <-\frac{\eta}{2}$.} Once again we divide the contour $\Gamma$ into $\Gamma_p\cup \Gamma^{in}\cup \Gamma^{out}$ with $p=-\frac{\eta}{2}$, and $\Gamma^{in}$ and $\Gamma^{out}$ defined as previously. This time, for all $\nu\in\Gamma_{-\frac{\eta}{2}}\subset B_\epsilon(0)$, we have
\begin{align*}
t\Re(\nu)+j \Re(\varpi(\nu)) &\leq -c_\star \Im(\nu)^2 t +\frac{t}{c_*}\left(-2\zeta p+\gamma p^2 \right) \\
& \leq -c_\star \Im(\nu)^2 t +\frac{t}{c_*}\left(\eta \zeta  +\gamma\left(\frac{\eta}{2}\right)^2 \right).
\end{align*}
Note that this time $\zeta<-\eta \gamma/2$, and we have
\bqs
\eta \zeta  +\gamma\left(\frac{\eta}{2}\right)^2 < -\gamma \frac{\eta^2}{4},
\eqs
thus
\bqs
\left|\frac{1}{2\pi \mbi}\int_{\Gamma_{-\frac{\eta}{2}}} e^{\nu t} \bG_j(\nu)\md \nu\right| \lesssim  \frac{e^{- \frac{t}{c_*}\gamma \frac{\eta^2}{4}}}{\sqrt{t}}.
\eqs
And once again we can conclude by noticing that $-\theta t\leq j-c_*t\leq -\eta \gamma t$ due to $\zeta<-\eta \gamma/2$, and the above exponential decaying in time bound can be subsumed into a Gaussian estimate.

Next, for all $\nu\in \Gamma^{in}\subset B_\epsilon(0)$, we have
\begin{align*}
t \Re(\nu)+j \Re(\varpi(\nu))&\leq \frac{t}{c_*}\left[ -(\eta+p) c_* -2\zeta p+ \gamma p^2 \right]\\
&=\frac{t}{c_*}\left[ -\frac{\eta}{2}c_* +\zeta \eta + \gamma \left( \frac{\eta}{2}\right)^2 \right]\\
&< - \frac{\eta}{2}t- \frac{t}{c_*}\gamma \frac{\eta^2}{4},
\end{align*}
and we obtain an estimate of the form
\bqs
\left| \frac{1}{2\pi \mbi}\int_{\Gamma^{in}} e^{\nu t} \bG_j(\nu)\md \nu \right| \lesssim e^{- \frac{\eta}{2}t- \frac{t}{c_*}\gamma \frac{\eta^2}{4}}.
\eqs
This term can be subsumed into a Gaussian estimate. Finally, the contribution along $\Gamma^{out}$ can be handled as previously. This concludes the proof of the Lemma.
\end{proof}

As a conclusion, summarizing all the above lemmas, we have obtained the following intermediate result.

\begin{prop}[Pointwise estimates]\label{propptwestim}
There are constants $\theta_*>0$, $\beta_i>0$ for $i=1,2,3$ and $C>0$ such that the temporal Green's function $\mathscr{G}(t)$ satisfies the following pointwise estimates:
\begin{itemize}
\item for $|j-c_*t|>\theta_*t$ and $t\geq 1$, one has
\bqs
\left| \mathscr{G}_j(t) \right| \leq C\exp\left(-\beta_1 t -\beta_2|j|\right);
\eqs
\item for $|j-c_*t| \leq \theta_*t$ and  $t\geq1$, one has
\bqs
\left| \mathscr{G}_j(t) \right| \leq \frac{C}{\sqrt{t}}\exp\left(-\beta_3 \frac{|j-c_*t|^2}{t}\right).
\eqs
\end{itemize}
\end{prop}
The above proposition ensures in particular that there exists a constant $C>0$, such that
\bqs
\sqrt{t}\, \|\mathscr{G}(t)\|_{\ell^\infty(\Z)} +  \|\mathscr{G}(t)\|_{\ell^1(\Z)}\leq C,  \quad  t\geq 1.
\eqs
This in turn implies that any solution of the linear problem \eqref{KPPlin} starting from a compactly supported initial condition satisfies  
\begin{equation}
	\label{bdd-w}
	|w_j(t)| = \left| \sum_{\ell \in \Z} \mathscr{G}_{j-\ell}(t) w_\ell^0 \right|  \leq \frac{C}{\sqrt{t}} \, \| w^0\|_{\ell^1(\Z)}, \quad t \geq 1, \quad j\in\Z.
\end{equation}
It turns out that such an estimate will not be enough in the forthcoming analysis leading to the proof of Theorem~\ref{thmlog}. It is the purpose of the next section to refine our pointwise estimates  for $|j-c_*t| \leq \theta_*t$ and $t$ large enough.

\section{Refined pointwise estimates in the sub-linear regime}\label{secRefined}

From the study conducted in the previous section and Lemma~\ref{lemgengauss}, we expect that $\G_j(t)$ behaves like a Gaussian profile around $j\approx c_*t$ for all $j\in \Z\cap [(c_*-\theta_*)t,(c_*+\theta_*)t]$. It is actually possible to refine our analysis and prove that $\G_j(t)$ can be decomposed as a universal Gaussian profile plus some reminder term which can be bounded and also satisfies a Gaussian estimate. To do so, we work in the asymptotic regime where for any fixed $\vartheta>0$ and $\alpha\in(0,1)$ the time $t$ is large and $|j-c_*t|\leq \vartheta t^\alpha$. Before stating our main result, we introduce the normalized Gaussian profile:
\bqs
\mathcal{G}(x):=\frac{1}{\sqrt{2\pi}}e^{-\frac{x^2}{2}},\quad x\in\R,
\eqs
and we define the following odd cubic polynomial function $\mathcal{P}:\R\mapsto\R$ 
\bqq
\mathcal{P}(x):=\frac{c_*}{24\cosh(\lambda_*)^2}\left(-3x+x^3\right).
\label{equationp}
\eqq

\begin{prop}[Refined asymptotic]\label{proprefined}
For any $\vartheta>0$ and $\alpha\in(0,1)$, there is some $T_0>1$ such that for each $j\in\Z$ with $|j-c_*t|\leq \vartheta t^\alpha$ and $t\geq T_0$, one can decompose the temporal Green's function as follows
\bqq
\mathscr{G}_j(t)=\H_j(t)+\mathscr{R}_j(t),
\label{refinedasymp}
\eqq
where the principal part $\H_j(t)$ is defined as
\bqs
\H_j(t):= \left[\frac{1}{\sqrt{2\cosh(\lambda_*)t}}+\frac{1}{t}\mathcal{P}\left(\frac{j-c_*t}{\sqrt{2\cosh(\lambda_*)t}} \right)\right]\mathcal{G}\left(\frac{j-c_*t}{\sqrt{2\cosh(\lambda_*)t}} \right), 
\eqs
and with a Gaussian bound on the remainder term $\mathscr{R}_j(t)$
\bqs
\left| \mathscr{R}_j(t)  \right| \leq \frac{C}{t^{3/2}} \exp\left(-\beta \frac{(j-c_*t)^2}{t}\right),
\eqs
for some uniform constants $C>0$ and $\beta>0$.\end{prop}

\begin{figure}[t!]
  \centering
  \includegraphics[width=.45\textwidth]{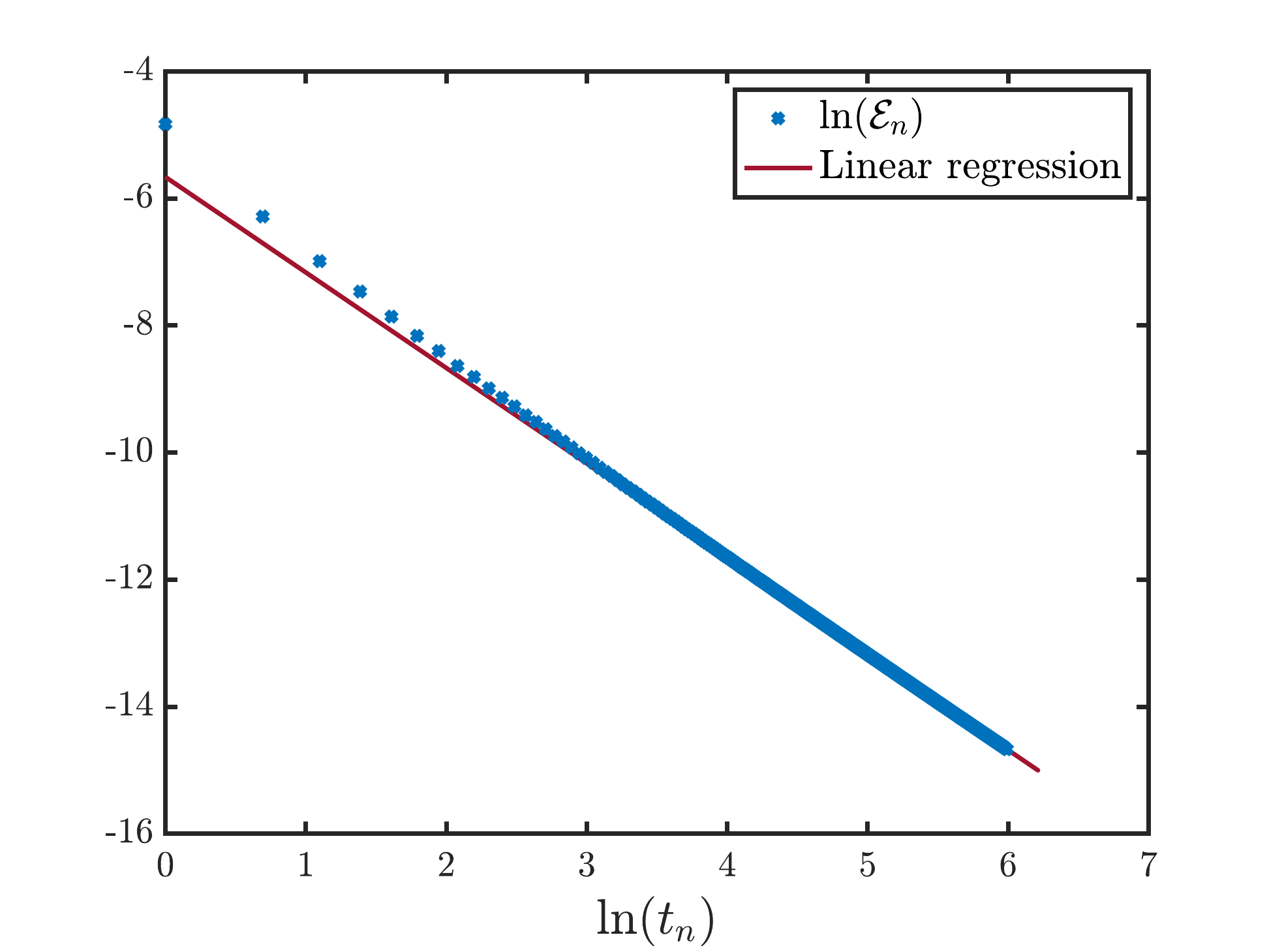}
\includegraphics[width=.45\textwidth]{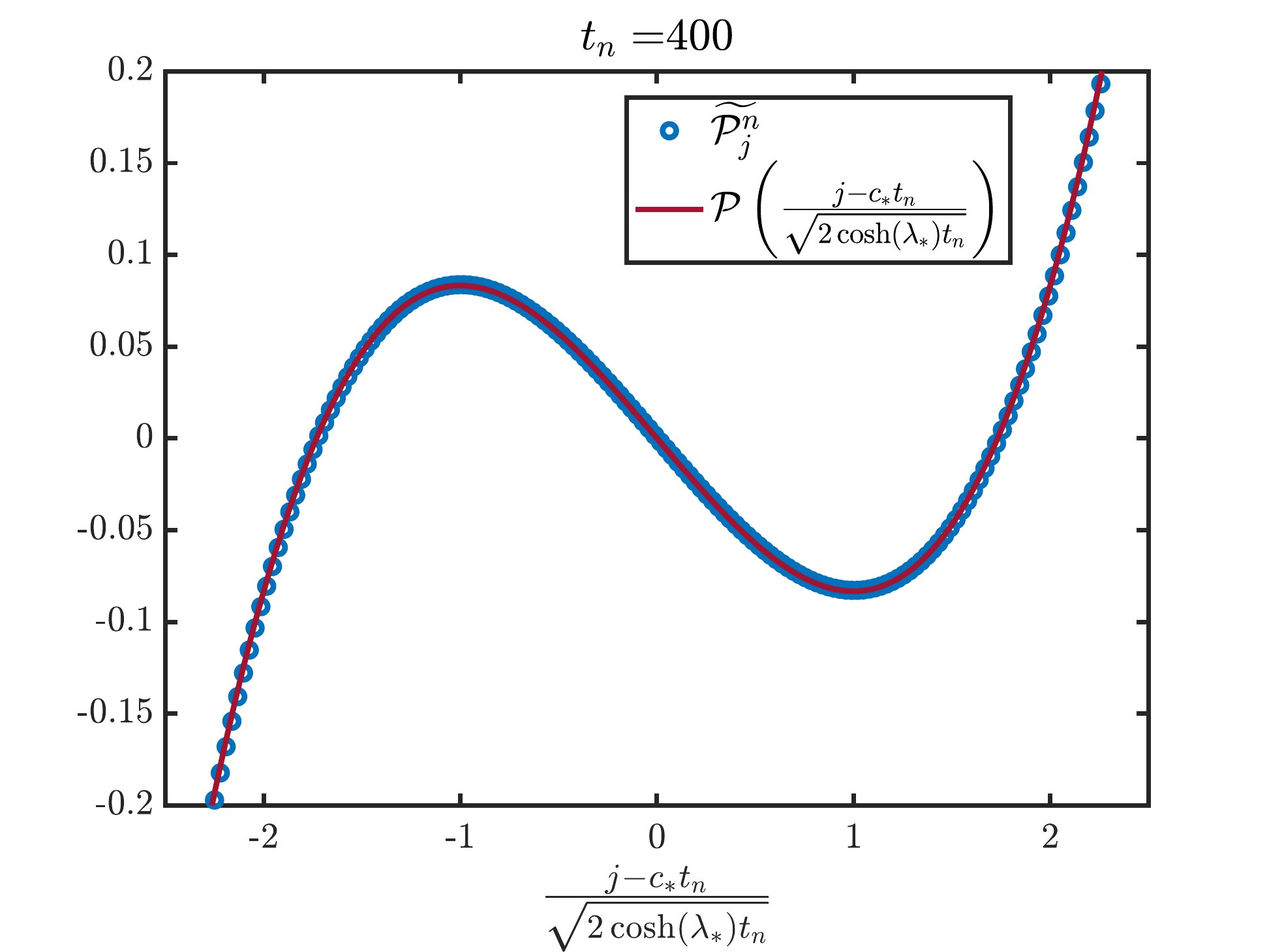}
\caption{Numerical illustrations of Proposition~\ref{proprefined}. Left. We plot $\ln(\mathcal{E}_n)$ (blue crosses), where $\mathcal{E}_n$ is defined in \eqref{definitionEn}, as a function of $\ln(t_n)$ for each $n=1,\cdots N$. We compare $\ln(\mathcal{E}_n)$ to its linear regression (dark red curve). The slope is found to be  $\simeq-1.5033$ which gives a $0.2\%$ error compared to the theoretical value $-\frac{3}{2}$. Right. We compare $\widetilde{\mathcal{P}_j^n}$ (blue circles), defined in \eqref{definitionPn}, to $\mathcal{P}\left(\frac{j-c_*t}{\sqrt{2\cosh(\lambda_*)t}} \right)$ (dark red curve) at time $t=400$. In the computations, we used $f'(0)=1$ giving $(c_*,\lambda_*)\simeq(2.0734,0.9071)$ together with $L=1000$, $\Delta t=0.01$ and $N=40001$.}
  \label{fig:numerics}
\end{figure}

We illustrate in Figure~\ref{fig:numerics} our main result of Proposition~\ref{proprefined}. We numerically solved \eqref{KPPlin} starting from the initial sequence $\boldsymbol{\delta}$ to approximate the temporal Green's functions $(\mathscr{G}_j(t))_{j\in\Z}$ for $t\geq1$. We used a Runge-Kutta method of order 4 to discretize in time the evolution equation \eqref{KPPlin} with fixed time step $\Delta t>0$, and the spatial domain was taken to $\llbracket-L,L\rrbracket$\footnote{Here, we used the notation $\llbracket A,B\rrbracket$ to denote the set of all integers between $A\in\Z$ and $B\in\Z$ with $A\leq B$.} for some given integer $L\geq 1$. And we denote by $\widetilde{\mathscr{G}_j^n}$ the approximation of $\mathscr{G}_j(t)$ at time $t=t_n=n\Delta t$ for each $j \in \llbracket-L,L\rrbracket$ and $n=1,\cdots N$. Let $\mathcal{E}_n$ be defined as
\bqq
\mathcal{E}_n:=\sup_{j \in \llbracket-L,L\rrbracket} \left|\widetilde{\mathscr{G}_j^n}-\H_j(t_n)\right|, \quad n=1,\cdots N.
\label{definitionEn}
\eqq
Then $\mathcal{E}_n$ represents a numerical approximation of $\sup_{j\in\Z}\left|\mathscr{R}_j(t_n)\right|$ which is of the order $O\left( t_n^{-3/2}\right)$ by Proposition~\ref{proprefined}. In the left panel of Figure~\ref{fig:numerics}, we show that $\ln( \mathcal{E}_n)$ is a linear function of $\ln(t_n)$ whose slope is found to be  $\simeq-1.5033$ which compares very well with the theoretical value $-\frac{3}{2}$. In the right panel, we recover the universal nature of the profile $\mathcal{P}$. More precisely, we compare $\widetilde{\mathcal{P}_j^n}$, defined as
\bqq
\widetilde{\mathcal{P}_j^n}:=t_n \left[\widetilde{\mathscr{G}_j^n}-\frac{1}{\sqrt{2\cosh(\lambda_*)t_n}}\mathcal{G}\left(\frac{j-c_*t_n}{\sqrt{2\cosh(\lambda_*)t_n}} \right)\right] \mathcal{G}\left(\frac{j-c_*t_n}{\sqrt{2\cosh(\lambda_*)t_n}} \right)^{-1}, \quad j \in \llbracket-L,L\rrbracket,
\label{definitionPn}
\eqq
to $\mathcal{P}\left(\frac{j-c_*t_n}{\sqrt{2\cosh(\lambda_*)t_n}}\right)$ as a function of $\frac{j-c_*t_n}{\sqrt{2\cosh(\lambda_*)t_n}}$ at fixed $t_n$ and observe a very good match confirming numerically the universal nature of our formula \eqref{equationp} for $\mathcal{P}$.

\subsection{Proof of Proposition~\ref{proprefined}} 

We first explain the strategy of the proof. We are going to write $\G_j(t)$ as
\bqs
\G_j(t)=\frac{1}{2\pi \mbi}\int_{\Gamma_{in}} e^{\nu t} \bG_j(\nu)\md \nu+\frac{1}{2\pi \mbi}\int_{\Gamma_{out}} e^{\nu t} \bG_j(\nu)\md \nu,
\eqs
where $\Gamma_{in}$ is a contour within the ball $B_\epsilon(0)$ and $\Gamma_{out}\subset \left\{\nu\in\mathbb{C}~|~\Re(\nu)\leq -\eta\right\}$ is some contour of the form $\left\{-\delta_0 -\delta_1|\xi|+\mbi \xi ~|~ \xi\in\R, ~ |\xi|\geq \sqrt{\epsilon^2-\eta^2}\right\}$ for well-chosen $\delta_{0,1}>0$ and where $\eta\in(0,\epsilon)$ can be chosen as in the proof of Lemma~\ref{lemgengauss}. Throughout, we assume $\Gamma_{in}\cap \Gamma_{out}=\left\{\nu_0,\overline{\nu_0}\right\}$ with $|\nu_0|=\epsilon$ and $\Re(\nu_0)=-\eta$. As a consequence, for $t$ large enough, we have $|j-c_*t|\leq \vartheta t^\alpha \leq \theta_*t$ and the proof of Lemma~\ref{lemgengauss} gives
\bqs
\left|\frac{1}{2\pi \mbi}\int_{\Gamma_{out}} e^{\nu t} \bG_j(\nu)\md \nu\right| \leq \frac{C}{t}e^{-\frac{\eta_\epsilon}{2}t}\leq \frac{\widetilde{C}}{t^{3/2}} \exp\left(-\beta   \frac{(j-c_*t)^2}{t}\right).
\eqs
As a consequence, we are led to prove that
\bqs
\left| \frac{1}{2\pi \mbi}\int_{\Gamma_{in}} e^{\nu t} \bG_j(\nu)\md \nu -\H_j(t) \right| \leq \frac{C}{t^{3/2}} \exp\left(-\beta \frac{(j-c_*t)^2}{t}\right),
\eqs
whenever $|j-c_*t|\leq \vartheta t^\alpha $ and $t$ large enough. For all $\nu\in  B_\epsilon(0)$, we recall from \eqref{defGjnu0} that for $j\geq1$ one has
\bqs
\bG_j(\nu)=\Psi(\nu)e^{\varpi(\nu)j}, 
\eqs
where $\Psi$ is defined in \eqref{defPsinu} and is holomorphic on $B_\epsilon(0)$ and $\Psi(0)=\frac{1}{c_*}>0$. In fact, simple computation gives
\bqs
\Psi(\nu)=-e^{\lambda_*}\frac{\langle \pi^0(\nu)\mathbf{e},\mathbf{e}\rangle}{e^{\varpi(\nu)}}=\frac{e^{\lambda_*}}{\rho_+(\nu)-\rho_-(\nu)}=\frac{1}{\sqrt{\left(\nu+e^{\lambda^*}+e^{-\lambda_*}\right)^2-4}},
\eqs
such that
\bqs
\Psi(\nu)=\frac{1}{c_*}-\frac{2\cosh(\lambda_*)}{c_*^3}\nu+\nu^2\widetilde{\Psi}(\nu), \quad \forall \, \nu\in B_\epsilon(0), 
\eqs
and $\widetilde{\Psi}$ is holomorphic on $B_\epsilon(0)$ with $|\widetilde{\Psi}(\nu)|\leq C$ for each $\nu\in B_\epsilon(0)$. Next, we are going to need an asymptotic expansion of $\varpi(\nu)$ up to quartic order, that is,
\bqq
\label{ExpansionvarpinuO4}
\varpi(\nu)=- \frac{\nu}{c_*}+\frac{\cosh(\lambda_*)}{c_*^3}\nu^2-\frac{\Lambda_*}{c_*^5}\nu^3+\nu^4\Omega(\nu), \quad \forall \, \nu \in B_\epsilon(0),
\eqq
where $\Omega$ is holomorphic on $B_\epsilon(0)$ with $|\Omega(\nu)|\leq C$ for each $\nu\in B_\epsilon(0)$ and
\bqq
\label{LambdaStar}
\Lambda_*:=2+\frac{c_*^2}{3}>0.
\eqq
For future reference, we also denote by $\varphi(\nu)$ the principal part of $\varpi(\nu)$, that is
\bqs
\varphi(\nu):=- \frac{\nu}{c_*}+\frac{\cosh(\lambda_*)}{c_*^3}\nu^2, \quad \forall \, \nu \in B_\epsilon(0).
\eqs
Finally, we remark that $\beta^*,\beta_*>0$ in \eqref{boundRevarpi} can always be taken larger and smaller respectively such that we can also ensure that 
\bqq
\Re(\varpi(\nu))+\frac{|\Lambda_*|}{c_*^5}|\nu|^3+|\nu|^4|\Omega(\nu)| \leq -\frac{\Re(\nu)}{c_*}+\frac{\beta^*}{c_*^3} \Re(\nu)^2-\frac{\beta_*}{c_*^3}\Im(\nu)^2, \quad \nu \in B_\epsilon(0).
\label{boundRevarpiImp}
\eqq

With all these notations in hands, we will decompose $\bG_j(\nu)$ as follows
\begin{align*}
\bG_j(\nu)&=\left(\frac{1}{c_*}-\frac{2\cosh(\lambda_*)}{c_*^3}\nu -j \frac{\Lambda_*}{c_*^6}\nu^3\right)e^{\varphi(\nu)j}+\nu^2\widetilde{\Psi}(\nu) e^{\varpi(\nu)j}-2\frac{\cosh(\lambda_*)}{c_*^3}\nu\left(e^{\varpi(\nu)j}-e^{\varphi(\nu)j}\right)\\
&~~~+\frac{1}{c_*}\left[e^{-j\frac{\Lambda_*}{c_*^5}\nu^3+j\nu^4\Omega(\nu)}-\left(1-j\frac{\Lambda_*}{c_*^5}\nu^3+j\nu^4\Omega(\nu)\right)\right]e^{\varphi(\nu)j}+\frac{j}{c_*}\nu^4\Omega(\nu)e^{\varphi(\nu)j},
\end{align*}
valid for each $\nu\in B_\epsilon(0)$ and $j\geq1$.

We now introduce the following integrals which will contribute as error terms
\begin{align*}
\mathscr{E}_j^1(t)&:=\frac{1}{2\pi \mbi}\int_{\Gamma_{in}} e^{\nu t} \nu^2\widetilde{\Psi}(\nu) e^{\varpi(\nu)j} \md \nu,\\
\mathscr{E}_j^2(t)&:=\frac{1}{2\pi \mbi}\int_{\Gamma_{in}} e^{\nu t} \nu\left(e^{\varpi(\nu)j}-e^{\varphi(\nu)j}\right) \md \nu,\\
\mathscr{E}_j^3(t)&:=\frac{1}{2\pi \mbi}\int_{\Gamma_{in}} e^{\nu t} \left[e^{-j\frac{\Lambda_*}{c_*^5}\nu^3+j\nu^4\Omega(\nu)}-\left(1-j\frac{\Lambda_*}{c_*^5}\nu^3+j\nu^4\Omega(\nu)\right)\right]e^{\varphi(\nu)j}\md \nu,\\
\mathscr{E}_j^4(t)&:=\frac{1}{2\pi \mbi}\int_{\Gamma_{in}} e^{\nu t} \nu^4\Omega(\nu)e^{\varphi(\nu)j} \md \nu.
\end{align*}

More precisely, we have the following lemma.

\begin{lem}
For any $\vartheta>0$ and $\alpha\in(0,1)$, there is some $T_0>1$ such that for each $j\in\Z$ with $|j-c_*t|\leq \vartheta t^\alpha$ and $t\geq T_0$ one has
\bqs
\left| \mathscr{E}_j^k(t) \right| \leq \frac{C}{t^{3/2}} \exp\left(-\beta \frac{(j-c_*t)^2}{t}\right), \quad k=1,\cdots, 4 \,,
\eqs
for some uniform constants $C>0$ and $\beta>0$.
\end{lem}

\begin{proof}
Recalling our notation 
\bqs
\zeta=\frac{j-c_*t}{2t}, \quad \text{ and } \quad \gamma=\frac{j}{t}\frac{\beta^*}{c_*^2},
\eqs
we can always ensure that $\zeta/\gamma\in(-\eta/2,\iota_\epsilon)$ since we consider $t$ large enough and $|j-c_*t|\leq \vartheta t^{\alpha}$ with $\alpha\in(0,1)$. Here $\iota_\epsilon>0$ is the unique real number such that $\Gamma_p$ defined in \eqref{contourGp} with $p=\iota_\epsilon$ intersects the line $-\eta+\mathbf{i}\R$ on the boundary of $B_\epsilon(0)$ with $\eta \in(0,\eta_\epsilon)$ and where $\eta_\epsilon>0$ is defined as in the proof of Lemma~\ref{lemgengauss}. Then, in each $\mathscr{E}_j^k(t)$ with $k=1,\cdots,4$, we decompose $\Gamma_{in}$ as $\Gamma_{in}=\Gamma_{\zeta/\gamma}\cup\Gamma_1$ with
\bqs
\Gamma_1=\left\{ -\eta + \mbi \xi ~|~  -\sqrt{\epsilon^2-\eta^2}\leq \xi \leq - \xi_* \text{ and } \xi_*\leq \xi \leq \sqrt{\epsilon^2-\eta^2}\right\},
\eqs
and 
\bqs
\xi_*=\sqrt{\frac{c_*^2}{\beta_*}\left(\psi(\zeta/\gamma)-\psi(-\eta)\right)}>0.
\eqs

\paragraph{Study of $\mathscr{E}_j^1(t)$.} First, we can bound $\mathscr{E}_j^1(t)$ by
\begin{align*}
\left|\mathscr{E}_j^1(t)\right|& \lesssim \int_{\Gamma_{\zeta/\gamma}}e^{\Re(\nu)t}|\nu|^2e^{\Re(\varpi(\nu))j}|\md \nu|+\int_{\Gamma_1}e^{\Re(\nu)t}|\nu|^2e^{\Re(\varpi(\nu))j}|\md \nu|\\
&\lesssim \int_{\Gamma_{\zeta/\gamma}}e^{\Re(\nu)t}\left(\frac{\zeta^2}{\gamma^2}+|\Im(\nu)|^2\right)e^{\Re(\varpi(\nu))j}|\md \nu|+\int_{\Gamma_{1}}e^{\Re(\nu)t}\left(\eta^2+|\Im(\nu)|^2\right)e^{\Re(\varpi(\nu))j}|\md \nu|.
\end{align*}
Using similar computations as in the proof of Lemma~\ref{lemgengauss}, we find
\begin{align*}
\frac{\zeta^2}{\gamma^2} \int_{\Gamma_{\zeta/\gamma}}e^{\Re(\nu)t}e^{\Re(\varpi(\nu))j}|\md \nu|& \lesssim \frac{1}{\gamma^2} \left( \frac{|j-c_*t|}{\sqrt{t}}\right)^2 \frac{e^{-\beta_0 \frac{|j-c_*t|^2}{t}}}{t^{3/2}} \lesssim \frac{e^{-\beta_0 \frac{|j-c_*t|^2}{2t}}}{t^{3/2}}, \\
\int_{\Gamma_{\zeta/\gamma}} \left|\Im(\nu)\right|^2 e^{\Re(\nu)t}e^{\Re(\varpi(\nu))j}|\md \nu| &\lesssim e^{-\frac{\zeta^2}{\gamma}\frac{t}{c_*}} \underbrace{\int_{\Gamma_{\zeta/\gamma}} \left|\Im(\nu)\right|^2 e^{-c_\star \Im(\nu)^2 t}|\md \nu|}_{\leq \frac{C}{t^{3/2}}} \lesssim \frac{e^{-\beta_0 \frac{|j-c_*t|^2}{t}}}{t^{3/2}}.
\end{align*}
Let us note that in the first estimate, we used the fact that for each integer $m\geq1$ one can always find some constant $C_m>0$ such that $x^me^{-\beta_0 x^2}\leq C_m e^{-\frac{\beta_0}{2} x^2}$ for any $x\geq0$.   

Finally, the remaining contribution in $\mathscr{E}_j^1(t)$ along $\Gamma_1$ gives
\bqs
\int_{\Gamma_{1}}e^{\Re(\nu)t}\left(\eta^2+|\Im(\nu)|^2\right)e^{\Re(\varpi(\nu))j}|\md \nu| \lesssim e^{-\frac{\eta}{2}t-\frac{\zeta^2}{\gamma}\frac{t}{c_*}}\lesssim \frac{e^{-\beta_0 \frac{|j-c_*t|^2}{t}}}{t^{3/2}}.
\eqs

\paragraph{Study of $\mathscr{E}_j^2(t)$.}To handle the next error term $\mathscr{E}_j^2(t)$, we will use the fact that
\bqs
\left|e^{\varpi(\nu)j}-e^{\varphi(\nu)j}\right| \leq C j |\nu|^3 e^{j\left(\Re(\varpi(\nu))+\frac{|\Lambda_*|}{c_*^5}|\nu|^3+|\nu|^4|\Omega(\nu)|\right)}, \quad \nu \in B_\epsilon(0),
\eqs
together with our estimates \eqref{boundRevarpiImp} and \eqref{IneqRenu} to get
\begin{align*}
\left| \frac{1}{2\pi \mbi}\int_{\Gamma_{\zeta/\gamma}} e^{\nu t} \nu \left(e^{\varpi(\nu)j}-e^{\varphi(\nu)j}\right)\md \nu \right| &\lesssim  j \int_{\Gamma_{\zeta/\gamma}} \left( \frac{|\zeta|^4}{\gamma^4}+|\Im(\nu)|^4 \right) e^{t\Re(\nu)+j\left(\Re(\varpi(\nu))+\frac{|\Lambda_*|}{c_*^5}|\nu|^3+|\nu|^4|\Omega(\nu)|\right)} |\md \nu| \\
&\lesssim j e^{-\frac{\zeta^2}{\gamma}\frac{t}{c_*}} \int_{\Gamma_{\zeta/\gamma}} \left( \frac{|\zeta|^4}{\gamma^4}+|\Im(\nu)|^4 \right) e^{-c_\star \Im(\nu)^2 t} |\md \nu| \\
&\lesssim j \left( \frac{|\zeta|^4}{\sqrt{t}}+ \frac{1}{t^{9/2}}\right) e^{-\frac{\zeta^2}{\gamma}\frac{t}{c_*}}.
\end{align*}
Next, we remark that
\bqs
j \left( \frac{|\zeta|^4}{\sqrt{t}}+ \frac{1}{t^{9/2}}\right) = \frac{j}{t} \left(\left(\frac{|j-c_*t|}{2\sqrt{t}}\right)^4\frac{1}{t^{3/2}} +\frac{1}{t^{7/2}}\right),
\eqs
which then implies that
\bqs
\left| \frac{1}{2\pi \mbi}\int_{\Gamma_{\zeta/\gamma}} e^{\nu t} \nu \left(e^{\varpi(\nu)j}-e^{\varphi(\nu)j}\right)\md \nu \right|  \lesssim \frac{e^{-\beta_0 \frac{|j-c_*t|^2}{t}}}{t^{3/2}}.
\eqs
Similarly, the contribution along $\Gamma_1$ gives a Gaussian estimate with an exponential in time decaying factor leading to
\bqs
\left| \frac{1}{2\pi \mbi}\int_{\Gamma_1} e^{\nu t} \nu \left(e^{\varpi(\nu)j}-e^{\varphi(\nu)j}\right)\md \nu \right| \lesssim j e^{-\frac{\eta}{2}t-\frac{\zeta^2}{\gamma}\frac{t}{c_*}}\lesssim \frac{e^{-\beta_0 \frac{|j-c_*t|^2}{t}}}{t^{3/2}},
\eqs
where we simply used that $|\nu|^4\leq \epsilon^2$ since $\Gamma_1\subset B_\epsilon(0)$ and that
\bqs
 j e^{-\frac{\eta}{2}t} = \frac{j}{t} te^{-\frac{\eta}{2}t}\lesssim \frac{1}{t^{3/2}},
\eqs
for $t$ large enough and $|j-c_*t|\leq \vartheta t^{\alpha}$.

\paragraph{Study of $\mathscr{E}_j^4(t)$.} Before proceeding with $\mathscr{E}_j^3(t)$, we first treat $\mathscr{E}_j^4(t)$ since the computations are very close to the previous case. Indeed, for each $\nu\in\Gamma_{in}$, we use that
\bqs
\left|e^{\nu t} \nu^4\Omega(\nu)e^{\varphi(\nu)j}\right| \leq C |\nu|^4 e^{\Re(\nu)t}e^{\Re(\varphi(\nu))j},
\eqs
such that we deduce
\begin{align*}
\left| \frac{1}{2\pi \mbi}\int_{\Gamma_{\zeta/\gamma}} e^{\nu t} \nu^4\Omega(\nu)e^{\varphi(\nu)j}\md \nu \right| &\lesssim   \int_{\Gamma_{\zeta/\gamma}} \left( \frac{|\zeta|^4}{\gamma^4}+|\Im(\nu)|^4 \right) e^{t\Re(\nu)+j\Re(\varphi(\nu))} |\md \nu| \\
&\lesssim  e^{-\frac{\zeta^2}{\gamma}\frac{t}{c_*}} \int_{\Gamma_{\zeta/\gamma}} \left( \frac{|\zeta|^4}{\gamma^4}+|\Im(\nu)|^4 \right) e^{-c_\star \Im(\nu)^2 t} |\md \nu| \\
&\lesssim \left( \frac{|\zeta|^4}{\sqrt{t}}+ \frac{1}{t^{9/2}}\right) e^{-\frac{\zeta^2}{\gamma}\frac{t}{c_*}}\lesssim \frac{e^{-\beta_0 \frac{|j-c_*t|^2}{t}}}{t^{3/2}}.
\end{align*}
The contribution along $\Gamma_1$ is treated as usual.

\paragraph{Study of $\mathscr{E}_j^3(t)$.} Finally, for the last term $\mathscr{E}_j^3(t)$, we use that
\bqs
\left| e^{-j\frac{\Lambda_*}{c_*^5}\nu^3+j\nu^4\Omega(\nu)}-\left(1-j\frac{\Lambda_*}{c_*^5}\nu^3+j\nu^4\Omega(\nu)\right) \right| \leq C j^2 |\nu|^6e^{j\frac{|\Lambda_*|}{c_*^5}|\nu|^3+j|\nu|^4|\Omega(\nu)|},\quad \nu \in B_\epsilon(0),
\eqs
together with
\bqs
\Re(\varphi(\nu))+\frac{|\Lambda_*|}{c_*^5}|\nu|^3+|\nu|^4|\Omega(\nu)|\leq -\frac{\Re(\nu)}{c_*}+\frac{\beta^*}{c_*^3} \Re(\nu)^2-\frac{\beta_*}{c_*^3}\Im(\nu)^2, \quad \nu \in B_\epsilon(0),
\eqs
to obtain that
\begin{align*}
&\left| \frac{1}{2\pi \mbi}\int_{\Gamma_{\zeta/\gamma}} e^{\nu t} \left[e^{-j\frac{\Lambda_*}{c_*^5}\nu^3+j\nu^4\Omega(\nu)}-\left(1-j\frac{\Lambda_*}{c_*^5}\nu^3+j\nu^4\Omega(\nu)\right)\right]e^{\varphi(\nu)j}\md \nu \right| \\
&\lesssim  j^2 \int_{\Gamma_{\zeta/\gamma}} \left( \frac{|\zeta|^6}{\gamma^6}+|\Im(\nu)|^6 \right) e^{t\Re(\nu)+j\left(\Re(\varphi(\nu))+\frac{|\Lambda_*|}{c_*^5}|\nu|^3+|\nu|^4|\Omega(\nu)|\right)} |\md \nu| \\
&\lesssim j^2 e^{-\frac{\zeta^2}{\gamma}\frac{t}{c_*}} \int_{\Gamma_{\zeta/\gamma}} \left( \frac{|\zeta|^6}{\gamma^6}+|\Im(\nu)|^6 \right) e^{-c_\star \Im(\nu)^2 t} |\md \nu| \\
&\lesssim j^2 \left( \frac{|\zeta|^6}{\sqrt{t}}+ \frac{1}{t^{13/2}}\right) e^{-\frac{\zeta^2}{\gamma}\frac{t}{c_*}}.
\end{align*}
Finally, we note that
\bqs
j^2\frac{|\zeta|^6}{\sqrt{t}}=\left(\frac{j}{t}\right)^2 \left( \frac{|j-c_*t|}{2\sqrt{t}}\right)^6 \frac{1}{t^{3/2}},
\eqs
such that we eventually get
\bqs
\left|\mathscr{E}_j^3(t)\right|\lesssim \frac{e^{-\beta_0 \frac{|j-c_*t|^2}{t}}}{t^{3/2}},
\eqs
by, once again, noticing that the contribution along $\Gamma_1$ can be subsumed into the one obtained along $\Gamma_{\zeta/\gamma}$.
\end{proof}

Next, we move on with the analysis of the leading order terms. We define
\begin{align*}
\mathscr{I}_j^1(t)&:=\frac{1}{c_*}\frac{1}{2\pi \mbi}\int_{\Gamma_{in}}  e^{\nu t} e^{\varphi(\nu)j}\md \nu,\\
\mathscr{I}_j^2(t)&:=-\frac{2\cosh(\lambda_*)}{c_*^3}\frac{1}{2\pi \mbi}\int_{\Gamma_{in}} e^{\nu t}   \nu e^{\varphi(\nu)j} \md \nu,\\
\mathscr{I}_j^3(t)&:=-\frac{\Lambda_*}{c_*^6}\frac{1}{2\pi \mbi}\int_{\Gamma_{in}} e^{\nu t} j  \nu^3 e^{\varphi(\nu)j} \md \nu.
\end{align*}

In order to evaluate the above three integrals, we introduce some new contours
\bqs
\Gamma_d :=\left\{ z -\mbi \sqrt{\epsilon^2-\eta^2} ~|~ -\eta \leq z \leq \frac{\zeta}{\gamma}\right\}, \quad \Gamma_u :=\left\{ z +\mbi \sqrt{\epsilon^2-\eta^2} ~|~ -\eta \leq z \leq \frac{\zeta}{\gamma}\right\}
\eqs
and 
\bqs
\Gamma_i := \left\{\frac{\zeta}{\gamma}+ \mbi \xi ~|~ |\xi| \leq \sqrt{\epsilon^2-\eta^2} \right\},
\eqs
with this time $\zeta=\frac{j-c_*t}{2t}$ and $\gamma=\frac{j}{t}\frac{\cosh(\lambda_*)}{c_*^2}$. We choose $t$ large enough such that 
\bqs
-\frac{\eta}{2} \leq \frac{\zeta}{\gamma} \leq \frac{\eta}{2},
\eqs
which is always possible since $|\zeta|\leq \frac{\vartheta}{2t^{1-\alpha}}$ and $t$ is taken large enough. In the following, we will write $\Gamma_{in}=\Gamma_d\cup\Gamma_u\cup\Gamma_i\subset B_\epsilon(0)$ and we refer to Figure~\ref{fig:Gammain} for an illustration.

\begin{figure}[t!]
  \centering
  \includegraphics[width=.45\textwidth]{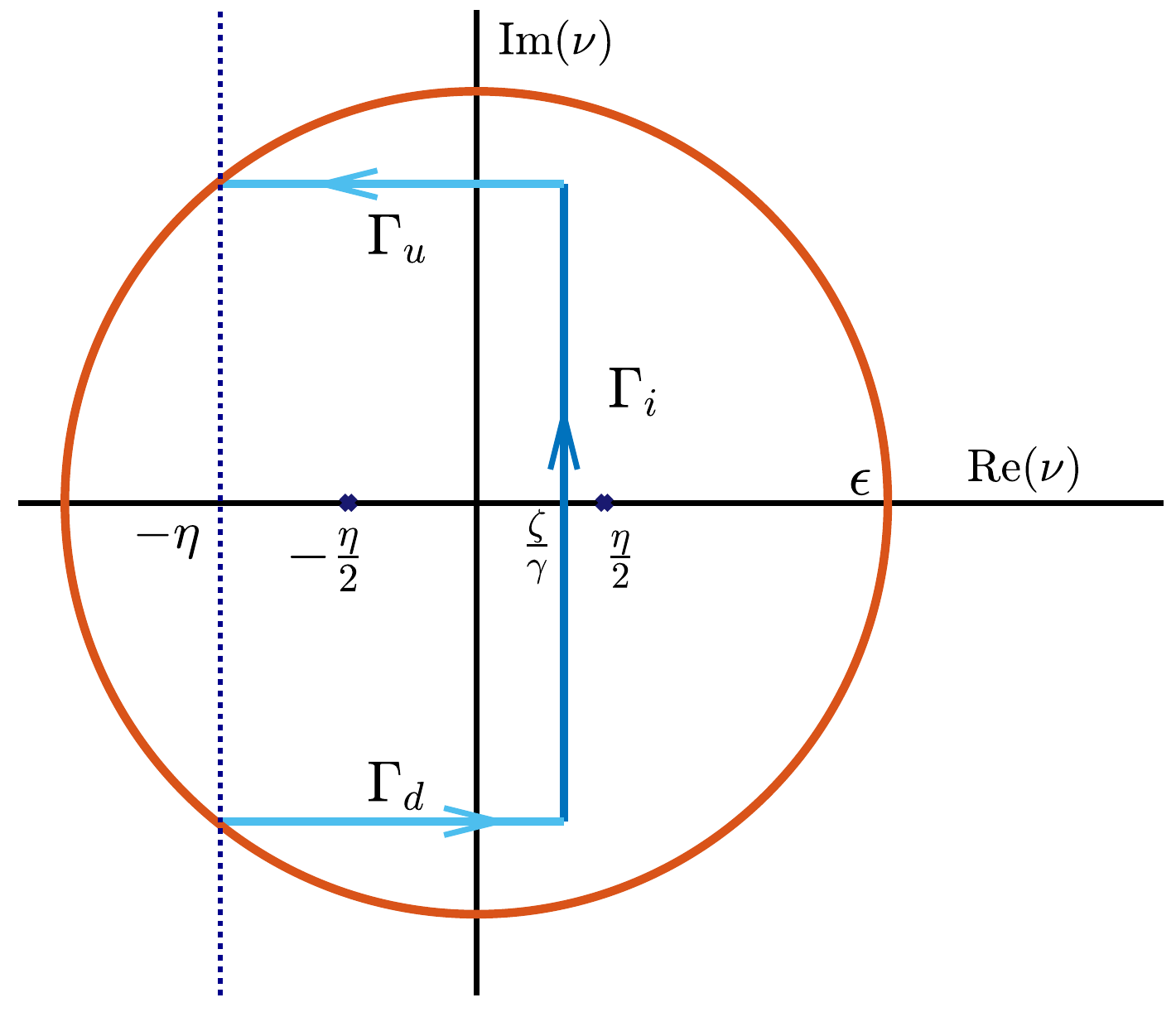}
  \caption{Illustration of the geometry of the contour $\Gamma_{in}=\Gamma_d\cup\Gamma_u\cup\Gamma_i\subset B_\epsilon(0)$ used in the computation of the integrals $\mathscr{I}_j^k(t)$ for $k=1,2,3$. }
  \label{fig:Gammain}
\end{figure}


\begin{lem}\label{lemIj1}
For any $\vartheta>0$ and $\alpha\in(0,1)$, there is some $T_0>1$ such that for each $j\in\Z$ with $|j-c_*t|\leq \vartheta t^\alpha$ and $t\geq T_0$ one can decompose $\mathscr{I}_j^1(t)$ as
\begin{align*}
\mathscr{I}_j^1(t)&=\left(\frac{1}{\sqrt{2t\cosh(\lambda_*)}}+\frac{1}{2tc_*}\left(-\left(\frac{j-c_*t}{\sqrt{2\cosh(\lambda_*)t}} \right)+\left(\frac{j-c_*t}{\sqrt{2\cosh(\lambda_*)t}} \right)^3 \right)\right)\mathcal{G}\left(\frac{j-c_*t}{\sqrt{2\cosh(\lambda_*)t}} \right)\\
&~~~+\widetilde{\mathscr{R}_j^1}(t),
\end{align*}
with
\bqs
\left| \widetilde{\mathscr{R}_j^1}(t)\right| \leq \frac{C}{t^{3/2}} \exp\left(-\beta \frac{(j-c_*t)^2}{t}\right),
\eqs
for some uniform constants $C>0$ and $\beta>0$.
\end{lem}

\begin{proof}

We first compute
\begin{align*}
\frac{1}{c_*}\frac{1}{2\pi \mbi}\int_{\Gamma_{i}} e^{\nu t} e^{\varphi(\nu)j}\md \nu &=\frac{1}{c_*} \frac{1}{2\pi} \exp\left(-\frac{(j-c_*t)^2}{4t c_* \gamma} \right) \int_{-\sqrt{\epsilon^2-\eta^2}}^{\sqrt{\epsilon^2-\eta^2}}  e^{-\frac{t\gamma}{c_*}\xi^2} \md \xi \\
&= \frac{1}{2\pi \sqrt{t\gamma c_*}}\exp\left(-\frac{(j-c_*t)^2}{4t c_* \gamma} \right) \int_{-\sqrt{\epsilon^2-\eta^2}\sqrt{\frac{t\gamma}{c_*}}}^{\sqrt{\epsilon^2-\eta^2}\sqrt{\frac{t\gamma}{c_*}}}e^{-z^2}\md z.
\end{align*}
Since $|j-c_*t|\leq \vartheta t^{\alpha}$, we have that
\bqs
\frac{1}{\sqrt{t\gamma c_*}}=\frac{1}{\sqrt{t\cosh(\lambda_*)}}\left(1-\frac{1}{2}\frac{j-c_*t}{tc_*}+ O\left(\left(\frac{j-c_*t}{t}\right)^2\right)\right),
\eqs
together with
\bqs
\left| e^{-\frac{(j-c_*t)^2}{4t c_* \gamma}}-e^{-\frac{(j-c_*t)^2}{4t\cosh(\lambda_*)}} -\frac{(j-c_*t)^3}{4t^2\cosh(\lambda_*)c_*}e^{-\frac{(j-c_*t)^2}{4t\cosh(\lambda_*)}} \right| \lesssim \frac{1}{t} e^{-\beta \frac{(j-c_*t)^2}{t}},
\eqs
for some $\beta>0$ and $t$ large enough. Next we remark that
\bqs
\int_\R e^{-z^2}\md z-\int_{-\sqrt{\epsilon^2-\eta^2}\sqrt{\frac{t\gamma}{c_*}}}^{\sqrt{\epsilon^2-\eta^2}\sqrt{\frac{t\gamma}{c_*}}}e^{-z^2}\md z = 2 \int_{\sqrt{\epsilon^2-\eta^2}\sqrt{\frac{t\gamma}{c_*}}}^{+\infty}e^{-z^2}\md z,
\eqs
for which we can use estimate of the complementary error function $\mathrm{erfc}(x):=2\int_x^{+\infty}e^{-z^2}\md z$ which satisfies for $x>1$
\bqs
\mathrm{erfc}(x)\leq \frac{e^{-x^2}}{x}.
\eqs
As consequence, for $t$ large enough such that $\sqrt{\epsilon^2-\eta^2}\sqrt{\frac{t\gamma}{c_*}}>1$, one has
\bqs
\left|\int_\R e^{-z^2}\md z-\int_{-\sqrt{\epsilon^2-\eta^2}\sqrt{\frac{t\gamma}{c_*}}}^{\sqrt{\epsilon^2-\eta^2}\sqrt{\frac{t\gamma}{c_*}}}e^{-z^2}\md z\right| \lesssim \frac{1}{\sqrt{t}} e^{-\beta t},
\eqs
for some $\beta>0$. 

Using the above computations, we see that
\bqs
\frac{1}{c_*}\frac{1}{2\pi \mbi}\int_{\Gamma_{i}} e^{\nu t} e^{\varphi(\nu)j}\md \nu = \frac{1}{ \sqrt{4\pi t\gamma c_*}}e^{-\frac{(j-c_*t)^2}{4t c_* \gamma}}  \underbrace{- \frac{1}{\pi \sqrt{t\gamma c_*}}e^{-\frac{(j-c_*t)^2}{4t c_* \gamma}} \int_{\sqrt{\epsilon^2-\eta^2}\sqrt{\frac{t\gamma}{c_*}}}^{+\infty}e^{-z^2}\md z}_{:=\mathscr{R}^1_j(t)},
\eqs
with
\bqs
\left|\mathscr{R}^1_j(t)\right| \lesssim \frac{1}{t^{3/2}} e^{-\beta \frac{(j-c_*t)^2}{t}}.
\eqs
Next, we also see that if we define 
\begin{align*}
\mathscr{R}^2_j(t):=\frac{1}{ \sqrt{4\pi t\gamma c_*}}e^{-\frac{(j-c_*t)^2}{4t c_* \gamma}} -\frac{1}{\sqrt{4\pi t \cosh(\lambda_*)}} \left(  1 -\frac{1}{2}\frac{j-c_*t}{tc_*}+\frac{(j-c_*t)^3}{4t^2\cosh(\lambda_*)c_*}  \right)e^{-\frac{(j-c_*t)^2}{4t\cosh(\lambda_*)}},
\end{align*}
then for $t$ large enough we have that
\bqs
\left|\mathscr{R}^2_j(t)\right| \lesssim \frac{1}{t^{3/2}} e^{-\beta \frac{(j-c_*t)^2}{t}},
\eqs
and
\bqs
\frac{1}{c_*}\frac{1}{2\pi \mbi}\int_{\Gamma_{i}} e^{\nu t} e^{\varphi(\nu)j}\md \nu = \frac{1}{\sqrt{4\pi t \cosh(\lambda_*)}} \left(  1 -\frac{1}{2}\frac{j-c_*t}{tc_*}+\frac{(j-c_*t)^3}{4t^2\cosh(\lambda_*)c_*}  \right)e^{-\frac{(j-c_*t)^2}{4t\cosh(\lambda_*)}}+\mathscr{R}^1_j(t)+\mathscr{R}^2_j(t).
\eqs
Finally, we rewrite the above leading order terms as
\begin{align*}
\mathscr{A}_j(t)&:=\frac{1}{\sqrt{4\pi t \cosh(\lambda_*)}} \exp\left(-\frac{(j-c_*t)^2}{4t\cosh(\lambda_*)}\right)=\frac{1}{\sqrt{2t\cosh(\lambda_*)}}\mathcal{G}\left(\frac{j-c_*t}{\sqrt{2\cosh(\lambda_*)t}} \right),\\
\mathscr{B}_j(t)&:=-\frac{1}{2}\frac{1}{\sqrt{4\pi t \cosh(\lambda_*)}} \frac{j-c_*t}{tc_*}e^{-\frac{(j-c_*t)^2}{4t\cosh(\lambda_*)}}+\frac{1}{\sqrt{4\pi t \cosh(\lambda_*)}} \frac{(j-c_*t)^3}{4t^2\cosh(\lambda_*)c_*} e^{-\frac{(j-c_*t)^2}{4t\cosh(\lambda_*)}}\\
&=\frac{1}{2tc_*}\left(- \left(\frac{j-c_*t}{\sqrt{2\cosh(\lambda_*)t}}\right) +\left(\frac{j-c_*t}{\sqrt{2\cosh(\lambda_*)t}}\right) ^3\right)\mathcal{G}\left(\frac{j-c_*t}{\sqrt{2\cosh(\lambda_*)t}} \right),
\end{align*}
such that we have the decomposition
\bqs
\mathscr{I}_j^1(t) =\mathscr{A}_j(t) +\mathscr{B}_j(t)+\mathscr{R}^1_j(t)+\mathscr{R}^2_j(t)+\frac{1}{c_*}\frac{1}{2\pi \mbi}\int_{\Gamma_{u}\cup\Gamma_d} e^{\nu t} e^{\varphi(\nu)j}\md \nu.
\eqs
To conclude the study of $\mathscr{I}_j^1(t)$, it remains to evaluate the integrals on $\Gamma_{d,u}$. We only treat the case $\Gamma_d$ since the other integral can be handled similarly. We need to bound
\bqs
\frac{1}{2\pi \mbi}\int_{\Gamma_{d}} e^{\nu t} e^{\varphi(\nu)j}\md \nu =\frac{1}{2\pi \mbi} \int_{-\eta}^{\frac{\zeta}{\gamma}} e^{t\left(z -\mbi \sqrt{\epsilon^2-\eta^2}\right)} e^{\varphi\left(z -\mbi \sqrt{\epsilon^2-\eta^2}\right)j}\md z
\eqs
and upon setting $\xi_\epsilon:=\sqrt{\epsilon^2-\eta^2}$ we have
\bqs
t(z-\mbi \xi_\epsilon)+\varphi\left(z -\mbi \xi_\epsilon\right)j=\mbi \xi_\epsilon\left( -t+\frac{j}{c_*}\right)-j\frac{\cosh(\lambda_*)}{c_*^3}\xi_\epsilon^2+\frac{t}{c_*}\left( -2\zeta z +\gamma z^2\right)-2\mbi j \xi_\epsilon \frac{\cosh(\lambda_*)}{c_*^3}z.
\eqs
As a consequence,
\begin{align*}
\left|  \frac{1}{2\pi \mbi}\int_{\Gamma_{d}} e^{\nu t} e^{\varphi(\nu)j}\md \nu \right| & \lesssim \exp\left(- \frac{t}{c_*}\gamma \xi_\epsilon^2 - \frac{t}{c_*}\frac{\zeta^2}{\gamma}  \right) \int_{-\eta}^{\frac{\zeta}{\gamma}} e^{\frac{t}{c_*}\gamma\left(z-\frac{\zeta}{\gamma} \right)^2}\md z \\
&\lesssim \exp\left(- \frac{t}{c_*}\gamma \xi_\epsilon^2 - \frac{t}{c_*}\frac{\zeta^2}{\gamma} +\frac{t}{c_*}\gamma\left(\eta+\frac{\zeta}{\gamma} \right)^2  \right).
\end{align*}
Finally, we note that
\bqs
-\xi_\epsilon^2+\left(\eta+\frac{\zeta}{\gamma} \right)^2=-\epsilon^2+\eta^2+\left(\eta+\frac{\zeta}{\gamma} \right)^2\leq -\epsilon^2+\eta^2+\left(\frac{3}{2}\eta \right)^2 < 0,
\eqs
since for fixed $\epsilon>0$, we can always choose $\eta>0$ small enough such that the above strict inequality holds. Thus we get
\bqs
\left|  \frac{1}{2\pi \mbi}\int_{\Gamma_{d}} e^{\nu t} e^{\varphi(\nu)j}\md \nu \right|\lesssim \frac{1}{t^{3/2}} e^{-\beta \frac{(j-c_*t)^2}{t}}.
\eqs
\end{proof}

We now turn our attention to $\mathscr{I}_j^2(t)$.

\begin{lem}\label{lemIj2}
For any $\vartheta>0$ and $\alpha\in(0,1)$, there is some $T_0>1$ such that for each $j\in\Z$ with $|j-c_*t|\leq \vartheta t^\alpha$ and $t\geq T_0$ one can decompose $\mathscr{I}_j^2(t)$ as
\bqs
\mathscr{I}_j^2(t)=-\frac{1}{tc_*}\left(\frac{j-c_*t}{\sqrt{2\cosh(\lambda_*)t}}\right) \mathcal{G}\left(\frac{j-c_*t}{\sqrt{2\cosh(\lambda_*)t}} \right)+\widetilde{\mathscr{R}_j^2}(t),
\eqs
with
\bqs
\left| \widetilde{\mathscr{R}_j^2}(t)\right| \leq \frac{C}{t^{3/2}} \exp\left(-\beta \frac{(j-c_*t)^2}{t}\right),
\eqs
for some uniform constants $C>0$ and $\beta>0$.
\end{lem}

\begin{proof} We first compute
\begin{align*}
\frac{1}{c_*}\frac{1}{2\pi \mbi}\int_{\Gamma_{i}} e^{\nu t} \nu e^{\varphi(\nu)j}\md \nu &=\frac{1}{c_*} \frac{1}{2\pi} \exp\left(-\frac{(j-c_*t)^2}{4t c_* \gamma} \right) \int_{-\sqrt{\epsilon^2-\eta^2}}^{\sqrt{\epsilon^2-\eta^2}} \left(\frac{\zeta}{\gamma}+ \mbi \xi \right)  e^{-\frac{t\gamma}{c_*}\xi^2} \md \xi \\
&=  \frac{1}{2\pi \sqrt{t\gamma c_*}} \frac{\zeta}{\gamma} \exp\left(-\frac{(j-c_*t)^2}{4t c_* \gamma} \right) \int_{-\sqrt{\epsilon^2-\eta^2}\sqrt{\frac{t\gamma}{c_*}}}^{\sqrt{\epsilon^2-\eta^2}\sqrt{\frac{t\gamma}{c_*}}}e^{-z^2}\md z\\
&=\frac{\zeta}{\gamma}\left(\mathscr{A}_j(t)+\mathscr{B}_j(t)\right)+\frac{\zeta}{\gamma}\left(\mathscr{R}^1_j(t)+\mathscr{R}^2_j(t)\right).
\end{align*}
For $t$ large enough, it is enough to remark that
\bqs
\frac{\zeta}{\gamma}=c_* \frac{j-c_*t}{2t\cosh(\lambda_*)}\left(1+O\left(\frac{j-c_*t}{tc_*}\right)\right),
\eqs
such that
\bqs
\frac{1}{c_*}\frac{1}{2\pi \mbi}\int_{\Gamma_{i}} e^{\nu t} \nu e^{\varphi(\nu)j}\md \nu=\frac{c_*}{2t\cosh(\lambda_*)}\left(\frac{j-c_*t}{\sqrt{2\cosh(\lambda_*)t}}\right) \mathcal{G}\left(\frac{j-c_*t}{\sqrt{2\cosh(\lambda_*)t}} \right)+\mathscr{R}^3_j(t),
\eqs
where 
\bqs
\mathscr{R}^3_j(t):=\left(\frac{\zeta}{\gamma}-c_* \frac{j-c_*t}{2t\cosh(\lambda_*)}\right)\mathscr{A}_j(t)+\frac{\zeta}{\gamma}\mathscr{B}_j(t)+\frac{\zeta}{\gamma}\left(\mathscr{R}^1_j(t)+\mathscr{R}^2_j(t)\right)
\eqs
with
\bqs
\left|\mathscr{R}^3_j(t)\right|\lesssim \frac{1}{t^{3/2}} e^{-\beta \frac{(j-c_*t)^2}{t}}.
\eqs
Similarly to the previous case, the integrals on $\Gamma_{d,u}$ produce terms which can be subsumed into the above Gaussian estimate.
\end{proof}

Finally, we study the last integral $\mathscr{I}_j^3(t)$.

\begin{lem}\label{lemIj3}
For any $\vartheta>0$ and $\alpha\in(0,1)$, there is some $T_0>1$ such that for each $j\in\Z$ with $|j-c_*t|\leq \vartheta t^\alpha$ and $t\geq T_0$ one can decompose $\mathscr{I}_j^3(t)$ as
\bqs
\mathscr{I}_j^3(t)=\frac{-\Lambda_*}{4tc_*\cosh(\lambda_*)^2}\left(-3\left(\frac{j-c_*t}{\sqrt{2\cosh(\lambda_*)t}} \right)+\left(\frac{j-c_*t}{\sqrt{2\cosh(\lambda_*)t}} \right)^3 \right)\mathcal{G}\left(\frac{j-c_*t}{\sqrt{2\cosh(\lambda_*)t}} \right)+\widetilde{\mathscr{R}_j^3}(t),
\eqs
with
\bqs
\left| \widetilde{\mathscr{R}_j^3}(t)\right| \leq \frac{C}{t^{3/2}} \exp\left(-\beta \frac{(j-c_*t)^2}{t}\right),
\eqs
for some uniform constants $C>0$ and $\beta>0$.
\end{lem}

\begin{proof}  As usual, the main contribution will come from the integration along $\Gamma_i$ and we have
 \begin{align*}
\frac{1}{2\pi \mbi}\int_{\Gamma_{i}} e^{\nu t} j \nu^3 e^{\varphi(\nu)j}\md \nu &=  \frac{j}{2\pi} \exp\left(-\frac{(j-c_*t)^2}{4t c_* \gamma} \right) \int_{-\sqrt{\epsilon^2-\eta^2}}^{\sqrt{\epsilon^2-\eta^2}} \left(\frac{\zeta}{\gamma}+ \mbi \xi \right)^3  e^{-\frac{t\gamma}{c_*}\xi^2} \md \xi \\
&=  \frac{j}{2\pi} \exp\left(-\frac{(j-c_*t)^2}{4t c_* \gamma} \right) \int_{-\sqrt{\epsilon^2-\eta^2}}^{\sqrt{\epsilon^2-\eta^2}}\left[ \left(\frac{\zeta}{\gamma} \right)^3 +3\mbi \left(\frac{\zeta}{\gamma} \right)^2\xi -3\frac{\zeta}{\gamma} \xi^2-\mbi \xi^3 \right]e^{-\frac{t\gamma}{c_*}\xi^2} \md \xi \\
&=  \frac{j}{2\pi} \exp\left(-\frac{(j-c_*t)^2}{4t c_* \gamma} \right) \int_{-\sqrt{\epsilon^2-\eta^2}}^{\sqrt{\epsilon^2-\eta^2}}\left[ \left(\frac{\zeta}{\gamma} \right)^3 -3\frac{\zeta}{\gamma} \xi^2 \right]e^{-\frac{t\gamma}{c_*}\xi^2} \md \xi.
\end{align*}
This time, for $t$ large enough, we have
\bqs
j \left(\frac{\zeta}{\gamma}\right)^3=c_*t \left(1+\frac{j-c_*t}{c_*t}\right)\left( c_* \frac{j-c_*t}{2t\cosh(\lambda_*)}\right)^3\left(1+O\left(\frac{j-c_*t}{tc_*}\right)\right),
\eqs
which gives that
\begin{align*}
\frac{j}{2\pi} \left(\frac{\zeta}{\gamma} \right)^3 \exp\left(-\frac{(j-c_*t)^2}{4t c_* \gamma} \right) \int_{-\sqrt{\epsilon^2-\eta^2}}^{\sqrt{\epsilon^2-\eta^2}}e^{-\frac{t\gamma}{c_*}\xi^2} \md \xi &=\frac{c_*^5}{4t\cosh(\lambda_*)^2}\left(\frac{j-c_*t}{\sqrt{2\cosh(\lambda_*)t}}\right)^3 \mathcal{G}\left(\frac{j-c_*t}{\sqrt{2\cosh(\lambda_*)t}} \right)\\
&~~~ +\mathscr{R}^4_j(t),
\end{align*}
with
\bqs
\left|\mathscr{R}^4_j(t)\right|\lesssim \frac{1}{t^{3/2}} e^{-\beta \frac{(j-c_*t)^2}{t}}.
\eqs
Regarding the last term, we note first that
\begin{align*}
\int_{-\sqrt{\epsilon^2-\eta^2}}^{\sqrt{\epsilon^2-\eta^2}} \xi^2 e^{-\frac{t\gamma}{c_*}\xi^2} \md \xi&=\frac{1}{\left(\frac{t\gamma}{c_*} \right)^{3/2}}\int_{-\sqrt{\epsilon^2-\eta^2}\sqrt{\frac{t\gamma}{c_*}}}^{\sqrt{\epsilon^2-\eta^2}\sqrt{\frac{t\gamma}{c_*}}} z^2 e^{-z^2} \md z\\
&=-\frac{1}{\left(\frac{t\gamma}{c_*} \right)^{3/2}}\sqrt{\epsilon^2-\eta^2}\sqrt{\frac{t\gamma}{c_*}}e^{-\frac{t\gamma}{c_*}\left(\epsilon^2-\eta^2\right)}+\frac{1}{2\left(\frac{t\gamma}{c_*} \right)^{3/2}}\int_{-\sqrt{\epsilon^2-\eta^2}\sqrt{\frac{t\gamma}{c_*}}}^{\sqrt{\epsilon^2-\eta^2}\sqrt{\frac{t\gamma}{c_*}}}  e^{-z^2} \md z.
\end{align*}
As a consequence, we have that
\bqs
-3 \frac{\zeta}{\gamma}  \frac{j}{2\pi} e^{-\frac{(j-c_*t)^2}{4t c_* \gamma} } \int_{-\sqrt{\epsilon^2-\eta^2}}^{\sqrt{\epsilon^2-\eta^2}} \xi^2 e^{-\frac{t\gamma}{c_*}\xi^2} \md \xi = -3 \frac{\zeta}{\gamma}  \frac{j}{4\pi \left(\frac{t\gamma}{c_*} \right)^{3/2}} e^{-\frac{(j-c_*t)^2}{4t c_* \gamma}} \int_{-\sqrt{\epsilon^2-\eta^2}\sqrt{\frac{t\gamma}{c_*}}}^{\sqrt{\epsilon^2-\eta^2}\sqrt{\frac{t\gamma}{c_*}}}  e^{-z^2} \md z +\mathscr{R}_j^5(t),
\eqs
with
\bqs
\left|\mathscr{R}_j^5(t)\right|\lesssim \frac{1}{t^{3/2}} e^{-\beta \frac{(j-c_*t)^2}{t}}.
\eqs

Finally, we remark that
\bqs
j \frac{\zeta}{\gamma} \frac{1}{\left(\frac{t\gamma}{c_*} \right)^{3/2}}=\frac{c_*^5}{2t\cosh(\lambda_*)^{5/2}} \frac{j-c_*t}{\sqrt{t}}\left(1+O\left(\frac{j-c_*t}{tc_*}\right)\right),
\eqs
for $t$ large enough such that we readily obtain that
\begin{align*}
-3 \frac{\zeta}{\gamma}  \frac{j}{4\pi \left(\frac{t\gamma}{c_*} \right)^{3/2}} e^{-\frac{(j-c_*t)^2}{4t c_* \gamma}} \int_{-\sqrt{\epsilon^2-\eta^2}\sqrt{\frac{t\gamma}{c_*}}}^{\sqrt{\epsilon^2-\eta^2}\sqrt{\frac{t\gamma}{c_*}}}  e^{-z^2} \md z&=-\frac{3c_*^5}{4t\cosh(\lambda_*)^2}\frac{j-c_*t}{\sqrt{2\cosh(\lambda_*)t}}\mathcal{G}\left(\frac{j-c_*t}{\sqrt{2\cosh(\lambda_*)t}} \right)\\
&~~~+\mathscr{R}_j^6(t),
\end{align*}
with
\bqs
\left|\mathscr{R}_j^6(t)\right|\lesssim \frac{1}{t^{3/2}} e^{-\beta \frac{(j-c_*t)^2}{t}}.
\eqs
Similarly to the previous case, the integrals on $\Gamma_{d,u}$ produce terms which can be subsumed into the above Gaussian estimate.
\end{proof}

\paragraph{Conclusion.} We can now put together the results that we have obtained. In a first step, we managed to decompose $\mathscr{G}_j(t)$ as
\bqs
\mathscr{G}_j(t)=\mathscr{I}_j^1(t)+\mathscr{I}_j^2(t)+\mathscr{I}_j^3(t)+\mathscr{E}_j^1(t)+\mathscr{E}_j^2(t)+\mathscr{E}_j^3(t)+\mathscr{E}_j^4(t)+\frac{1}{2\pi \mbi}\int_{\Gamma_{out}} e^{\nu t} \bG_j(\nu)\md \nu,
\eqs
and proved that
\bqs
\left| \mathscr{E}_j^1(t)\right|+\left|\mathscr{E}_j^2(t)\right|+\left|\mathscr{E}_j^3(t)\right|+\left|\mathscr{E}_j^4(t)\right|+\left|\frac{1}{2\pi \mbi}\int_{\Gamma_{out}} e^{\nu t} \bG_j(\nu)\md \nu\right| \lesssim \frac{1}{t^{3/2}} e^{-\beta \frac{(j-c_*t)^2}{t}}
\eqs
Next, we see that if we set
\begin{align*}
\mathscr{H}_j(t)&:=\frac{1}{\sqrt{2t\cosh(\lambda_*)}}\mathcal{G}\left(\frac{j-c_*t}{\sqrt{2\cosh(\lambda_*)t}} \right)-3\left(\frac{1}{2tc_*}-\frac{\Lambda_*}{4tc_*\cosh(\lambda_*)^2}\right) \left(\frac{j-c_*t}{\sqrt{2\cosh(\lambda_*)t}}\right) \mathcal{G}\left(\frac{j-c_*t}{\sqrt{2\cosh(\lambda_*)t}} \right)\\
&~~~+\left(\frac{1}{2tc_*}-\frac{\Lambda_*}{4tc_*\cosh(\lambda_*)^2}\right) \left(\frac{j-c_*t}{\sqrt{2\cosh(\lambda_*)t}}\right)^3 \mathcal{G}\left(\frac{j-c_*t}{\sqrt{2\cosh(\lambda_*)t}}\right),
\end{align*}
then using Lemma~\ref{lemIj1}, Lemma~\ref{lemIj2} and Lemma~\ref{lemIj3}
\bqs
\mathscr{I}_j^1(t)+\mathscr{I}_j^2(t)+\mathscr{I}_j^3(t)=\mathscr{H}_j(t)+\widetilde{\mathscr{R}_j^1}(t)+\widetilde{\mathscr{R}_j^2}(t)+\widetilde{\mathscr{R}_j^3}(t),
\eqs
with
\bqs
\left|\widetilde{\mathscr{R}_j^1}(t)\right| +\left|\widetilde{\mathscr{R}_j^2}(t)\right| +\left|\widetilde{\mathscr{R}_j^3}(t) \right| \lesssim \frac{1}{t^{3/2}} e^{-\beta \frac{(j-c_*t)^2}{t}}.
\eqs
As a consequence, we have that
\bqs
\mathscr{R}_j(t):=\mathscr{G}_j(t)-\mathscr{H}_j(t)=\widetilde{\mathscr{R}_j^1}(t)+\widetilde{\mathscr{R}_j^2}(t)+\widetilde{\mathscr{R}_j^3}(t)+\mathscr{E}_j^1(t)+\mathscr{E}_j^2(t)+\mathscr{E}_j^3(t)+\mathscr{E}_j^4(t)+\frac{1}{2\pi \mbi}\int_{\Gamma_{out}} e^{\nu t} \bG_j(\nu)\md \nu
\eqs
satisfies
\bqs
\left| \mathscr{R}_j(t)  \right| \leq \frac{C}{t^{3/2}} \exp\left(-\beta \frac{(j-c_*t)^2}{t}\right),
\eqs
for some uniform constants $C>0$ and $\beta>0$. Finally, we see that $\H_j(t)$ can be factored into the following condensed formula
\bqs
\H_j(t)= \left[\frac{1}{\sqrt{2t\cosh(\lambda_*)}}+\frac{1}{t}\mathcal{P}\left(\frac{j-c_*t}{\sqrt{2\cosh(\lambda_*)t}} \right)\right]\mathcal{G}\left(\frac{j-c_*t}{\sqrt{2\cosh(\lambda_*)t}} \right), \quad t>0, \quad j\in\Z.
\eqs
where the polynomial function $\mathcal{P}$ is given in \eqref{equationp}. Indeed, with the expression of $\Lambda_*$ in \eqref{LambdaStar}, we readily see that
\bqs
\frac{1}{2c_*}-\frac{\Lambda_*}{4c_*\cosh(\lambda_*)^2}=\frac{c_*}{24\cosh(\lambda_*)^2}.
\eqs
 This concludes the proof of Proposition~\ref{proprefined}.

\subsection{Sharp asymptotics from odd compactly supported initial conditions}
\label{sec2.4}

Now, recalling our notation for the linear operator $\mathscr{L}$ given by
$$(\mathscr{L}\br)_j=e^{\lambda_*} \left(r_{j-1}-2r_j+r_{j+1}\right)-c_*(r_{j+1}-r_j), \quad j\in\Z,$$
we let $w_j(t)$ be the solution of the linear Cauchy problem
\bqq
\frac{\md}{\md t} w_j(t)-\left(\mathscr{L}w(t)\right)_j
=0, ~~~~ t>0,~ j\in\Z,
\label{linear-eqn-w}
\eqq
with a nontrivial, odd and compactly supported initial condition $w^0_j$ for $j\in\Z$ in the sense that
\begin{equation*}
	w_j^0=-w_{-j}^0\ge 0 ~~\forall j\ge 0, ~~~w_j^0=0~~\forall j\ge J_w+1,
\end{equation*}
for some positive $J_w\geq1$. Then, $w_j(t)$ can be written in the form
\bqs
w_j(t)=\sum_{\ell = 1}^{J_w}\left(\G_{j-\ell}(t)-\G_{j+\ell}(t)\right)w_\ell^0, \quad t>0, \quad j\in\Z.
\eqs

Now, from Proposition~\ref{proprefined}, for $t\geq T_0$ and each $|j-c_*t|\leq \vartheta t^\alpha$ with  $\vartheta>0$ and $\alpha\in(0,1)$, we have that the temporal Green's function $\G_{j\pm\ell}(t)$ can be decomposed as 
\bqs
\G_{j\pm\ell}(t)=\left[\frac{1}{\sqrt{2\cosh(\lambda_*)t}}+\frac{1}{t}\mathcal{P}\left(\frac{j\pm\ell-c_*t}{\sqrt{2\cosh(\lambda_*)t}} \right)\right]\mathcal{G}\left(\frac{j\pm\ell-c_*t}{\sqrt{2\cosh(\lambda_*)t}} \right)+\mathscr{R}_{j\pm\ell}(t),
\eqs
with
\bqs
\left| \mathscr{R}_{j\pm\ell}(t)  \right| \leq \frac{C}{t^{3/2}} \exp\left(-\beta \frac{(j\pm\ell-c_*t)^2}{t}\right).
\eqs
Therefore
\begin{align*}
w_j(t)&=\frac{1}{\sqrt{2t\cosh(\lambda_*)}}\sum_{\ell = 1}^{J_w}\left(\mathcal{G}\left(\frac{j-\ell-c_*t}{\sqrt{2\cosh(\lambda_*)t}} \right)-\mathcal{G}\left(\frac{j+\ell-c_*t}{\sqrt{2\cosh(\lambda_*)t}} \right)\right)w_\ell^0\\
&~~~+\frac{1}{t}\sum_{\ell = 1}^{J_w}\left(\mathcal{P}\left(\frac{j-\ell-c_*t}{\sqrt{2\cosh(\lambda_*)t}} \right)\mathcal{G}\left(\frac{j-\ell-c_*t}{\sqrt{2\cosh(\lambda_*)t}} \right)-\mathcal{P}\left(\frac{j+\ell-c_*t}{\sqrt{2\cosh(\lambda_*)t}} \right)\mathcal{G}\left(\frac{j+\ell-c_*t}{\sqrt{2\cosh(\lambda_*)t}} \right) \right)w_\ell^0\\
&~~~+\sum_{\ell = 1}^{J_w}\left(\mathscr{R}_{j-\ell}(t)-\mathscr{R}_{j+\ell}(t)\right)w_\ell^0.
\end{align*}
By Proposition~\ref{proprefined}, we have the bound
\bqs
\left|\sum_{\ell = 1}^{J_w}\left(\mathscr{R}_{j-\ell}(t)-\mathscr{R}_{j+\ell}(t)\right)v_\ell^0\right| \lesssim \frac{1}{t^{3/2}} \exp\left(-\beta_0 \frac{(j-c_*t)^2}{t}\right)\left(\sum_{\ell = 1}^{J_w}w_\ell^0\right),
\eqs
for some $\beta_0>0$.
Next we remark that for $|j-c_*t|\leq \vartheta t^\alpha$ and as $t\rightarrow +\infty$ 
\begin{align*}
\exp\left(-\frac{(j-\ell-c_*t)^2}{4t\cosh(\lambda_*)}\right)&-\exp\left(-\frac{(j+\ell-c_*t)^2}{4t\cosh(\lambda_*)}\right)\\
&=\exp\left(-\frac{(j-c_*t)^2+\ell^2}{4t\cosh(\lambda_*)}\right)\left(e^{\frac{\ell(j-c_*t)}{2t\cosh(\lambda_*)}}-e^{-\frac{\ell(j-c_*t)}{2t\cosh(\lambda_*)}}\right)\\
&=\exp\left(-\frac{(j-c_*t)^2+\ell^2}{4t\cosh(\lambda_*)}\right)\left[\frac{\ell(j-c_*t)}{t\cosh(\lambda_*)}+O\left(\left( \frac{(j-c_*t)}{t}\right)^3\right)\right].
\end{align*}
As a consequence for $|j-c_*t|\leq \vartheta t^\alpha$, we get
\begin{align*}
\frac{1}{\sqrt{2t\cosh(\lambda_*)}}\sum_{\ell = 1}^{J_w}&\left(\mathcal{G}\left(\frac{j-\ell-c_*t}{\sqrt{2\cosh(\lambda_*)t}} \right)-\mathcal{G}\left(\frac{j+\ell-c_*t}{\sqrt{2\cosh(\lambda_*)t}} \right)\right)w_\ell^0\\
& \sim \frac{1}{\sqrt{4\pi t \cosh(\lambda_*)}} \frac{j-c_*t}{t\cosh(\lambda_*)}\exp\left(-\frac{(j-c_*t)^2}{4t\cosh(\lambda_*)} \right)\left( \sum_{\ell=1}^{J_w}\ell w_\ell^0\right),
\end{align*}
as $t\rightarrow +\infty$. Now, restricting ourselves to the diffusive regime $\alpha\in(0,1/2]$, we get that the remaining contributions coming from the difference of the terms in $\mathcal{P}\left(\frac{j\pm\ell-c_*t}{\sqrt{2\cosh(\lambda_*)t}} \right)\mathcal{G}\left(\frac{j\pm\ell-c_*t}{\sqrt{2\cosh(\lambda_*)t}} \right)$ are of higher order. Indeed, in that case, we get
\begin{align*}
\sum_{\ell = 1}^{J_w}&\left(\mathcal{P}\left(\frac{j-\ell-c_*t}{\sqrt{2\cosh(\lambda_*)t}} \right)\mathcal{G}\left(\frac{j-\ell-c_*t}{\sqrt{2\cosh(\lambda_*)t}} \right)-\mathcal{P}\left(\frac{j+\ell-c_*t}{\sqrt{2\cosh(\lambda_*)t}} \right)\mathcal{G}\left(\frac{j+\ell-c_*t}{\sqrt{2\cosh(\lambda_*)t}} \right)\right)w_\ell^0\\
&\sim - C_* \frac{j-c_*t}{t^{3/2}}\exp\left(-\frac{(j-c_*t)^2}{4t\cosh(\lambda_*)} \right)\left( \sum_{\ell=1}^{J_w}\ell w_\ell^0\right),
\end{align*}
as $t\rightarrow +\infty$ for some universal constant $C_*>0$. In summary, we have proved the following lemma.

\begin{lem}\label{lemAsymptLin}
The  solution $(w_j(t))_{j\in\Z}$ of the linear Cauchy problem \eqref{linear-eqn-w} starting from an odd, nontrivial and compactly supported sequence  $(w_j^0)_{j\in\Z} \in\ell^\infty(\Z)$ satisfies,  for $|j-c_*t|\leq \vartheta t^\alpha$  with $\alpha\in(0,1/2]$ and any $\vartheta>0$, the asymptotic expansion 
\begin{equation}
	\label{estimate-w_j}
	w_j(t) \sim \frac{1}{  \cosh(\lambda_*)^{3/2}\sqrt{4\pi}} \frac{j-c_*t}{t^{3/2}}\left( \sum_{\ell=1}^{\infty}\ell w_\ell^0\right)~~~~~\text{as}~t\rightarrow +\infty.
\end{equation}
In particular, there exists some large time $t_0>0$ such that
\begin{equation}
	\label{j_0}
	w_j(t)>0~~~~~~~\text{for all}~t\ge t_0~\text{and}~1 \leq j-c_*t \leq \vartheta\sqrt{t},
\end{equation}
for any $\vartheta>1/\sqrt{t_0}$.
\end{lem}

Note that the above Lemma~\ref{lemAsymptLin} gives a precise information on the  solution $(w_j(t))_{j\in\Z}$ of the linear Cauchy problem \eqref{linear-eqn-w} starting from an odd, nontrivial and compactly supported sequence at the diffusive scale. For the construction of the upper barrier in the forthcoming section, we will need a control of the solution beyond this diffusive regime. This is the purpose of the next lemma.
\begin{lem}\label{lemDiffusive} Let  $(w_j(t))_{j\in\Z}$ be the solution of the linear Cauchy problem \eqref{linear-eqn-w} starting from a nontrivial, bounded, compactly supported sequence  $(w_j^0)_{j\in\Z} \not\equiv0$ with $w_j^0=0$ for all $|j|\geq J$ for some $J\geq2$. Then, for each $A>1$, there exists $\eta_A>0$ such that
\bqs
\left| w_j(t)\right| \leq \|w^0\|_{\ell^{\infty}(\Z)}e^{-A\left( \frac{j-J-c_*t}{\sqrt{t+1}}-\eta_A\right)}, \quad t>0,\quad j-J-c_*t \geq \eta_A\sqrt{t+1}.
\eqs
\end{lem}
\begin{proof}
Let $\xi\in\R$ be given and define the sequence $\mathbf{r}:=\left(e^{\xi j}\right)_{j\in\Z}$. For each $j\in\Z$, we have
\bqs
(\mathscr{L}\mathbf{r})_j=\left(e^{\lambda_*}\left(e^{\xi}-2+e^{-\xi}\right)-c_*\left(e^{\xi}-1\right) \right)e^{\xi j}=\left[-c_*\xi+\cosh(\lambda_*)\xi^2\left(1+\omega(\xi)\right)\right]e^{\xi j},
\eqs
where $\omega:\R\mapsto\R$ is an analytic function with $\omega(0)=0$.  Next, for $A>1$, we define $\overline{w}_j(t)$ according to
\bqs
\overline{w}_j(t)=\|w^0\|_{\ell^{\infty}(\Z)}e^{-A\left( \frac{j-J-c_*t}{\sqrt{t+1}}-\eta_A\right)}, \quad t\geq0, \quad j\in\Z, 
\eqs
where $\eta_A>2$ is set to
\bqs
\eta_A:=2\cosh(\lambda_*)A\left(1+\|\omega\|_{L^\infty([-A,0])}\right).
\eqs
As a consequence, upon denoting $\eta=\frac{j-J-c_*t}{\sqrt{t+1}}$, we obtain that
\bqs
\frac{\md}{\md t} \overline{w}_j(t)-\left(\mathscr{L}\overline{w}(t)\right)_j=\frac{A}{(t+1)}\left(\frac{\eta}{2}-\cosh(\lambda_*)A\left(1+\omega\left(-\frac{A}{\sqrt{t+1}}\right)\right)\right)\overline{w}_j(t), \quad t>0, \quad j\in\Z,
\eqs
from which we deduce that
\bqs
\frac{\md}{\md t} \overline{w}_j(t)-\left(\mathscr{L}\overline{w}(t)\right)_j\geq0, \quad t>0, \quad j\geq\zeta(t),
\eqs
where we have set $\zeta(t):=J+c_*t+\eta_A\sqrt{t+1}$. We now check that we can apply the maximum principle from Proposition~\ref{mp} to get that $w_j(t)\leq \overline{w}_j(t)$ for all $j\geq\zeta(t)$ and $t>0$. First, at time $t=0$, since $w^0_j=0$ for $j\geq J$, we get that 
\bqs
w^0_j=0\leq \overline{w}_j(0)=\|w^0\|_{\ell^{\infty}(\Z)}e^{-A\left( j-J-\eta_A\right)}, \quad j\geq \zeta(0)-1=J+\eta_A-1>J.
\eqs
Then, since the solution of the linear Cauchy problem \eqref{linear-eqn-w} satisfies $w_j(t)\leq \|w^0\|_{\ell^{\infty}(\Z)}$ for all $t\geq0$ and $j\in\Z$, we get that
\bqs
w_j(t)\leq \|w^0\|_{\ell^{\infty}(\Z)}\leq\overline{w}_j(t),\text{ for each } \zeta(t)-1\leq j \leq \zeta(t), \quad t>0.
\eqs
From Proposition~\ref{mp}, we obtain that $w_j(t)\leq \overline{w}_j(t)$ for all $j\geq\zeta(t)$ and $t>0$, as claimed. By linearity of equation~\eqref{linear-eqn-w}, it is easy to check that $-\overline{w}_j(t)$ is a subsolution implying that we also have $-\overline{w}_j(t)\leq w_j(t)$ for $j\geq\zeta(t)$ and $t>0$. This concludes the proof of the lemma.
\end{proof}


\section{Logarithmic delay of the position for the level sets}\label{secproofthmlog}

The aim of this section is to prove Theorem \ref{thmlog} which shows the logarithmic delay on the expansion of the level sets of the solution. To start with, we set 
\bqs
v_j(t):=e^{\lambda_*(j-c_*t)}u_j(t), \quad t>0, \quad j\in\Z,
\eqs
such that the sequence $(v_j(t))_{j\in\Z}$ is now a solution to the following modified lattice Fisher-KPP equation
\begin{equation}
	\label{v}
	\begin{aligned}
		\begin{cases}
				\frac{\md}{\md t} v_j(t) =e^{\lambda_*} \left(v_{j-1}(t)\!-\!2v_j(t)+v_{j+1}(t)\right)\!-\!c_*(v_{j+1}(t)-v_j(t))\!-\!\mathcal{R}_j(t;v_j(t)), &t>0,~j\in\Z,\\
			v_j(0)  =v_j^0=e^{\lambda_*j}u_j^0, & j\in\Z.
		\end{cases}
	\end{aligned}
\end{equation}
 Here, the nonlinear term $\mathcal{R}_j(t;s)$ is defined as
\begin{equation}
	\label{term R-nonnegative}
	\mathcal{R}_j(t;s):=f'(0)s-e^{\lambda_*(j-c_*t)}f(e^{-\lambda_*(j-c_*t)}s),
\end{equation}
for $s\in\R$ and  $(t,j)\in(0,+\infty)\times\Z$. We remark that $\mathcal{R}_j(t;s)\geq0$ for all $s\in\R$ and $(t,j)\in(0,+\infty)\times\Z$. This directly comes from our standing assumptions that $0<f(s)\le f'(0)s$ for $s\in(0,1)$ and the fact that we extended linearly $f$ for $u\in(-\infty,0)\cup(1,+\infty)$. We readily notice that $\mathcal{R}_j(t;v_j(t))\ge 0$ for $(t,j)\in(0,+\infty)\times\Z$. 

\subsection{Upper and lower bounds for $v_j(t)$}
\label{sec-preliminary}

In this section, we provide upper and lower barriers for the function $v_j(t)$ for $t$ sufficiently large and $j\in\Z$ ahead of the position $j-c_*t\approx 0$.

\subsubsection{Upper barrier for $v_j(t)$}

We start with the construction of a supersolution to  \eqref{v} for all $t$ large enough and $j\in\Z$ ahead of $j-c_*t\approx 0$ by following the strategy we developed in Section~\ref{seccontinuous}. More precisely, consider $\delta\in(0,1/3)$, that will be as small as needed. We will estimate $v_j(t)$ ahead of $j-c_*t = -t^\delta$. To do so, we construct a supersolution for \eqref{v} as follows:
\begin{equation}
\label{overline v}
\overline v_j(t):=\overline\xi(t) w_j(t)+\frac{1}{(1+t)^{\frac{3}{2}-\beta}}\cos\left(\frac{j-c_*t}{(1+t)^\alpha}\right)\mathbbm{1}_{\left\{j\in\Z~|~ -t^\delta-1\le j-c_*t\le  \frac{3\pi}{2}(1+t)^\alpha\right\}}+\Xi\left(\frac{j-c_*t}{\sqrt{1+t}} \right),
\end{equation}
for $t$ large enough and $j\in\Z$ with $j-c_*t\ge -t^\delta-1$, where the unknown $\overline \xi(t)\in \mathscr{C}^1$ is assumed to be positive and bounded in $(0,+\infty)$, and $\overline \xi'(t)\ge 0$  in $(0,+\infty)$. Here, $(w_j(t))_{j\in\Z}$ is the solution of the linear Cauchy problem \eqref{linear-eqn-w} starting from an odd, nontrivial and compactly supported sequence  $(w_j^0)_{j\in\Z} \in\ell^\infty(\Z)$. The other parameters  $\alpha\in(1/3,1/2)$ and $\beta>0$ will be determined in the course of investigation. The function $\Xi:\R_+\to\R_+$ is defined as
\bqs
\Xi(\eta):=2\|w^0\|_{\ell^{\infty}(\Z)}\Gamma\left(\eta\right)e^{-a\left(\eta-\eta_2\right)}, \quad \eta \geq0,
\eqs
with $a>1$ and $\Gamma\geq0$ a smooth non decreasing cut-off function that satisfies $\Gamma(x)=0$ for $x\in[0,\eta_1]$ and $\Gamma(x)=1$ for $x\geq \eta_2$ with $0<\eta_1<\eta_2$. Note that the cut-off function $\Gamma$ can be constructed so as to ensure that
\bqs
\|\Gamma'\|_{L^{\infty}[\eta_1,\eta_2]}\leq \frac{C_1}{\eta_2-\eta_1}, \text{ and }  \|\Gamma''\|_{L^{\infty}[\eta_1,\eta_2]}\leq \frac{C_2}{(\eta_2-\eta_1)^2},
\eqs
for two positive constants $C_{1,2}>0$ independent of $\eta_{1}$ and $\eta_2$. The parameters $\eta_1,\eta_2>0$ and $a>1$ need to be properly chosen and will be fixed along the proof.

 Our aim is now to prove that $\overline v_j(t)$ is a supersolution of \eqref{v} for $t$ large enough and $j\in\Z$ with $j-c_*t\ge -t^\delta$. For notational convenience,  we define the following sequence
\bqq
p_j(t):=\frac{1}{(1+t)^{\frac{3}{2}-\beta}}\cos\left(\frac{j-c_*t}{(1+t)^\alpha}\right), \quad t>0, \quad j\in\Z.
\label{cosinepj}
\eqq
The key point of the forthcoming computations will be to verify that the action of the linear operator $\mathscr{L}$ on the above cosine perturbation and the function $\Xi$ is well behaved within the range of interest. As we have already seen in Section~\ref{seccontinuous} for the continuous setting, the cosine perturbation is designed in such a way to compensate for the lack of positivity of $w_j(t)$ within the range $-t^{\delta}-1\leq j-c_*t\leq0$. On the other hand, the exponential correction introduced with $\Xi$ will compensate for our lack of information regarding the positivity of $w_j(t)$ beyond the diffusive scale. We conjecture that it may be possible to directly prove that $w_j(t)>0$ for all $j-c_*t\geq1$ and $t$ large enough, but this is beyond the scope of the present paper.

 We divide the half-space $j-c_*t\ge -t^\delta$ with $j\in\Z$ and $t$ large enough into five different zones
\begin{align*}
R_1&:=\left\{j\in\Z~|~ -t^\delta\le j-c_*t\le 1\right\},\quad R_2:=\left\{j\in\Z~|~ 1\le j-c_*t\le \frac{\pi}{4}(1+t)^\alpha\right\},\\
R_3&:= \left\{j\in\Z~|~\frac{\pi}{4}(1+t)^\alpha\le j-c_*t\le \frac{3\pi}{2}(1+t)^\alpha\right\},\\
R_4&:= \left\{j\in\Z~|~  \frac{3\pi}{2}(1+t)^\alpha\le j-c_*t\le \eta_1 (1+t)^{\frac{1}{2}}\right\},\\
R_5&:= \left\{j\in\Z~|~ j-c_*t\geq  \eta_1 (1+t)^{\frac{1}{2}} \right\}.
\end{align*}
We point out that  the interfaces of each zone require rather delicate analysis, due to the nonlocal feature of the equation \eqref{v}.


\paragraph{In region $R_1$.} There holds
\bqs
\frac{-t^\delta}{(1+t)^\alpha}\le \frac{j-c_*t}{(1+t)^\alpha}\le \frac{1}{(1+t)^\alpha}.
\eqs
Due to $\delta<\alpha$ and due to the asymptotics \eqref{estimate-w_j} of $w_j(t)$, we get that
\bqs
p_j(t)\sim \frac{1}{(1+t)^{\frac{3}{2}-\beta}}~~\text{and}~~ -\frac{1}{(1+t)^{\frac{3}{2}-\delta}}\lesssim w_j(t)\lesssim \frac{1}{(1+t)^{\frac{3}{2}}}~~~\text{for}~t~\text{large enough}.
\eqs
Therefore  the cosine perturbation will enable $\overline v_j(t)$ to be positive in this region. To be specific, we first require that $\beta>\delta>0,$ so that the cosine term plays a dominant role here, namely,
\begin{equation}\label{BC-upper}
	\overline v_j(t)\sim \frac{1}{(1+t)^{\frac{3}{2}-\beta}}>0~~~~~\text{for}~t~\text{large enough}.
\end{equation}
 Moreover, a straightforward computation gives that, for $t$ large enough,
 \begin{equation*}
\frac{\md}{\md t}(\overline \xi(t)w_j(t))-\left(\mathscr{L}[\overline{\xi}(t)w(t)]\right)_j=\overline\xi'(t)w_j(t)\gtrsim -\frac{\overline\xi'(t)}{(1+t)^{\frac{3}{2}-\delta}}.
 \end{equation*}
 Define by $j_1$  the leftmost integer in $R_1$. By  double angle formulas, we see for $t$ large enough and $j\in R_1\backslash \left\{j_1\right\}$ that
\begin{align*}
	\frac{\md}{\md t} p_j(t)-\left(\mathscr{L}p(t)\right)_j
	&=\frac{\beta-\frac{3}{2}}{(1+t)^{\frac{3}{2}-\beta+1}}\cos\left(\frac{j-c_*t}{(1+t)^\alpha}\right)+\frac{1}{(1+t)^{\frac{3}{2}-\beta}}\Bigg(\frac{c_*}{(1+t)^\alpha}+\alpha\frac{j-c_*t}{(1+t)^{\alpha+1}}\Bigg)\sin\left(\frac{j-c_*t}{(1+t)^\alpha}\right)
	\\
	&~~~~~~~~+\frac{4e^{\lambda_*}}{(1+t)^{\frac{3}{2}-\beta}}\sin^2\left(\frac{1/2}{(1+t)^\alpha}\right)\cos\left(\frac{j-c_*t}{(1+t)^\alpha}
	\right)\\
	&~~~~~~~~ -\frac{2c_*}{(1+t)^{\frac{3}{2}-\beta}}\sin\left(\frac{1/2}{(1+t)^\alpha}\right)\sin\left(\frac{j-c_*t+1/2}{(1+t)^\alpha}\right)\\
	&\sim \frac{1}{(1+t)^{\frac{3}{2}-\beta+2\alpha}}.
\end{align*}
At $j=j_1$, since $p_{j-1}(t)=0$ for all $t>0$, it follows from the Taylor expansion that
	\begin{align*}
\frac{\md}{\md t} p_{j_1}(t)-\left(\mathscr{L}p(t)\right)_{j_1}&=	\frac{\beta-\frac{3}{2}}{(1+t)^{\frac{3}{2}-\beta+1}}\cos\left(\frac{j_1-c_*t}{(1+t)^\alpha}\right)\\
&~~~+\frac{1}{(1+t)^{\frac{3}{2}-\beta}}\Bigg(\frac{c_*}{(1+t)^\alpha}+\alpha\frac{j_1-c_*t}{(1+t)^{\alpha+1}}\Bigg)\sin\left(\frac{j_1-c_*t}{(1+t)^\alpha}\right)\\
	&~~~+\Bigg( e^{\lambda_*}\cos \left(\frac{j_1-c_*t}{(1+t)^\alpha}\right)-
	\sin \left(\frac{j_1-c_*t}{(1+t)^\alpha}\right)\frac{c_*-e^{\lambda_*}}{(1+t)^\alpha}+O\left(\frac{1}{(1+t)^{2\alpha}}\right)
	\Bigg)
	\frac{1}{(1+t)^{\frac{3}{2}-\beta}}\\
	&\sim\frac{1}{(1+t)^{\frac{3}{2}-\beta}}~~~~~~~~~~~\text{for}~t~\text{large enough}.
	\end{align*}
Therefore, in order to ensure that $\frac{\md}{\md t}\overline{v}_j(t)-\left(\mathscr{L}\overline{v}(t)\right)_j\ge 0$ for $t$ large enough in this region, it suffices for $\overline\xi(t)$ to satisfy
\bqs
\frac{1}{(1+t)^{\frac{3}{2}-\beta+2\alpha}}\gg\frac{\overline \xi'(t)}{(1+t)^{\frac{3}{2}-\delta}}~~~~~\text{for}~t~\text{large enough}.
\eqs
We can then require that
\begin{equation}\label{R1-cdn}
	0\le  \overline \xi'(t)\ll \frac{1}{(1+t)^{2\alpha+\delta-\beta}} ~~~~~\text{for}~t~\text{large enough}.
\end{equation}

\paragraph{In region $R_2$.} There holds
\bqs
\frac{1}{(1+t)^\alpha}
\le \frac{j-c_*t}{(1+t)^\alpha}\le \frac{\pi}{4}.
\eqs
 We notice that $w_j(t)$ is positive for all $t> 0$ in this region, as is the cosine perturbation. Since $\overline \xi(t)$ is assumed a priori to be  positive in $(0,+\infty)$, one has that the function  $\overline v_j(t)>0$ for all $t> 0$ in this region. Moreover,
$$\frac{\md}{\md t}(\overline \xi(t)w_j(t))-\left(\mathscr{L}[\overline{\xi}(t)w(t)]\right)_j=\overline \xi'(t)w_j(t)\ge 0$$ 
for all $t> 0$ in this region, due to our requirement that $\overline \xi'(t)\ge 0$ in $(0,+\infty)$.
By the same calculation as in region $R_1$, we have the following asymptotics in $R_2$
\bqs
\frac{\md}{\md t}p_j(t)-\left(\mathscr{L}p(t)\right)_j \sim \frac{1}{(1+t)^{\frac{3}{2}-\beta+2\alpha}}>0~~~\text{for}~t~\text{large enough}.
\eqs
Consequently,  there holds 
$\frac{\md}{\md t}\overline{v}_j(t)-\left(\mathscr{L}\overline{v}(t)\right)_j\ge 0$ for $t$ large enough in region $R_2$.

\paragraph{In  region $R_3$.} There holds
\bqs
\frac{\pi}{4}
\le \frac{j-c_*t}{(1+t)^\alpha}\le \frac{3\pi}{2}.
\eqs
 Notice that the cosine perturbation may be negative in this region. Let $j_3$ be the rightmost integer in $R_3$. We first look at  $R_3\backslash \left\{j_3\right\}$, where
 it follows from an analogous procedure as in preceding cases that
\bqs
\frac{\md}{\md t}p_j(t)-\left(\mathscr{L}p(t)\right)_j  \ge \frac{-1}{(1+t)^{\frac{3}{2}-\beta+2\alpha}}~~~~~\text{for}~t~\text{large enough}.
\eqs
At  $j_3$, by noticing that $p_{j_3+1}(t)=0$ for all $t>0$, we derive by using the Taylor expansion that 
\begin{align*}
	&~~~\frac{\md}{\md t}p_{j_3}(t)-\left(\mathscr{L}p(t)\right)_{j_3} \\	&=\frac{\beta-\frac{3}{2}}{(1+t)^{\frac{3}{2}-\beta+1}}\cos\left(\frac{j_3-c_*t}{(1+t)^\alpha}\right)+\frac{1}{(1+t)^{\frac{3}{2}-\beta}}\Bigg(\frac{c_*}{(1+t)^\alpha}+\alpha\frac{j_3-c_*t}{(1+t)^{\alpha+1}}\Bigg)\sin\left(\frac{j_3-c_*t}{(1+t)^\alpha}\right)\\	
	&~~~~~~~-\frac{1}{(1+t)^{\frac{3}{2}-\beta}}\Bigg(  e^{\lambda_*} \sin \left(\frac{j_3-c_*t}{(1+t)^\alpha}\right) \frac{1}{(1+t)^\alpha}+O\left(\frac{1}{(1+t)^{3\alpha}}\right)-e^{-\lambda_*}\cos\left(\frac{j_3-c_*t}{(1+t)^\alpha}\right)\Bigg)\\
	&\sim -\frac{e^{-\lambda_*}}{(1+t)^{\frac{3}{2}-\beta+\alpha}}\sin \left(\frac{j_3-c_*t}{(1+t)^\alpha}\right)+\frac{e^{-\lambda_*}}{(1+t)^{\frac{3}{2}-\beta}}\cos\left(\frac{j_3-c_*t}{(1+t)^\alpha}\right)\\
	&\sim\frac{e^{-\lambda_*}}{(1+t)^{\frac{3}{2}-\beta+\alpha}}~~~~~~~~~~~~~\text{for}~t~\text{large enough}.
\end{align*}
On the other hand, one observes from \eqref{estimate-w_j} that 
\bqs
w_j(t)\sim \frac{1}{(1+t)^{\frac{3}{2}-\alpha}}~~~~~\text{for}~t~\text{large enough}.
\eqs
However, the cosine perturbation may be negative in this area. Therefore, to ensure that $\overline v_j(t)>0$ and $\frac{\md}{\md t}\overline{v}_j(t)-\left(\mathscr{L}\overline{v}(t)\right)_j\ge 0$ for $t$ large enough in this region, we require this time $\alpha>\beta$ and
\bqq\label{R3-cdn}
\frac{\overline \xi'(t)}{(1+t)^{\frac{3}{2}-\alpha}}\gg\frac{1}{(1+t)^{\frac{3}{2}-\beta+2\alpha}}~~~\text{for}~t~\text{large enough}.
\eqq

\paragraph{In  region $R_4$.} We first define $\underline{j}_4$ and $\overline{j}_4$, respectively, as the leftmost and rightmost integers in region $R_4$. We note that in $R_4$ both the cosine perturbation and the exponential correction $\Xi$ are identically equal to zero, such that for all  $t$ large enough, we have $\overline v_j(t)=\overline\xi(t)w_j(t)>0$ thanks to \eqref{j_0}. Now standard computations give that for each $j\in R_4\backslash\left\{\underline{j}_4,\overline{j}_4 \right\}$ that
\bqs
\frac{\md}{\md t}\overline{v}_j(t)-\left(\mathscr{L}\overline{v}(t)\right)_j=\overline\xi'(t)w_j(t)\ge 0, \quad \text{ for }~t~\text{large enough},
\eqs
thanks to the nonnegativity assumption of $\overline \xi'(t)$ for $t> 0$. Now at $\underline{j}_4$ we have $p_{\underline{j}_4}(t)=p_{\underline{j}_4+1}(t)=0$ together with $p_{\underline{j}_4-1}(t)\leq0$, so that we get
that
\bqs
\frac{\md}{\md t}\overline{v}_{\underline{j}_4}(t)-\left(\mathscr{L}\overline{v}(t)\right)_{\underline{j}_4}=\overline\xi'(t)w_{\underline{j}_4}(t)-e^{\lambda_*}p_{\underline{j}_4-1}(t)\ge 0, \quad \text{ for }~t~\text{large enough}.
\eqs
On the other hand, at $\overline{j}_4$, we remark that $\Gamma(\eta)=0$ and $\Gamma\left(\eta+\frac{1}{\sqrt{1+t}}\right)>0$ with $\eta=\frac{\overline{j}_4-c_*t}{\sqrt{1+t}}$ such that we obtain that
\bqs
\frac{\md}{\md t}\overline{v}_{\overline{j}_4}(t)-\left(\mathscr{L}\overline{v}(t)\right)_{\overline{j}_4}=\overline\xi'(t)w_{\overline{j}_4}(t)-e^{-\lambda_*}\Xi\left(\eta+\frac{1}{\sqrt{1+t}}\right).
\eqs
By definition of the cut-off function, we have that $\Gamma^{(k)}(\eta_1)=0$ for all $k\geq0$, such that 
\bqs
\Xi\left(\eta+\frac{1}{\sqrt{1+t}}\right)=O\left(\frac{1}{(1+t)^{k/2}} \right), \text{ for any } k\geq0, \text{ and } t \text{ large},
\eqs
and thus
\bqs
\frac{\md}{\md t}\overline{v}_{\overline{j}_4}(t)-\left(\mathscr{L}\overline{v}(t)\right)_{\overline{j}_4}=\overline\xi'(t)w_{\overline{j}_4}(t)+o\left(\overline\xi'(t)w_{\overline{j}_4}(t)\right)\geq0, \quad \text{for}~t~\text{large enough}.
\eqs

\paragraph{Partial conclusion.} Gathering \eqref{R1-cdn} and \eqref{R3-cdn}, we should impose as in the continuous case
$$\frac{1}{(1+t)^{3\alpha-\beta}}\ll \overline \xi'(t)\ll \frac{1}{(1+t)^{2\alpha+\delta-\beta}}~~~~~\text{for}~t~\text{large enough}.$$
This is possible so long as $\delta<\alpha$, which is exactly what we have assumed. Let us take 
$$\overline \xi'(t)\sim \frac{1}{(1+t)^{3\alpha-2\beta}}~~ \text{and}~~\overline \xi(t)=1-\frac{1}{(1+t)^{3\alpha-2\beta-1}}~~\text{for}~t\ge 0,$$
which yields that $3\alpha-2\beta>2\alpha+\delta-\beta$, i.e., $\beta<\alpha-\delta$.
Due to our assumption that the function $\overline\xi(t)$ is positive and bounded in $(0,+\infty)$, it suffices to require
$3\alpha-2\beta-1>0$. Hence, we can fix $\delta\in(0,1/4)$ very small, then there exist $\alpha\in(1/3,1/2)$ and $\beta>0$ such that 
\begin{equation}
	\label{parameters}
	0<\delta<\beta<\min\left(\alpha-\delta,\frac{3\alpha-1}{2}\right)<\alpha<\frac{1}{2}.
\end{equation}

\paragraph{In region $R_5$.} We now turn our attention to the final region $R_5=\left\{ j\in\Z~|~ j-c_*t \geq \eta_1\sqrt{1+t}\right\}$ where one gets contributions from the correction $\Xi$. We first remark that for all $t$ large enough one has
\bqs
0<\overline{v}_j(t)=\overline\xi(t) w_j(t)+\Xi\left(\frac{j-c_*t}{\sqrt{t+1}}\right), \quad  \eta_1 \leq \frac{j-c_*t}{\sqrt{t+1}} \leq \eta_2 ,
\eqs
for any choice of $0<\eta_1<\eta_2$ thanks to Lemma~\ref{lemAsymptLin}. We now use Lemma~\ref{lemDiffusive} and let $A>a>1$  such that there exists $\eta_A>0$ for which
\bqs
\left| w_j(t)\right| \leq \|w^0\|_{\ell^{\infty}(\Z)}e^{-A\left( \frac{j-J-c_*t}{\sqrt{t+1}}-\eta_A\right)}, \quad \frac{j-J-c_*t}{\sqrt{t+1}} \geq \eta_A,
\eqs
where $J=J_\omega+1\geq2$ is the range of the support of the initial condition.  From the proof of Lemma~\ref{lemDiffusive}, we know that we can take: 
\bqs
\eta_A= 4\cosh(\lambda_*)A\left(1+\|\omega\|_{L^\infty([-A,0])}\right),
\eqs
for some analytic function $\omega$ verifying $\omega(0)=0$. We can also always assume that $t$ is large enough such that $\frac{J}{\sqrt{t+1}}\leq \eta_A$. With $\eta_2=2\eta_A$, we get that
\begin{align*}
\overline{v}_j(t)&\geq -\overline\xi(t)\|w^0\|_{\ell^{\infty}(\Z)}e^{-A\left( \frac{j-J-c_*t}{\sqrt{t+1}}-\eta_A\right)}+2\|w^0\|_{\ell^{\infty}(\Z)}e^{-a\left( \frac{j-c_*t}{\sqrt{t+1}}-\eta_2\right)}\\
&\geq \|w^0\|_{\ell^{\infty}(\Z)}\left(-e^{-(A-a)\left( \frac{j-c_*t}{\sqrt{t+1}}-\eta_2\right)}+2\right)e^{-a\left( \frac{j-c_*t}{\sqrt{t+1}}-\eta_2\right)}>0, \quad \quad \frac{j-c_*t}{\sqrt{t+1}} \geq \eta_2.
\end{align*}
As a consequence, from now on we fix $\eta_2>0$ as
\bqq
\eta_2:=8\cosh(\lambda_*)A\left(1+\|\omega\|_{L^\infty([-A,0])}\right),
\label{eqeta2}
\eqq
and then chose $\eta_1>0$, depending on $a$, as
\bqq
\label{eqeta1}
\eta_1:=3\cosh(\lambda_*)a.
\eqq
We remark that since $1<a<A$ and $\eta_2=2\eta_A$, we always have 
\bqs
0<\eta_1=3\cosh(\lambda_*)a<2\eta_1<8\cosh(\lambda_*)a<8\cosh(\lambda_*)A\left(1+\|\omega\|_{L^\infty([-A,0])}\right)=\eta_2.
\eqs
Next, we verify that $\overline{v}_j(t)$ is indeed a supersolution in region $R_5$. We divide it into two subregions:
\bqs
R_5^1:=\left\{ j\in\Z ~|~ \eta_1\leq \frac{j-c_*t}{\sqrt{1+t}} \leq \eta_2\right\}, \quad R_5^2:=\left\{ j\in\Z ~|~ \frac{j-c_*t}{\sqrt{1+t}} \geq \eta_2\right\},
\eqs 
and throughout we denote $\eta=\frac{j-c_*t}{\sqrt{1+t}}$.

\paragraph{In region  $R_5^1$.} What changes in the intermediate range $R_5^1$ is that one gets an extra contribution from the cut-off function $\Gamma$. More precisely, if we denote by $\underline{j}_5^1$ and $\overline{j}_5^1$, respectively, as the leftmost and rightmost integers in region $R_5^1$. On the one hand, for any $j\in R_5^1\backslash\left\{\underline{j}_5^1,\overline{j}_5^1\right\}$ we compute:
\begin{align*}
e^{a(\eta-\eta_2)}\left(\mathscr{L}\left[\Gamma(\eta)e^{-a(\eta-\eta_2)}\right]\right)_j&=e^{\lambda_*}\left[\Gamma\left(\eta+\frac{1}{\sqrt{t+1}}\right)e^{-\frac{a}{\sqrt{t+1}}}-2\Gamma(\eta)+\Gamma\left(\eta-\frac{1}{\sqrt{t+1}}\right)e^{\frac{a}{\sqrt{t+1}}}\right]\\
&~~~-c_*\left[\Gamma\left(\eta+\frac{1}{\sqrt{t+1}}\right)e^{-\frac{a}{\sqrt{t+1}}}-\Gamma(\eta)\right]\\
&=e^{\lambda_*}\left[\frac{a^2}{(t+1)}\Gamma(\eta)+\frac{1}{(t+1)}\Gamma''(\eta)-\frac{2a}{(t+1)}\Gamma'(\eta)\right]\\
&~~~-c_*\left[\frac{1}{\sqrt{t+1}}\Gamma'(\eta)+ \frac{1}{2(t+1)}\Gamma''(\eta)-\frac{a}{\sqrt{t+1}}\Gamma(\eta)\right]\\
&~~~-c_*\left[\frac{a^2}{2(t+1)}\Gamma(\eta)-\frac{a}{(t+1)}\Gamma'(\eta)\right]+O\left( \frac{1}{(t+1)^{3/2}}\right)\\
&=-\frac{c_*}{\sqrt{t+1}}\left[\Gamma'(\eta)-a\Gamma(\eta)\right]+\frac{\cosh(\lambda_*)a^2}{(t+1)}\Gamma(\eta)+\frac{\cosh(\lambda_*)}{(t+1)}\Gamma''(\eta)\\
&~~~-\frac{2\cosh(\lambda_*)a}{(t+1)}\Gamma'(\eta)+O\left( \frac{1}{(t+1)^{3/2}}\right).
\end{align*} 
On the other hand, for any $j\in R_5^1$, we have
\begin{align*}
e^{a(\eta-\eta_2)}\left(\frac{\md}{\md t}\left[\Gamma(\eta)e^{-a(\eta-\eta_2)}\right]\right)&=\left(\Gamma'(\eta)-a\Gamma(\eta)\right)\frac{\md \eta}{\md t}\\
&=\left(\Gamma'(\eta)-a\Gamma(\eta)\right)\left(-\frac{c_*}{\sqrt{t+1}}-\frac{\eta}{2(t+1)}\right).
\end{align*}
As a consequence, for any $j\in R_5^1\backslash\left\{\underline{j}_5^1,\overline{j}_5^1\right\}$, one has
\begin{align*}
e^{a(\eta-\eta_2)}&\left(\frac{\md}{\md t}\left[\Gamma(\eta)e^{-a(\eta-\eta_2)}\right]-\left(\mathscr{L}\left[\Gamma(\eta)e^{-a(\eta-\eta_2)}\right]\right)_j\right)\\
&=\frac{1}{(t+1)}\left[a\left(\frac{\eta}{2}-a\cosh(\lambda_*)\right)\Gamma(\eta)+\left(2\cosh(\lambda_*)a-\frac{\eta}{2}\right)\Gamma'(\eta)  -\cosh(\lambda_*)\Gamma''(\eta)\right]\\
&~~~+O\left( \frac{1}{(t+1)^{3/2}}\right)\\
&\geq \frac{1}{(t+1)}\left[\underbrace{\frac{a^2\cosh(\lambda_*)}{2}\Gamma(\eta)+\left(2\cosh(\lambda_*)a-\frac{\eta}{2}\right)\Gamma'(\eta)  -\cosh(\lambda_*)\Gamma''(\eta)}_{:=\Theta_a(\eta)} \right]+O\left( \frac{1}{(t+1)^{3/2}}\right).
\end{align*}
Since $\Gamma'(\eta)> 0$ for $\eta\in(\eta_1,\eta_2)$, we get that
\bqs
\frac{a^2\cosh(\lambda_*)}{2}\Gamma(\eta)+\left(2\cosh(\lambda_*)a-\frac{\eta}{2}\right)\Gamma'(\eta)>0
\eqs
 for all $\eta\in(\eta_1,4\cosh(\lambda_*)a]$. And since we have
\bqs
 \|\Gamma''\|_{L^{\infty}[\eta_1,\eta_2]}\leq \frac{C_2}{(\eta_2-\eta_1)^2}\leq \frac{C_2}{\left(5\cosh(\lambda_*)a\right)^2},
\eqs
we get that $\Theta_a(\eta)>0$  for all $\eta\in(\eta_1,4\cosh(\lambda_*)a]$ provided that $a>1$ is large enough.  Now, for $\eta\in(4\cosh(\lambda_*)a,\eta_2)$, we get that
\bqs
\frac{a^2\cosh(\lambda_*)}{2}\Gamma(\eta)+\left(2\cosh(\lambda_*)a-\frac{\eta}{2}\right)\Gamma'(\eta)\geq \frac{a^2\cosh(\lambda_*)}{2}\Gamma(\eta)+2\cosh(\lambda_*)a\Gamma'(\eta)-\frac{C_1}{2(1-\eta_1/\eta_2)}.
\eqs
And since $2\eta_1=6\cosh(\lambda_*)a< 8 \cosh(\lambda_*)a < \eta_2$, we get that
\bqs
\frac{C_1}{2(1-\eta_1/\eta_2)}\leq C_1.
\eqs
As a consequence, for $\eta\in(4\cosh(\lambda_*)a,\eta_2)$, we have
\bqs
\Theta_a(\eta)\geq \frac{a^2\cosh(\lambda_*)}{2}\Gamma(\eta)+2\cosh(\lambda_*)a\Gamma'(\eta)-C_1-\frac{C_2}{\left(5\cosh(\lambda_*)a\right)^2}>0,
\eqs
provided that $a>1$ is large enough. As a conclusion, we can always find $a>1$ large enough such that $\Theta_a(\eta)>0$ for all $\eta\in(\eta_1,\eta_2)$, this then implies that 
\bqs
\frac{\md}{\md t} \overline{v}_j(t)-\left(\mathscr{L}\overline{v}(t)\right)_j\geq 0,~~~\text{for}~t~\text{large enough},
\eqs
in the range $j\in R_5^1\backslash\left\{\underline{j}_5^1,\overline{j}_5^1\right\}$. Finally, at the extremal end points of $R_5^1$ we get the following contributions. First, at $j=\underline{j}_5^1$, we observe that
\begin{align*}
e^{a(\eta-\eta_2)}\left(\mathscr{L}\left[\Gamma(\eta)e^{-a(\eta-\eta_2)}\right]\right)_{\underline{j}_5^1}&=-\Gamma(\eta)e^{\lambda_*}+\left(e^{\lambda_*}-c_*\right)\left[\Gamma\left(\eta+\frac{1}{\sqrt{t+1}}\right)e^{-\frac{a}{\sqrt{t+1}}}-\Gamma(\eta)\right]\\
&=-\Gamma(\eta)e^{\lambda_*}+O\left(\frac{1}{\sqrt{t+1}}\right),
\end{align*}
together with
\bqs
e^{a(\eta-\eta_2)}\left(\frac{\md}{\md t}\left[\Gamma(\eta)e^{-a(\eta-\eta_2)}\right]\right)=\left(\Gamma'(\eta)-a\Gamma(\eta)\right)\left(-\frac{c_*}{\sqrt{t+1}}-\frac{\eta}{2(t+1)}\right)=O\left(\frac{1}{\sqrt{t+1}}\right).
\eqs
As a consequence, we have at $j=\underline{j}_5^1$
\begin{align*}
e^{a(\eta-\eta_2)}&\left(\frac{\md}{\md t}\left[\Gamma(\eta)e^{-a(\eta-\eta_2)}\right]-\left(\mathscr{L}\left[\Gamma(\eta)e^{-a(\eta-\eta_2)}\right]\right)_{\underline{j}_5^1}\right)=\Gamma(\eta)e^{\lambda_*}+O\left(\frac{1}{\sqrt{t+1}}\right) \geq0,
\end{align*}
for $t$ large enough, from which we deduce that $\frac{\md}{\md t} \overline{v}_j(t)-\left(\mathscr{L}\overline{v}(t)\right)_j\geq 0$. Similar computations at the other boundary point $j=\overline{j}_5^1$ also yields to $\frac{\md}{\md t} \overline{v}_j(t)-\left(\mathscr{L}\overline{v}(t)\right)_j\geq 0$ for $t$ large enough.
\paragraph{In region  $R_5^2$.} Once again, we denote by $\underline{j}_5^2$ the left most integer in region $R_5^2$. In this regime, we have $\Gamma(\eta)=1$ such that we get for each $j\in R_5^2\backslash\left\{\underline{j}_5^2\right\}$
\begin{align*}
\frac{\md}{\md t} \overline{v}_j(t)-\left(\mathscr{L}\overline{v}(t)\right)_j&=\overline{\xi}'(t)w_j(t)+\frac{2a \|w^0\|_{\ell^{\infty}(\Z)}}{(t+1)}\left(\frac{\eta}{2}-\cosh(\lambda_*)a\left(1+\omega\left(-\frac{a}{\sqrt{t+1}}\right)\right)\right)e^{-a\left( \eta-\eta_2\right)}\\
&\geq \left(-\overline{\xi}'(t)+\frac{2a}{(t+1)}\left(\frac{\eta_2}{2}-\cosh(\lambda_*)a\left(1+\omega\left(-\frac{a}{\sqrt{t+1}}\right)\right)\right)\right)\|w^0\|_{\ell^{\infty}(\Z)}e^{-a\left( \eta-\eta_2\right)}.
\end{align*}
By analyticity of the function $\omega$ and the fact that $\omega(0)=0$, we can always ensure that 
\bqs
1+\omega\left(-\frac{a}{\sqrt{t+1}}\right)\leq \frac{3}{2}, ~~~\text{for}~t~\text{large enough}.
\eqs
As a consequence, using the fact that $\eta_2=2\eta_A> 8 \cosh(\lambda_*)a$, we get
\bqs
\frac{\md}{\md t} \overline{v}_j(t)-\left(\mathscr{L}\overline{v}(t)\right)_j\geq \left(-\overline{\xi}'(t)+\frac{5a\cosh(\lambda_*)}{(t+1)}\right)\|w^0\|_{\ell^{\infty}(\Z)}e^{-a\left( \eta-\eta_2\right)}\geq0,~~~\text{for}~t~\text{large enough},
\eqs
since
\bqs
\overline \xi'(t)\sim \frac{1}{(1+t)^{3\alpha-2\beta}},
\eqs
with $3\alpha-2\beta>1$. Now, at $j=\underline{j}_5^2$, we have
\bqs
e^{a(\eta-\eta_2)}\left(\mathscr{L}\left[\Gamma(\eta)e^{-a(\eta-\eta_2)}\right]\right)_{\underline{j}_5^2}=-e^{\lambda_*}+e^{\lambda_*}\left[\Gamma\left(\eta-\frac{1}{\sqrt{t+1}}\right)e^{\frac{a}{\sqrt{t+1}}}-1\right]=-e^{\lambda_*}+O\left(\frac{1}{\sqrt{t+1}}\right),
\eqs
for $t$ large enough. Then we conclude that $\frac{\md}{\md t} \overline{v}_j(t)-\left(\mathscr{L}\overline{v}(t)\right)_j\geq 0$ is also satisfied at $j=\underline{j}_5^2$ for $t$ large enough.

\paragraph{Final conclusion.} First, we set $\zeta(t):=c_*t-t^\delta$. From the above analysis, one can choose $T_0>0$ sufficiently large such that $\zeta(T_0)-1>J$ (recall that $u_j^0=0$ for $j\ge J$) and such that
$\frac{\md}{\md t}\overline{v}_j(t)-\left(\mathscr{L}\overline{v}(t)\right)_j\ge 0$ for $t\ge T_0$ and  $j\in\Z$ with $j\geq \zeta(t)$. This together with \eqref{term R-nonnegative} then implies that
$\frac{\md}{\md t}\overline{v}_j(t)-\left(\mathscr{L}\overline{v}(t)\right)_j+\mathcal{R}_j(t;\overline v_j(t))\ge 0$ for $t\ge T_0$ and $j\in\Z$ with $j\ge \zeta(t)$.
 Moreover, due to the choice of $T_0$, we have $\overline v_j(T_0)>0=e^{\lambda_*j} u^0_j=v_j(0)$ for $j\in\Z$ with $j \ge\zeta(T_0)-1$. For $j\in\left[\zeta(t) -1, \zeta(t)\right)$, we observe that $v_j(t-T_0)=e^{\lambda_*(j-c_*(t-T_0))}u_j(t-T_0)\le e^{\lambda_*(c_*T_0- t^\delta)}$ for $t\ge T_0$ (since  $0\le u_j(t)\le  1$ for all $t\ge 0$ and $j\in\Z$), while $\overline v_j(t)\sim (1+t)^{-3/2+\beta}$ for $t\ge T_0$ by \eqref{BC-upper}, up to increasing $T_0$. Up to increasing $T_0$ again, we further have $(1+t)^{-3/2+\beta}>e^{\lambda_*(c_*T_0-t^\delta)}$ for all $t\ge T_0$, which will yield that $\overline v_j(t)\geq  v_j(t-T_0)$ at $j\in\left[\zeta(t) -1, \zeta(t)\right)$ for all $t\ge T_0$. 
 We then conclude that $\overline v_j(t)$ is a supersolution of \eqref{v} for all $t\ge T_0$ and  $j\in\Z$ with $j\ge  \zeta(t) $.  It follows from the comparison principle Proposition~\ref{cp} that
 \begin{equation}
 	\label{511-conclusion}
 	\overline v_j(t+T_0)\geq v_j(t)~~~~\text{for}~t\ge 0,~j-c_*t\geq -t^\delta.
 \end{equation}

\subsubsection{Lower barrier for $v_j(t)$}

In the special case that $f$ is linear in a small neighborhood of 0, namely, $f(s)=f'(0)s$ for $s\in[0,s_0)$, with $s_0\in(0,1)$ small, we would be able to control $v_j(t)$ by some multiple of $w_j(t)$ from below. Nevertheless,  the nonlinear term $f$ is not linear in the vicinity of 0 in general, for which we still hope to manage controlling $v_j(t)$ by $w_j(t)$. Accordingly, we need to do it in an area where the nonlinear term is negligible. Let $\delta$, $\beta$ and $\alpha$ be fixed as in \eqref{parameters}. The idea is to estimate $v_j(t)$ ahead of $j- c_*t= t^\delta $. To do so, let us construct a lower barrier as follows:
\begin{equation}
	\label{underline v}
	\underline v_j(t):=\underline\xi(t) \widetilde{w}_j(t)-\frac{1}{(1+t)^{\frac{3}{2}-\beta}}\cos\left(\frac{j-c_*t}{(1+t)^\alpha}\right)\mathbbm{1}_{\left\{j\in\Z~|~ t^\delta-1\le j-c_*t\le  \frac{3\pi}{2}(1+t)^\alpha\right\}},
\end{equation}
for $t$ large enough and  $j\in\Z$ with $1<t^\delta-1\le  j-c_*t$, where we assume that the unknown $\underline \xi(t)\in \mathscr{C}^1$  is positive and bounded away from 0 in $(0,+\infty)$ and satisfies $\underline \xi'(t)\le 0$ in $(0,+\infty)$, which will be made clear in the sequel. Here $\widetilde{w}_j(t)$ is defined out of $w_j(t)$ as follows. We first define $\chi(t)=\sup\left\{ z \geq1 ~|~ w_k(t)>0 \text{ for all } 1\leq k-c_*t  \leq z\right\}$, and set $\widetilde{w}_j(t)=w_j(t)\mathbbm{1}_{\left\{j\in\Z~|~ 1 \le j-c_*t\le \chi(t)\right\}}$.

Again, we define the sequence $p(t)=(p_j(t))_{j\in\Z}$  as
\bqs
p_j(t)=\frac{1}{(1+t)^{\frac{3}{2}-\beta}}\cos\left(\frac{j-c_*t}{(1+t)^\alpha}\right), \quad t>0, \quad j\in\Z.
\eqs 

Let now verify that $\underline v_j(t)$ is a subsolution of \eqref{v} for $t$ large enough and $j\in\Z$ with $j-c_*t\ge t^\delta$.  We first note that, since $f\in\mathscr{C}^2([0,1])$, there exist $M>0$ and $s_0\in(0,1)$ such that $f(s)-f'(0)s\ge -Ms^2$ for $s\in [0,s_0)$. Gathering this with the linear extension of $f$ on $(-\infty,0)$, one deduces, for $t$ large enough and $j\in\Z$ with $j-c_*t\ge t^\delta$,
\begin{align*}
	\mathcal{R}_j(t;s)= &f'(0)s-e^{\lambda_*(j-c_*t)}f(e^{-\lambda_*(j-c_*t)}s)\\
	=& e^{\lambda_*(j-c_*t)}\left(f'(0)e^{-\lambda_*(j-c_*t)}s-f(e^{-\lambda_*(j-c_*t)}s)	\right),\\
	\le & M e^{-\lambda_*(j-c_*t)}s^2\le M e^{-\lambda_*t^\delta}s^2
\end{align*}
as long as $e^{-\lambda_*t^\delta}s\in (0,s_0)$ for $t$ large enough and $j\in\Z$ with $j-c_*t\ge t^\delta$, while $\mathcal{R}_j(t;s)=0$ when $s\le 0$. We shall require  $\underline v_j(t)$ to satisfy $e^{-\lambda_*t^\delta} \underline v_j(t)<s_0$ in  $\{j\in\Z~|~j-c_*t\ge t^\delta\}$ for $t$ large enough, which is possible due to the asymptotics of $w_j(t)$ as well as our assumptions on $\underline\xi(t)$ and on the parameters.


As proceeded in the previous section, we start with the region $\{j\in\Z~|~j-c_*t\ge \frac{3\pi}{2}(1+t)^\alpha\}$ where $\underline v_j(t)=\underline\xi(t) \widetilde{w}_j(t)$. From Lemma~\ref{lemAsymptLin} and \eqref{j_0}, for $t$ large enough, we have that $\frac{3\pi}{2}(1+t)^\alpha\leq \chi(t)$ such that we define two regions
\bqs
 R_1:= \left\{j\in\Z~|~ \frac{3\pi}{2}(1+t)^\alpha\leq j-c_*t\leq  \chi(t) \right\} \text{ and } R_2:= \left\{j\in\Z~|~ \chi(t)< j-c_*t\right\}.
 \eqs 
\paragraph{In region $R_1$.} By definition, for all $j\in R_1$, we have $\underline v_j(t)=\underline\xi(t) \widetilde{w}_j(t)=\underline\xi(t)w_j(t)>0$. Let us denote $j_1$ the rightmost integer in $R_1$. Then, for all $j\in R_1\backslash\left\{j_1\right\}$, we get
\begin{align*}
	\frac{\md}{\md t}\underline{v}_j(t)-\left(\mathscr{L}\underline{v}(t)\right)_j+\mathcal{R}_j(t; \underline v_j(t))&\le 	\frac{\md}{\md t}\underline{v}_j(t)-\left(\mathscr{L}\underline{v}(t)\right)_j+ M e^{-\lambda_*t^\delta}(\underline v_j(t))^2\\
	&=  \underline\xi'(t) w_j(t)+M e^{-\lambda_* t^\delta}\underline\xi^2(t)(w_j(t))^2\\
	&=  \left( \underline\xi'(t)+M\underline\xi^2(t) e^{-\lambda_*t^\delta} w_j(t)\right) w_j(t).
\end{align*}
 Taking into account the boundedness of $w_j$ from \eqref{bdd-w}, we see that 
 \begin{equation}
 	\label{bound-w}
 	e^{-\lambda_* t^\delta} w_j(t)\le C(1+t)^{-2},
 \end{equation}
 with some $C>0$, for all $t$ large enough and $j\in\Z$ with  $j-c_*t\ge t^\delta$. It is then sufficient for $\underline\xi(t)$ to solve 
 \begin{equation}
 	\label{ODE}
 	\underline\xi'(t)=-CM\underline\xi^2(t)(1+t)^{-2}~~~~~\text{for}~t>0.
 \end{equation}
From \eqref{bound-w} and \eqref{ODE}, one has that $	\frac{\md}{\md t}\underline{v}_j(t)-\left(\mathscr{L}\underline{v}(t)\right)_j+\mathcal{R}_j(t; v_j(t))\le 0$ for $t$ large enough and $j\in R_1\backslash\left\{j_1\right\}$. Now, at $j=j_1$, we first notice that
\bqs
\frac{\md}{\md t}\underline{v}_{j_1}(t)-\left(\mathscr{L}\underline{v}(t)\right)_{j_1}= \underline{\xi}'(t)w_{j_1}(t)+e^{-\lambda_*}\underline{\xi}(t)w_{j_1+1}(t) \leq \underline{\xi}'(t)w_{j_1}(t),
\eqs
since $w_{j_1+1}(t)\leq 0$ by definition of $\chi(t)$. As a consequence, we deduce that
\bqs
\frac{\md}{\md t}\underline{v}_{j_1}(t)-\left(\mathscr{L}\underline{v}(t)\right)_{j_1}+\mathcal{R}_{j_1}(t; \underline v_{j_1}(t))\leq 0,
\eqs
for $t$ large enough.
\paragraph{In region $R_2$.} For each $j\in R_2$, we first note that $\underline v_j(t)=0$. Then, we denote $j_2$ the leftmost integer in $R_2$ and we readily obtain that
 \bqs
\frac{\md}{\md t}\underline{v}_{j}(t)-\left(\mathscr{L}\underline{v}(t)\right)_{j}+\mathcal{R}_{j}(t; \underline v_{j}(t))= 0,
\eqs
for $t$ large enough and all $j\in R_2\backslash\left\{j_2\right\}$. Now, at $j=j_2$, we simply have
\bqs
\frac{\md}{\md t}\underline{v}_{j_2}(t)-\left(\mathscr{L}\underline{v}(t)\right)_{j_2}+\mathcal{R}_{j_2}(t; \underline v_{j_2}(t))= -e^{\lambda_*}\underline{\xi}(t)w_{j_2-1}(t) <0,
\eqs
for $t$ large enough.

 Next, we divide the remaining region into two zones: 
 \bqs
 R_3:= \left\{j\in\Z~|~ t^\delta\le j-c_*t\le  (1+t)^\beta\right\} \text{ and } R_4:= \left\{j\in\Z~|~ (1+t)^\beta\le j-c_*t\le  \frac{3\pi}{2}(1+t)^\alpha\right\},
 \eqs 
 and quickly check that $\underline v_j(t)$ is a sub-solution in these regions too, the computations being similar as in the previous cases.

\paragraph{In region $R_3$.}
Thanks to our choice of $\delta$, $\beta$ and $\alpha$, by requiring $\underline\xi(0)>0$ to be sufficiently small, we derive the following estimate
\begin{equation}
	\label{BC-lower}
	\underline v_j(t)\lesssim \frac{j-c_*t}{(1+t)^{\frac{3}{2}}}-\frac{1}{(1+t)^{\frac{3}{2}-\beta}}\le 0 ~~~~~\text{for}~t~\text{large enough},
\end{equation} 
 which then implies $\mathcal{R}_j(t;\underline v_j(t))=0$.
 Since $w_j(t)>0$ for $t$ large enough in $R_3$ and since $\underline\xi'(t)\le 0$ in $(0,+\infty)$, we have
 \begin{equation*}
 	\frac{\md}{\md t}(\underline \xi(t)w_j(t))-\left(\mathscr{L}[\underline{\xi}(t)w(t)]\right)_j=\underline \xi'(t)w_j(t) 
 	\le 0 ~~~~~\text{for}~t~\text{large enough}.
 \end{equation*}
  Moreover, it follows from a direct calculation that
\begin{equation*}
		\frac{\md}{\md t}p_j(t)-\left(\mathscr{L}p(t)\right)_j \sim \frac{1}{(1+t)^{\frac{3}{2}-\beta+2\alpha}}~~~~~\text{for}~t~\text{large enough}.
\end{equation*}
We eventually get, in  region $R_3$,
\begin{align*}
	\frac{\md}{\md t}\underline{v}_j(t)-\left(\mathscr{L}\underline{v}(t)\right)_j
\sim \underline \xi'(t)w_j(t)-\frac{1}{(1+t)^{\frac{3}{2}-\beta+2\alpha}}< 0~~~~~\text{for}~t~\text{large enough}.
\end{align*}

\paragraph{In region $R_4$.}
If $\underline v_j(t)\le 0$, that is, $j-c_*t\sim (1+t)^\beta$, then $\mathcal{R}_j(t;\underline v_j(t))=0$ and the analysis in the previous case shows that $\frac{\md}{\md t}\underline{v}_j(t)-\left(\mathscr{L}\underline{v}(t)\right)_j\le 0$ for $t$ large enough. Now, it is left to discuss the situation that $\underline v_j(t)>0$ for $t$ large enough in this area $R_4$. We note here that 
\begin{equation*}
 0<\underline v_j(t)\lesssim \frac{1}{(1+t)^{\frac{3}{2}-\alpha}}~~~\text{for}~t~\text{large enough}.
\end{equation*}
It is obvious to see that  $e^{-\lambda_*t^\delta}\frac{1}{(1+t)^{3-2\alpha}}\ll \frac{1}{(1+t)^{\frac{3}{2}-\beta+2\alpha}}$  for $t$ large enough, which will imply $e^{-\lambda_*t^\delta}(\underline v_j(t))^2\ll \frac{1}{(1+t)^{\frac{3}{2}-\beta+2\alpha}}$  for all $t$ large in this area, whence , for $t$ large enough in $R_4$,
\begin{align*}
		\frac{\md}{\md t}\underline{v}_j(t)-\left(\mathscr{L}\underline{v}(t)\right)_j+\mathcal{R}_j(t; \underline v_j(t))&\le 	\frac{\md}{\md t}\underline{v}_j(t)-\left(\mathscr{L}\underline{v}(t)\right)_j+ M e^{-\lambda_*t^\delta}(\underline v_j(t))^2\\
	&\sim  \frac{\underline\xi'(t)}{(1+t)^{\frac{3}{2}-\beta}}-\frac{1}{(1+t)^{\frac{3}{2}-\beta+2\alpha}}\le 0.
\end{align*}

\textbf{Conclusion}.
We require that $\underline\xi(t)$ is the solution of the ODE \eqref{ODE} starting from a sufficiently small initial datum $\underline\xi(0)=\underline\xi_0>0$. Then, $\underline\xi(t)$ is positive and uniformly bounded from above and below in $[0,+\infty)$ such that
$$0<\frac{\underline\xi_0}{1+\underline\xi_0CM}\le \underline\xi(t)\le \underline\xi_0<+\infty~~~~~\text{for}~t\ge 0.$$
There is $t_0>0$  large enough such that $e^{-\lambda_*t^\delta}\frac{1}{(1+t)^{3-2\alpha}}< \frac{1}{(1+t)^{\frac{3}{2}-\beta+2\alpha}}$ for $t\ge t_0$ and such that  $e^{-\lambda_*t^\delta}\underline v_j(t)<s_0$ and $e^{-\lambda_*t^\delta} w_j(t)\le C(1+t)^{-2}$ for all $t\ge t_0$ and $j\in\Z$ with $j\geq \zeta(t)$ where we have set $\zeta(t):=c_*t+ t^\delta$. Moreover,  it follows from above analysis that $\frac{\md}{\md t}\underline{v}_j(t)-\left(\mathscr{L}\underline{v}(t)\right)_j+\mathcal{R}_j(t;\underline v_j(t))\le 0$ for $t\ge t_0$ and $j\in\Z$ with $j\geq\zeta(t)$.
 Now let us prove that, at time $t=t_0$, there is $\kappa>0$ small enough such that $v_j(t_0+1)\geq \kappa\underline v_j(t_0)$ for $j\in\Z$ with $j\geq \zeta(t_0)$. As a matter of fact, we notice that the sequence function $z_j(t):=e^{-f'(0)t}w_j(t)$  is the solution of
\begin{equation}
	\label{linear eqn}
	\frac{\md}{\md t}z_j(t)-(\mathscr{L}z(t))_j+f'(0)z_j(t)= 0~~~~t>0,~j\in\Z,
\end{equation}
with compactly supported initial value $z_j(0)=w_j(0)$, while $v_j(t)$ satisfies \eqref{linear eqn} with ``='' replaced by ``$\ge$''. Moreover, there exists $\kappa_1>0$ small enough such that  $$v_j(1)=e^{\lambda_*(j-c_*)}u_j(1)>\kappa_1w_j(0)=\kappa_1 z_j(0)~~~\text{for}~j\in \Z.$$
The comparison principle immediately yields that $v_j(t+1)>\kappa_1 e^{-f'(0)t}w_j(t)$ for all $t\ge 0$ and $j\in\Z$. In particular, we have $v_j(t_0+1)>\kappa_1 e^{-f'(0)t_0}w_j(t_0)$ for $j\in\Z$ with $j\geq\zeta(t_0)-1$. We can then choose $\kappa>0$ small enough such that $\kappa_1 e^{-f'(0)t_0}w_j(t_0)>\kappa\underline\xi_0 w_j(t_0)$ for $j\in\Z$ with $j\geq\zeta(t_0)-1$ and such that particularly $\kappa_1 e^{-f'(0)t_0}w_j(t_0)>\kappa\left(\underline\xi_0 w_j(t_0)+(1+t_0)^{-3/2+\beta}\right)$ for $j\in\Z$ with $c_*t_0+\frac{\pi}{2}(1+t_0)^\alpha\le j \le c_*t_0+\frac{3\pi}{2}(1+t_0)^\alpha$. This implies that $$v_j(t_0+1)>\kappa\underline v_j(t_0)~~~\text{for}~j\in\Z ~\text{with}~j\geq\zeta(t_0)-1.$$
For $j\in[\zeta(t)-1,\zeta(t))$, we have, up to increasing $t_0$ if necessary, $\kappa\underline v_j(t)< 0$ for all $t\ge t_0$ by \eqref{BC-lower} and  $v_j(t+1)>e^{\lambda_*t^\delta}u_j(t+1)>0$  for all $t\ge t_0$, which implies that $v_j(t+1)>\kappa\underline v_j(t)$ for all $t\ge t_0$ and $j\in\Z$ with $j\in[\zeta(t)-1,\zeta(t))$.
 Therefore, $\kappa\underline v_j(t)$ is a subsolution of \eqref{v} for $t\ge t_0$ and $j\in\Z$ with $ j\ge \zeta(t)$.  The comparison principle Proposition~\ref{cp} yields that 
 \begin{equation}
 	\label{512-conclusion}
  v_j(t)\geq \kappa \underline v_j(t-1)~~~~\text{for}~t\ge t_0+1,~j\in\Z~\text{with}~ j-c_*t\geq t^\delta.
 \end{equation}

\subsection{Proof of Theorem \ref{thmlog}}
Using the upper and lower barriers of $v_j$ constructed in preceding sections as  key ingredients, we are now in position to prove Theorem \ref{thmlog}, which gives a refined estimate, up to $O(1)$ precision, of the level sets for initially localized solutions $u_j(t)$  to \eqref{KPP} for large times.

\begin{proof}[Proof of Theorem \ref{thmlog}]
 The proof is based on the comparison between $v_j(t)$ and a variant of the shifted minimal traveling front  with logarithmic correction  for all large times in the moving zone $|j-c_*t|\le t^\eta$ for some small $\eta$.  Define 
 \begin{equation*}
 	V_j(t)=t^{\frac{3}{2}}v_j(t)~~~\text{for}~t\ge 1,~j\in\Z.
 \end{equation*} 
Then, $V_j(t)$ satisfies
 \begin{equation*}
 	\begin{aligned}
 		\begin{cases}
 			\frac{\md}{\md t} V_j(t) =e^{\lambda_*}\! \left(V_{j-1}(t)\!-\!2V_j(t)\!+\!V_{j+1}(t)\right)\!-\!c_*(V_{j+1}(t)\!-\!V_j(t))\! +\frac{3}{2t}V_j(t)\!-\!\widehat{\mathcal{R}}_j(t;V_j(t)), &t>1,~j\in\Z,\\
 			V_j(1)  =v_j(1), & j\in\Z,
 		\end{cases}
 	\end{aligned}
 \end{equation*}
 with nonnegative term $\widehat{\mathcal{R}}_j(t;s)$ given explicitly by 
 \begin{equation}
 	\label{hat R-nonnegative}
 	\widehat{\mathcal{R}}_j(t;s):=f'(0)s-e^{\lambda_*(j-c_*t+\frac{3}{2\lambda_*}\ln t)}f(e^{-\lambda_*(j-c_*t+\frac{3}{2\lambda_*}\ln t)}s), ~~~t\ge 1,~j\in\Z, ~s\in\R.
 \end{equation} 
  Let $\delta$, $\beta$ and $\alpha$ be chosen as in \eqref{parameters}, which also implies that $\beta<1/4$. Fix now
 \begin{equation}
 \label{eta}
 \eta=\beta+\varepsilon<\alpha
 \end{equation} 
  for some $\varepsilon>0$ small enough.

\paragraph{Step 1: Upper bound.}    We notice from \eqref{511-conclusion} that  $t^{\frac{3}{2}}\overline v_j(t+T_0)\geq V_j(t)$ for $t\ge 0$ and $j\in\Z$ with $j-c_*t\geq-t^\delta$, with $\overline v_j(t)$ given in \eqref{overline v}. This implies that there exists some constant $C>0$ such that, for $t$ large enough,
\bqs
V_j(t) \leq t^{\frac{3}{2}}\overline v_j(t+T_0) \leq C t^\eta~~~~~~~\text{for}~ \xi(t)<j\leq \xi(t)+1,
\eqs
where we have set $\xi(t):=c_*t+t^\eta -\frac{3}{2\lambda_*}\ln t$.
  Define now the sequence $\psi_j(t)$  by
\begin{equation*}
\psi_j(t)=e^{\lambda_*(j-c_*t +\frac{3}{2\lambda_*}\ln t})U_{c_*}\left(j-c_*t+\frac{3}{2\lambda_*}\ln t+b\right)~~~~\text{for}~ \zeta(t)-1\leq j \leq \xi(t)+1,
\end{equation*}
for $t$ large enough, with $\zeta(t):=c_*t-t^\eta -\frac{3}{2\lambda_*}\ln t$ and where $b\in\R$ is fixed such that $\psi_j(t) \geq t^{\frac{3}{2}}\overline v_j(t+T_0)$ for $t$ large enough and $\xi(t)<j\leq \xi(t)+1$. This is always possible thanks to our normalization for the minimal traveling front $U_{c_*}$ in \eqref{normalization} which ensures that, for $t$ large enough,
\bqs
\psi_j(t) \sim e^{-\lambda_*b}t^{\eta}~~~~~~~\text{for}~\xi(t)<j\leq \xi(t)+1.
\eqs
 Substituting $\psi_j(t)$ into the equation of $V_j(t)$ leads to
\begin{align*}
\left|\frac{\md}{\md t}\psi_j(t)-e^{\lambda_*}\right.&\left.\left(\psi_{j-1}(t)-2\psi_j(t)+\psi_{j+1}(t)\right)+c_*\left(\psi_{j+1}(t)-\psi_j(t)\right) -\frac{3}{2t}\psi_j(t)+\widehat{\mathcal{R}}_j(t;\psi_j(t))\right|\\
&=\frac{3}{2\lambda_* t} e^{\lambda_*(j-c_*t+\frac{3}{2\lambda_*}\ln t)}\left|U'_{c_*}\left(j-c_*t+\frac{3}{2\lambda_*}\ln t+b\right)\right|\lesssim t^{-(1-\eta)}\,,
\end{align*}
for $t$ large enough and $j\in\Z$ with $\zeta(t)\leq j \leq \xi(t)$. Now, set $s_j(t):=\left(V_j(t)-\psi_j(t)\right)^+$, then $s_j(t)$ satisfies
\begin{equation}\label{s+}
	\begin{aligned}
		\begin{cases}
			\!	\frac{\md}{\md t}s_j(t)\! -\! \left(\mathscr{L}s(t)\right)_j\! -\frac{3}{2t}s_j(t) \!+\!\mathcal{Q}_j(t;\!s_j(t)) \!\lesssim \!\frac{1}{t^{1-\eta}},\!\! \!\!\!\!&~~~\zeta(t)\leq j \leq \xi(t),\\
			\!s_j(t)  = O\big(t^{\frac{3}{2}}e^{-\lambda_*t^\eta}\big), ~~~~~~  &~~~\zeta(t)-1\leq j <\zeta(t),\\
			\!s_j(t)  =0, ~~~~~~~&~~~\xi(t)<j\leq \xi(t)+1,
		\end{cases}
	\end{aligned}
\end{equation}
for $t$ large enough.
Here, $\mathcal{Q}_j(t;s_j(t))=0$ if $s_j(t)=0$, and otherwise,
\begin{align*}
\mathcal{Q}_j(t;s_j(t))&=\widehat{\mathcal{R}}_j(t;V_j(t))-\widehat{\mathcal{R}}_j(t;\psi_j(t))\\
&=f'(0)s_j(t)-e^{\lambda_*(j-c_*t+\frac{3}{2\lambda_*}\ln t)}\left(f(e^{-\lambda_*(j-c_*t+\frac{3}{2\lambda_*}\ln t)}V_j(t))-f(e^{-\lambda_*(j-c_*t+\frac{3}{2\lambda_*}\ln t)}\psi_j(t))\right)\\
&=f'(0)s_j(t)-b_j(t)s_j(t)\ge0, 
\end{align*} 
in which $b_j$ is a continuous function, and all $b_j$'s are bounded in $\ell^\infty$ norm by $f'(0)$
since $0<f(s)\le f'(0)s$ for $s\in(0,1)$. We claim that, there holds
\begin{equation}
\label{claim1}
\lim\limits_{t\to+\infty}\sup_{j\in\Z,~|j-c_*t +\frac{3}{2\lambda_*}\ln t|\le t^\eta} s_j(t)=0.
\end{equation}
We use a comparison argument to verify this. Define 
\begin{equation*}
~~~\overline s_j(t)=\frac{1}{t^\lambda}\cos\left(\frac{j-c_*t}{t^\gamma}\right)~~~~~~~\text{for}~t~\text{large enough and}~\zeta(t)-1\leq j \leq \xi(t)+1,
\end{equation*}
 with $0<\eta<1/4<\gamma<1/3$  such that  $2\gamma+\eta<1$ and with $0<\lambda<1-2\gamma-\eta$.  We notice that $\overline s_j(t)\sim t^{-\lambda}\gg t^{\frac{3}{2}}e^{-\lambda_*t^\eta}$ for $t$ large enough and $j\in\Z$ with $\zeta(t)-1\leq j \leq \xi(t)+1$. Through a direct computation, one also gets that, for $t$ large enough and $j\in\Z$ with $\zeta(t)\leq j \leq \xi(t)$,
\begin{equation*}
\frac{\md}{\md t}\overline s_j(t) 	 - \left(\mathscr{L}\overline s(t)\right)_j  -\frac{3}{2t}\overline s_j(t)\sim \frac{1}{t^{2\gamma+\lambda}}\gg \frac{1}{t^{1-\eta}}.
\end{equation*} 
Since $\mathcal{Q}_j(t;\overline s_j(t))$ is nonnegative, $\overline s_j(t)$ is a supersolution of \eqref{s+} for $t$ large enough and $j\in\Z$ with $\zeta(t)\leq j \leq \xi(t)$.
Our claim \eqref{claim1} is then reached by noticing that
$$\lim\limits_{t\to+\infty}\sup_{j\in\Z,~|j-c_*t +\frac{3}{2\lambda_*}\ln t|\le t^\eta} \overline s_j(t)=0.$$
Consequently, one gets $V_j(t)\le \psi_j(t)+o(1)$ uniformly in $j\in\Z$ with $|j-c_*t +\frac{3}{2\lambda_*}\ln t|\le t^\eta$ as $t\to+\infty$. This implies
\begin{equation}
\label{u-upper bound}
u_j(t)\le U_{c_*}\left(j-c_*t+\frac{3}{2\lambda_*}\ln t+b\right)+o(1)e^{-\lambda_* (j-c_*t+\frac{3}{2\lambda_*}\ln t)},
\end{equation}
uniformly in $j\in\Z$ with $1\le j-c_*t  +\frac{3}{2\lambda_*}\ln t \le t^\eta$ as $t\to+\infty$. 

\paragraph{Step 2: Lower bound.} The proof of this part is similar to Step 1. We sketch it for the sake of completeness.  By virtue of \eqref{512-conclusion}, we deduce that  $V_j(t)\geq \kappa t^{\frac{3}{2}}\underline v_j(t-1)$  for $t\ge t_0+1$ and $j\in\Z$ with $j-c_*t\ge t^\delta$, where $\underline v_j(t)$ is given in \eqref{underline v}.    We then infer that there exists some constant $C>0$ such that for $t$ large enough,
\bqs
V_j(t) \geq \kappa t^{\frac{3}{2}}\underline v_j(t-1) \geq C t^\eta~~~~~~~\text{for}~\xi(t)<j\leq \xi(t)+1.
\eqs
Define the sequence $\phi_j(t)$  by
\begin{equation*}
\phi_j(t)=e^{\lambda_*(j-c_*t+\frac{3}{2\lambda_*}\ln t)}U_{c_*}\left(j-c_*t+\frac{3}{2\lambda_*}\ln t+a\right).
\end{equation*} 
for $t$ large enough and $\zeta(t)-1\leq j \leq \xi(t)+1$.
Here, we fix $a\in\R$  such that 
	$\phi_j(t)\leq \kappa t^{\frac{3}{2}}\underline v_j(t-1)$ for $t$ large enough and $\xi(t)<j\leq \xi(t)+1$. It is also noticed that necessarily $a>b$. 
Substituting $\phi_j(t)$ into the equation of $V_j(t)$ yields
\begin{equation*}
\left|\frac{\md}{\md t}\phi_j(t)-\left(\mathscr{L} \phi(t)\right)_j  -\frac{3}{2t}\phi_j(t) +\widehat{\mathcal{R}}_j(t;\phi_j(t))\right|\lesssim t^{-(1-\eta)}
\end{equation*}
for $t$ large enough and $j\in\Z$ with $\zeta(t)\leq j \leq \xi(t)$.
Set $z_j(t):=\left(V_j(t)-\phi_j(t)\right)^-$, then $z_j(t)$ satisfies
\begin{equation*}
	\begin{aligned}
		\begin{cases}
			\frac{\md}{\md t}z_j(t) - \!\left(\mathscr{L}z(t)\right)_j  -\frac{3}{2t}z_j(t) \!+\!\mathcal{H}_j(t;s_j(t))\lesssim \frac{1}{t^{1-\eta}},\!\!\!&~~~\zeta(t)\leq j \leq \xi(t),\\
			z_j(t)  = O\big(t^{\frac{3}{2}}e^{-\lambda_*t^\eta}\big), ~~~~~~  &~~~\zeta(t)-1\leq j < \zeta(t),\\
			z_j(t)  =0, ~~~~~~~&~~~\xi(t)< j \leq \xi(t)+1,
		\end{cases}
	\end{aligned}
\end{equation*}
for $t$ large enough.
Here, $\mathcal{H}_j(t;z_j(t))=0$ when $z_j(t)=0$; otherwise,
\begin{align*}
\mathcal{H}_j(t;z_j(t))&=\widehat{\mathcal{R}}_j(t;V_j(t))-\widehat{\mathcal{R}}_j(t;\phi_j(t))\\
&=f'(0)z_j(t)-e^{\lambda_*(j-c_*t +\frac{3}{2\lambda_*}\ln t)}\left(f(e^{-\lambda_*(j-c_*t+\frac{3}{2\lambda_*}\ln t)}V_j(t))-f(e^{-\lambda_*(j-c_*t+\frac{3}{2\lambda_*}\ln t)}\phi_j(t))\right)\\
&=f'(0)z_j(t)-d_j(t)z_j(t)\ge0, 
\end{align*} 
in which $d_j$ is a continuous function, and all $d_j$'s are bounded in $\ell^\infty$ norm by $f'(0)$
since $0<f(s)\le f'(0)s$ for $s\in(0,1)$. Following the proof of \eqref{claim1} in Step 1, one can show that
\begin{equation*}
\lim\limits_{t\to+\infty}\sup_{j\in\Z,~|j-c_*t +\frac{3}{2\lambda_*}\ln t|\le t^\eta} z_j(t)=0.
\end{equation*}
 It then follows that $V_j(t)\ge \phi_j(t)+o(1)$ uniformly in $j\in\Z$ with $|j-c_*t  +\frac{3}{2\lambda_*}\ln t|\le t^\eta$ as $t\to+\infty$, whence
\begin{equation}
\label{u-lower bound}
u_j(t)\ge U_{c_*}\left(j-c_*t+\frac{3}{2\lambda_*}\ln t+a\right)+o(1) e^{-\lambda_* (j-c_*t+\frac{3}{2\lambda_*}\ln t)},
\end{equation}
uniformly in  $j\in\Z$ with $1\le j-c_*t  +\frac{3}{2\lambda_*}\ln t\le t^\eta$ as $t\to+\infty$.

\paragraph{Step 3: Conclusion.} Combining \eqref{u-upper bound} and \eqref{u-lower bound}, 
along with the asymptotics of $U_{c_*}$, it follows that for any small $\varepsilon>0$, there exists $T>0$ sufficiently large and   $\hat a, \hat b\in\R$ satisfying $-\infty<\hat b<b<a<\hat a<+\infty$ such that
\begin{equation}\label{5-conclusion}
 	(1-\varepsilon)U_{c_*}\left(j-c_*t+\frac{3}{2\lambda_*}\ln t+\hat a\right)\le u_j(t)\le 	(1+\varepsilon) U_{c_*}\left(j-c_*t+\frac{3}{2\lambda_*}\ln t+\hat b\right)
\end{equation}
uniformly in  $j\in\Z$ with $1\le j-c_*t +\frac{3}{2\lambda_*}\ln t\le t^\eta$ for $t\ge T$. We therefore conclude that, for every $m\in(0,1)$, there exists $C\in\R$ such that 
$$ j_m(t)\subset\left\{j\in\Z~|~c_*t-\frac{3}{2\lambda_*}\ln t-C\le j\le  c_*t-\frac{3}{2\lambda_*}\ln t+C\right\}~~~\text{for all}~t~\text{large enough}.$$
The proof of Theorem \ref{thmlog} is thereby complete.
\end{proof}

\section{Convergence to the logarithmically shifted critical pulled front}\label{secconv}

This section is devoted to the proof of Theorem \ref{thmconv}. As an immediate conclusion from Theorem \ref{thmlog}, we know that the transition zone of $u_j(t)$ between  the two equilibria 0 and 1 is located around the position $c_*t-\frac{3}{2\lambda_*}\ln t$ for $t$ large enough. Also, it is crucial to note that the proof of Theorem \ref{thmlog} shows in particular that $u_j(t)$ is indeed sandwiched between two finitely shifted minimal traveling fronts with logarithmic delay as $t\to+\infty$ in a well chosen moving zone. Therefore, the Liouville type result, Proposition 3.3 of \cite{GH2006}, can be directly applied and we then follow the strategy proposed in the continuous case \cite{HNRR13} to accomplish our proof. Finally, throughout the section, for $x\in\R$ the integer part of $x$ will be denoted as $\lfloor x \rfloor \in \Z$.


\begin{proof}[Proof of Theorem \ref{thmconv}.] 
	Let $C>0$ be such that $-C<b<a<C$. Assume by contradiction that \eqref{convfront} is not true, then there exist $\varepsilon>0$ and a sequence $(t_n)_{n\in\N}$ such that $t_n\to+\infty$ as $n\to+\infty$ and 
	\begin{equation*}
	\min_{|\zeta|\le C}\sup_{j\in\Z,~j\ge 0}\left|u_j(t_n)-U_{c_*}\left(j-c_*t_n+\frac{3}{2\lambda_*}\ln t_n+\zeta\right)\right|\ge \varepsilon
	\end{equation*}
	for all $n\in\N$. Since  $U_{c_*}(-\infty)=1$ and $U_{c_*}(+\infty)=0$, together with properties
	\eqref{spreading1refined} and \eqref{spreading2refined}, there exists $L>0$ such that 
	\begin{equation}
	\label{contradiction}
	\min_{|\zeta|\le C}\max_{j\in\Z,~|j|\le L}\left|u_{j+j_n}(t_n)-U_{c_*}(j+\zeta)\right|\ge \varepsilon~~~\text{with}~~j_n=\left\lfloor c_*t_n-\frac{3}{2\lambda_*}\ln t_n\right\rfloor
	\end{equation}
	for all $n\in\N$. Up to extraction of a subsequence, the functions $(u_n)_j(t)=u_{j+j_n}(t+t_n)$ with $j_n=\left\lfloor c_*t_n-\frac{3}{2\lambda_*}\ln t_n\right\rfloor$ converge, in $C^1_{loc}(\R)$ for each  $j\in\Z$, to a solution $(u_\infty)_j(t)$ of 
	\begin{equation*}
	\frac{\md}{\md t}(u_\infty)_j(t)=(u_\infty)_{j-1}(t)-2(u_\infty)_j(t)+(u_\infty)_{j+1}(t)+f((u_\infty)_j(t))~~~\text{for}~(t,j)\in\R\times\Z,
	\end{equation*}
	such that $0\le (u_\infty)_j(t)\le 1$ in $\R\times\Z$. Furthermore, \eqref{spreading1refined} and \eqref{spreading2refined} imply that 
	\begin{equation}
	\lim_{y\to+\infty} \inf_{(t,j)\in\R\times\Z,\atop j\le c_*t-y} (u_\infty)_j(t)=1,~~~	\lim_{y\to+\infty} \sup_{(t,j)\in\R\times\Z,\atop j\geq c_*t+y} (u_\infty)_j(t)=0.
	\end{equation}
	On the other hand, fix $t\in\R$ and $j\in\Z$  with $j\ge 1$, we then have $j+\frac{3}{2\lambda_*}\ln \big((t+t_n)/t_n\big)\ge 1$ for $n$ large enough. Moreover, notice also that $t+t_n\ge 1$ and $1\le j+ \frac{3}{2\lambda_*}\ln \big((t+t_n)/t_n\big)\le (t+t_n)^\eta$  for $n$ large enough, with $\eta$ chosen in \eqref{eta}. Hence, it follows from \eqref{5-conclusion} that, for any small $\varepsilon>0$,
	\begin{equation*}
	(1-\varepsilon)U_{c_*}\Big(j+\frac{3}{2\lambda_*}\ln \big((t+t_n)/t_n\big)+\hat a\Big)\le (u_n)_{j+\lfloor c_*t\rfloor}(t)\le (1+\varepsilon) U_{c_*}\Big(j+\frac{3}{2\lambda_*}\ln \big((t+t_n)/t_n\big)+\hat b\Big)
	\end{equation*}
	for all $n$ large enough. Therefore,
	\begin{equation*}
	U_{c_*}(j+\hat a)\le (u_\infty)_{j+\lfloor c_*t\rfloor}(t)\le U_{c_*}(j+\hat b)~~~~~\text{for}~t\in\R~\text{and}~j\ge 1.
	\end{equation*} 
	Applying the Liouville-type result \cite[Proposition 3.3]{GH2006} by taking any positive integer $N$ as the period, one then gets the existence of $z\in\R$ such that $\hat b\le z\le \hat a$ and $$(u_\infty)_j(t)=U_{c_*}(j-c_*t+z)~~~~\text{for}~ (t,j)\in\R\times\Z.$$
	Since $u_n$ converges, up to extraction of a subsequence, to $u_\infty$ locally uniformly in $\R\times\Z$, it follows in particular that $(u_n)_j(0)-U_{c_*}(j+z)\to 0$ uniformly in $j\in\Z$ with $|j|\le L$, that is,
	\begin{equation*}
\max_{j\in\Z,~|j|\le L}\left|u_{j+j_n}(t_n)-U_{c_*}(j+z)\right|\to 0~~~\text{with}~~j_n=\left\lfloor c_*t_n-\frac{3}{2\lambda_*}\ln t_n\right\rfloor.
\end{equation*}
	Notice that $z\in [\hat b,\hat a]\subset[-C,C]$, one then gets a contradiction with \eqref{contradiction}. This completes the proof of \eqref{convfront}.
	
	Fix now any $m\in(0,1)$ and let $(t_n)_{n\in\N}$ and $(j_n)_{n\in\N}$ be sequences of positive real numbers and of positive integers, respectively, such that $t_n\to+\infty$ as $n\to+\infty$ and $u_{j_n}(t_n)=m$ for all $n\in\N$.
	Set 
	$$\xi_n=j_n-\left\lfloor c_*t_n-\frac{3}{2\lambda_*}\ln t_n\right\rfloor,$$
	Theorem \ref{thmlog} implies that the sequence $(\xi_n)_{n\in\N}$ is bounded and then, up to extraction of a subsequence, $\xi_n\to\xi_0\in\Z$ as $n\to+\infty$. From the argument of \eqref{convfront}, the functions 
	\begin{equation*}
	(v_n)_j(t)=u_{j+j_n}(t+t_n)~~~\text{with}~j_n=\xi_n+\left\lfloor c_*t_n-\frac{3}{2\lambda_*}\ln t_n\right\rfloor
	\end{equation*}
	converge, up to another subsequence, locally uniformly in $t$ for all $j\in\Z$, to $(v_\infty)_j(t)=U_{c_*}(j-c_*t+\xi_0+z_0)$ for some $z_0\in[-C,C]$. Since $(v_n)_{j=0}(0)=u_{j_n}(t_n)=m$, one then has $\xi_0+z_0=U^{-1}_{c_*}(m)$. Therefore, the limit $v_\infty$ is uniquely determined and the whole sequence $(v_n)_{n\in\N}$ then converges to the traveling wave $U_{c_*}\left(j-c_*t+U^{-1}_{c_*}(m)\right)$. The conclusion of Theorem~\ref{thmconv} follows.
\end{proof}

\section*{Acknowledgements} 
We are indebted to the referees for an extremely careful reading of an earlier version of the manuscript, which has greatly helped us to improve the presentation. G.F. acknowledges support from the ANR via the project Indyana under grant agreement ANR- 21- CE40-0008 and an ANITI (Artificial and Natural Intelligence Toulouse Institute) Research Chair.  This work was partially supported by the French National Institute of Mathematical Sciences and their Interactions (Insmi) via the platform MODCOV19 and by Labex CIMI under grant agreement ANR-11-LABX-0040.

 \appendix
\section{Maximum and comparison principles}
We recall that for a given sequence $\mathbf{z}=(z_j)_{j\in\Z}$, the linear operator $\mathscr{L}$ acts on $\mathbf{z}$ as
\bqs
(\mathscr{L}\mathbf{z})_j=e^\lambda_*\big(z_{j-1}-2z_j+z_{j+1}\big)-c_*(z_{j+1}-z_j),
\eqs
where the couple $(c_*,\lambda_*)$ is solution to \eqref{systemclamb} and satisfies in particular $c_*=e^{\lambda_*}-e^{-\lambda_*}$.
\begin{prop}[Maximum principle with a single moving boundary] \label{mp}
	Assume that $\mathbf{z}(t)=(z_j(t))_{j\in\Z}$ satisfies 
	\begin{equation*}
		\begin{aligned}
			\begin{cases}
				\frac{\md}{\md t} z_j(t)-(\mathscr{L}\mathbf{z}(t))_j\le 0,&t>0,~j\ge \zeta(t),\\
				z_j(0)\le 0, &j\ge \zeta(0)-1,\\
				z_j(t)\le 0, &t>0,~j\in[\zeta(t)-1,\zeta(t)).
			\end{cases}
		\end{aligned}
	\end{equation*}
	with  $\zeta: \R_+\to\R$ being a continuous function such that $\zeta(t)\ge \zeta(0)-1$ for all $t>0$.  
	Then  $z_j(t)\le 0$ for $t>0$ and $j\ge \zeta(t)$. 
\end{prop}
\begin{proof}[Proof of Proposition \ref{mp}]
 From the equation, we derive that (we obey the convention that $a^+=\max(0,a)$)
 \begin{equation*}
 		\frac{\md}{\md t} z_j(t)+2\cosh(\lambda_*)z_j(t)\le e^{\lambda_*}z^+_{j-1}(t)+e^{-\lambda_*}z^+_{j+1}(t)~~~~~\text{for}~t>0,~j\ge \zeta(t).
 \end{equation*}
This further implies that
\begin{equation*}
	z_j(t)\le z_j(0)e^{-2\cosh(\lambda_*)t}+\int_0^t e^{-2\cosh(\lambda_*)(t-s)}\big(e^{\lambda_*}z^+_{j-1}(s)+e^{-\lambda_*}z^+_{j+1}(s)\big)\md s ~~~~~\text{for}~t>0,~j\ge \zeta(t).
\end{equation*}
By applying the boundary condition, we derive
$$\sup_{j\ge \zeta(t)} z^+_{j-1}(t)\le \sup_{j\ge \zeta(t)}z^+_{j}(t)~~~\text{and}~~~\sup_{j\ge \zeta(t)} z^+_{j+1}(t)\le \sup_{j\ge \zeta(t)}z^+_{j}(t)~~~\text{for}~t>0,$$ 
 whence $$\sup_{j\ge \zeta(t)}\big(e^{\lambda_*}z^+_{j-1}(t)+e^{-\lambda_*}z^+_{j+1}(t)\big)
\le 2\cosh(\lambda_*)\sup_{j\ge \zeta(t)}z^+_j(t)~~~~~\text{for}~t>0.$$
Together with $z_j(0)\le 0$ for $j\ge \zeta(t)\ge \zeta(0)-1$ with $t>0$, we obtain that
\begin{equation*}
	z^+_j(t)\le 2\cosh(\lambda_*)\int_0^t e^{-2\cosh(\lambda_*)(t-s)}\sup_{j\ge \zeta(s)}z^+_j(s)\md s ~~~~~\text{for}~t>0,~j\ge \zeta(t).
\end{equation*}
Therefore, $\sup_{j\ge \zeta(t)}z^+_j(t)\le 2\cosh(\lambda_*)\int_0^t e^{-2\cosh(\lambda_*)(t-s)}\sup_{j\ge \zeta(s)}z^+_j(s)\md s$ for $t>0$. The Gronwall's inequality implies that $\sup_{j\ge \zeta(t)}z^+_j(t)=0$ for all $t>0$, that is, $z_j(t)\le 0$ for $t>0$ and $j\ge \zeta(t)$.
\end{proof}
By adaptation of above maximum principle, we have the following comparison principle.
\begin{prop}[Comparison principle with a single moving boundary]\label{cp}
	Assume that  $\overline v_j(t)$ $($resp. $\underline v_j(t)$$)$ satisfies the equation \eqref{v} with ``='' replaced by ``$\ge$'' $($resp. ``$\le$''$)$ for $t>0$ and $j\ge \zeta(t)$, where the function $\zeta:\R_+\to\R$ is  continuous such that $\zeta(t)\ge \zeta(0)-1$. Moreover, $\overline v_j(0)\ge \underline v_j(0)$ for $j\ge \zeta(0)-1$ and $\overline v_j(t)\ge\underline v_j(t)$ for $t>0$ and $j\in[\zeta(t)-1,\zeta(t))$. Then, $\overline v_j(t)\ge \underline v_j(t)$ for all $t>0$ and $j\ge \zeta(t)$.
\end{prop}
We now turn our attention to proving a maximum principle with two moving boundaries.
\begin{prop}[Maximum principle with two moving boundaries] \label{mp2}
	Assume that $\mathbf{z}(t)=(z_j(t))$, defined for $t>0$ and $j\in[\zeta(t)-1,\xi(t)+1]$, satisfies 
	\begin{equation*}
		\begin{aligned}
			\begin{cases}
				\frac{\md}{\md t} z_j(t)-(\mathscr{L}\mathbf{z}(t))_j\le 0,&t>0,~\zeta(t) \leq j \leq \xi(t),\\
				z_j(0)\le 0, &\zeta(0)-1\leq j \leq \xi(0)+1,\\
				z_j(t)\le 0, &t>0,~j\in[\zeta(t)-1,\zeta(t))\cup (\xi(t),\xi(t)+1].
			\end{cases}
		\end{aligned}
	\end{equation*}
	with  $\zeta: \R_+\to\R$ and $\xi: \R_+\to\R$ being two continuous functions. Then  $z_j(t)\le 0$ for $t>0$ and $\zeta(t) \leq j \leq \xi(t)$. 
\end{prop}
\begin{proof}
We assume by contradiction that there exist some $t_0>0$ and $j_0\in[\zeta(t_0),\xi(t_0)]$ such that $z_{j_0}(t_0)>0$. We let $\Omega_{t_0}=[0,t_0] \times \left\{j\in\Z~|~ j\in[\zeta(t),\xi(t)]\right\}$. Then,  we get the existence of $(t_*,j_*)\in\Omega_{t_0}$ such that
\bqs
z_{j_*}(t_*)=\underset{(t,j)\in\Omega_{t_0}}{\max} z_j(t)>0.
\eqs
There are two cases. Let us first assume that $0<t_*<t_0$, whence $\frac{\md}{\md t} z_{j_*}(t_*)=0$ and the inequation gives
\bqs
-e^{-\lambda_*}\left(z_{j_*+1}(t_*)-z_{j_*}(t_*)\right)-e^{\lambda_*}\left(z_{j_*-1}(t_*)-z_{j_*}(t_*)\right)\le 0.
\eqs
This implies that $z_{j_*}(t_*)=z_{j_*+1}(t_*)=z_{j_*-1}(t_*)=\underset{(t,j)\in\Omega_{t_0}}{\max} z_j(t)>0$ together with $\frac{\md}{\md t} z_{j_*}(t_*)=\frac{\md}{\md t} z_{j_*-1}(t_*)=\frac{\md}{\md t} z_{j_*+1}(t_*)=0$. By induction, if $j_\star$ denotes the leftmost integer in $[\zeta(t_*),\xi(t_*)]$, then we get that $z_{j_\star}(t_*)=z_{j_\star+1}(t_*)>0$ and $\frac{\md}{\md t} z_{j_\star}(t_*)=0$, using once again the inequation, we get that necessarily $0<z_{j_\star-1}(t_*)\leq0$ since $z_{j_\star-1}(t_*)\in[\zeta(t_*)-1,\zeta(t_*))$, which is impossible. If now $t_*=t_0$, we only have that $\frac{\md}{\md t} z_{j_*}(t_*)\geq 0$, but then 
\bqs
0\leq \frac{\md}{\md t} z_{j_*}(t_*) -e^{-\lambda_*}\left(z_{j_*+1}(t_*)-z_{j_*}(t_*)\right)-e^{\lambda_*}\left(z_{j_*-1}(t_*)-z_{j_*}(t_*)\right)\le 0,
\eqs
from which we obtain $\frac{\md}{\md t} z_{j_*}(t_*)=0$ with $z_{j_*}(t_*)=z_{j_*+1}(t_*)=z_{j_*-1}(t_*)=\underset{(t,j)\in\Omega_{t_0}}{\max} z_j(t)>0$. And we can repeat the previous arguments to reach a contradiction.
\end{proof}
By adaptation of above maximum principle, we have the following comparison principle.
\begin{prop}[Comparison principle with two moving boundaries]\label{cp2}
	Assume that  $\overline v_j(t)$ $($resp. $\underline v_j(t)$$)$ satisfies the equation \eqref{v} with ``='' replaced by ``$\ge$'' $($resp. ``$\le$''$)$ for $t>0$ and $\zeta(t) \leq j \leq \xi(t)$, where the functions $\zeta:\R_+\to\R$ and $\xi: \R_+\to\R$ are continuous. Moreover, $\overline v_j(0)\ge \underline v_j(0)$ for $\zeta(0)-1\leq j \leq \xi(0)+1$ and $\overline v_j(t)\ge\underline v_j(t)$ for $t>0$ and $j\in[\zeta(t)-1,\zeta(t))\cup (\xi(t),\xi(t)+1]$. Then, $\overline v_j(t)\ge \underline v_j(t)$ for all $t>0$ and $\zeta(t) \leq j \leq \xi(t)$.
\end{prop}

\end{document}